\documentclass[a4paper,12pt]{amsart}
\usepackage{amssymb,amsfonts,amsthm}
\usepackage{amstext}
\usepackage{comment}
\usepackage{amsmath}
\usepackage{hyperref}  % cria links para teoremas, lemas, ... e as referencias tambem
\usepackage{color}
\usepackage[latin1]{inputenc}
\usepackage[active]{srcltx}
\usepackage{mathrsfs} % \mathscr{B} cria  simbolos parecidos com os de sigma algebra de borel 
\usepackage{comment} % adiciona a possibilidade de usar \begin{comment} \end{comment}
\usepackage{graphicx}
\usepackage{amssymb,amsfonts,amstext, amsmath}
\usepackage[latin1]{inputenc}
\usepackage{tensor}
\usepackage{color}
\usepackage{mathptmx}

\newtheorem{thm}{Theorem}[section]
\newtheorem{cor}[thm]{Corollary}
\newtheorem{lem}[thm]{Lemma}
\newtheorem{prop}[thm]{Proposition}
\newtheorem{ex}[thm]{Example}

\newtheorem{defn}[thm]{Definition}
\newtheorem{rem}[thm]{Remark}

\DeclareMathAlphabet\EuScript{U}{eus}{m}{n}
\SetMathAlphabet\EuScript{bold}{U}{eus}{b}{n}

\newtheorem*{rep@theorem}{\rep@title}
\newcommand{\newreptheorem}[2]{%
	\newenvironment{rep#1}[1]{%
		\def\rep@title{#2 \ref{##1}}%
		\begin{rep@theorem}}%
		{\end{rep@theorem}}}
\makeatother

\theoremstyle{definition}

\usepackage{chngcntr} % number the theorems and definitions on the appendix with the same  a new counting method. Must be used together with \counterwithin{thm}{section}  after \appendix

\usepackage{mathrsfs} % \mathscr{B} cria  simbolos parecidos com os de sigma algebra de borel 

\usepackage{comment} % adiciona a possibilidade de usar \begin{comment} \end{comment}

\usepackage{hyperref}  % cria links para teoremas, lemas, ... e as referencias tambem
\begin{document}
	
	\title{Operator valued positive definite  kernels and differentiable universality}
	
	\author{J. C. Guella}
	\email{jean.guella@riken.jp}
	\address{RIKEN Center for Advanced Intelligence Project, Tokyo, Japan}
	
	\begin{abstract}
		We present a characterization for a positive definite operator valued kernel to be universal or $C_{0}$-universal, and apply these characterizations to a family of operator valued kernels that are shown to be  well behaved. Later, we obtain  a characterization for an operator valued differentiable kernel to be  $C^{q}$-universal and $C^{q}_{0}$-universal. In order to obtain such characterization and examples we generalize some well known results concerning the structure of differentiable kernels to the operator valued context. On the examples is given an emphasis on the radial kernels on Euclidean spaces.   
	\end{abstract}
	

	\maketitle

		%\keywords{ Positive definite kernels \and Universality \and Differentiable universality \and Operator valued kernels \and Radial kernels} 
		%\subclass{  42A82 \and 46E20 \and 46E40 \and 46G10}

	\tableofcontents

	\section{\textbf{Introduction}}

	The concept of a complex valued positive definite kernel has been permeating Mathematics since the beginning of the $20$th century, especially after the seminal work  \cite{Aronszajn1950}, which laid down the connection  between positive definite kernels and Reproducing Kernel Hilbert Spaces (RKHS). In applications (especially in Machine Learning), one of the main desirable properties on a RKHS is if it can approximate a target (but usually unknown) function. In this sense, the concepts of  universality (ability to approximate continuous functions on compact sets) and $C_{0}$-universality (ability to approximate  any $C_{0}$ function) are a basic requirement \cite{CS}, \cite{CZ}. Recently the concept of $C^{q}$-universality (ability to approximate a function and its derivatives up to order $q$ on compact sets) and $C_{0}^{q}$-universality  (ability to approximate  any $C^{q}_{0}$ function and its derivatives up to order $q$) has gained some attraction \cite{derivativekernel}, \cite{Michelli2014} as a natural condition for approximating a target function and its derivatives up to order $q$.
	
	A generalization of the concept of complex valued valued positive definite kernel to the operator valued context also has been attracting attention \cite{Caponnetto2008},\cite{Minh2016}, \cite{zhang2012}. Let $\mathcal{H}$ be a separable Hilbert space and $\mathcal{L}(\mathcal{H})$ be the space of all continuous linear operators from $\mathcal{H}$ to $\mathcal{H}$. An operator valued kernel $K: X \times X \to \mathcal{L}(\mathcal{H})$ is called \textbf{positive definite} (or Multi-task kernel \cite{Caponnetto2008}) if for every finite quantity of distinct points $x_{1}, \ldots, x_{n} \in X$ and vectors $v_{1}, \ldots, v_{n} \in \mathcal{H}$ we have that
	$$
	\sum_{\mu, \nu =1}^{n}\langle K(x_{\mu}, x_{\nu})v_{\mu, v_{\nu}} \rangle_{\mathcal{H}} \geq 0.
	$$
	In addition, if the above double sum is zero only when all vectors $v_{\mu}$ are zero, we say that the kernel is \textbf{strictly positive definite}.

	Once a definition in Mathematics is extended to a broader context, one of the first questions that comes to mind is how the generalization behaves comparing to the objects that were extended. For operator valued positive definite kernels and associated definitions, one way of doing it is by the \textbf{scalar valued  projections of the kernel}, meaning if $K: X \times X \to \mathcal{L}(\mathcal{H}) $ is a positive definite kernel,  its scalar valued projections are the kernels $K_{v}: X \times X \to \mathbb{C} $, $v \in \mathcal{H}\setminus{\{0\}}$, given by
	$$
	K_{v}(x,y):= \langle  K(x,y)v ,v \rangle_{\mathcal{H}} \in \mathbb{C}.   
	$$
	
	It is easy to verify that if the operator valued kernel is positive definite and satisfy any associated definition that we present at Section \ref{Definitions},  then all scalar valued projections of the kernel are positive definite and satisfy the same associated definition.  For the convenience of the reader we prove this affirmation on Lemma \ref{scalarvaluedprojectionsuniversality}. In general the converse is not valid, not even on the positive definite case scenario. In Example \ref{examplenonuniversalityradial} and Example \ref{exampleconverseuniversality}, we prove that universality and differentiable universality on the scalar valued projections of a positive definite radial kernel on an Euclidean space does not imply the same property on the operator valued kernel. However, in Theorem \ref{universaleuclidian2} and Theorem \ref{universaleuclidian2cq}, we prove that  if the operator valued radial kernel is positive definite on all Euclidean spaces (under some technical conditions due to the subtelities of operator valued measures)  information on its scalar valued projections implies a similar property on the operator valued kernel.    
	
	One of the reasons of why the results on Theorem \ref{universaleuclidian2} are so well behaved, is connected to the characterization of the complex valued positive definite radial kernels on all Euclidean spaces \cite{schoenb}, which are related to the Gaussian kernels $e^{-r\|x-y\|^{2}}$, $r\geq 0$ and those kernels are $C^{\infty}_{0}$-universal. 
	
	This example can be put into a very general framework, that is of a bounded continuous function $p: \Omega \times (X\times X ) \to \mathbb{C}$ for which $p_{w}: X \times X \to \mathbb{C}$ is a positive definite kernel for all $w \in \Omega$, a  scalar valued nonnegative finite Radon measure $\lambda$ on $\Omega$ and the positive definite kernel being analyzed satisfies 
	\begin{equation}\label{mainexample}
	P(x,y)= \int_{\Omega}p_{w}(x,y)d\lambda(w).
	\end{equation}
	
	Actually, most well known families of complex valued positive definite kernels can be written in a similar way, for instance:
	\begin{enumerate}
		\item[$\circ$](Isotropic kernels on spheres) $X=S^{d}$, $\Omega = \mathbb{Z}_{+}$,\quad  $p_{n}(x,y)= C^{d-2/2}_{n}(\langle x ,y\rangle)/C^{d-2/2}_{n}(1)$;  
		%	\item[] (Entire kernels) $X = \mathbb{R}^{m}$, $\Omega = \mathbb{Z}_{+}$,\quad  $p_{n}(x,y)= (\langle x ,y\rangle)^{n}$;
		\item[$\circ$] (Bochner  kernels)  $X = \mathbb{R}^{m}$, $\Omega =\mathbb{R}^{m}$,\quad  $p_{\xi}(x,y)= e^{-i(x-y)\cdot\xi}$;
		\item[$\circ$] (Askey kernels) $X = \mathbb{R}^{m}$, $\Omega = [0, \infty)$,\quad  $p_{r}(x,y)= (1 - r\|x-y\|)^{\ell}_{+}$;
		\item[$\circ$] (Radial kernels) $X = \mathbb{R}^{m}$, $\Omega = [0, \infty)$,\quad  $p_{r}(x,y)= \frac{1}{Vol(S^{m-1})}\int_{S^{m-1}}e^{-ir(x-y)\cdot w}dw$;    
	\end{enumerate} 
	
	In Subsection \ref{examplesuniversalkernels} we prove that if all kernels $p_{w}$ satisfy one property (strictly positive definite/ universal/ integrally strictly positive definite) then an operator valued version of \ref{mainexample}, by integrating it with a nonzero and finite operator valued nonnegative Radon measure that admits a Radon-Nikod\'ym decomposition satisfy the same property, and moreover, that this family of kernels is well behaved with respect to the scalar valued projections of the kernel. But in order to present the results on Subsection \ref{examplesuniversalkernels}, we need a method to analyze the universality and $C_{0}$-universality of an operator valued positive definite kernel, which is the focus of Subsection \ref{Characterization of operator valued positive definite universal (and related) kernels}. The method we present is a mixture of Theorem $11$ in \cite{Caponnetto2008} and Proposition $4$ in \cite{Sriperumbudur3} to the operator valued setting, and with this method we define the concept of operator valued integrally strictly positive definite kernel, in a similar way as \cite{Sriperumbudur3}. Additionally, we present a criteria for when the the inclusion-restriction $I : \mathcal{H}_{K} \to C(\mathcal{A}, \mathcal{H})$ is a compact operator for every compact set $\mathcal{A} \subset X$, as well as a version of this result on the $C_{0}$ case.       
	
	On Section \ref{On differentiable positive definite operator valued universal (and related) kernels} we move to differentiable kernels. First, at Subsection \ref{On differentiable positive definite kernels and properties of its RKHS} we prove a formula for the derivatives of  a kernel satisfying  Equation \ref{mainexample} on the operator valued context, as well as providing an operator valued generalization of the results proved in \cite{Michelli2014} and \cite{carmelivitotoigoumanita2010}, regarding the structure of the RKHS for a differentiable kernel. With these technical results at hand, we are able to prove a characterization for the $C^{q}$-universality and $C_{0}^{q}$-universality on the operator valued setting at Subsection \ref{Sufficient conditions for operator valued positive definite differentiable universal (and related) kernels}, generalizing the complex valued results proved in \cite{Simon-Gabriel2018}. We conclude this article at section \ref{Kernels on Euclidean spaces}, where we apply the results on the previous Subsections and obtain several families of operator valued kernels for which the differentiable universalities  are well behaved with respect to the scalar valued projections of the kernel, with an emphasis on radial kernels.                                                                             
	
	For the convenience of the reader, at the Appendix  \ref{Appendix} we discuss the definition and main properties of the several types of vector integration that we use in this article and at Section \ref{Definitions} we review the several concepts of universality that we use  in detail.

	\section{\textbf{Definitions}}\label{Definitions}
	
	In this preliminary section we present several definitions that describes qualitative properties for an operator valued positive definite kernel and its RKHS. If $\mathcal{H}$ is a Hilbert space, a continuous linear operator $T : \mathcal{H} \to \mathcal{H}$ is called 
	\begin{enumerate}
		\item[$\circ$] \textbf{Positive semidefinite} if $\langle Tv,v\rangle_{\mathcal{H}}\geq 0 $, for all $v \in \mathcal{H}$.
		\item[$\circ$]\textbf{Positive definite} if $\langle Tv,v\rangle_{\mathcal{H}}> 0 $, for all $v \in \mathcal{H}\setminus{\{0\}}$.
		\item[$\circ$]\textbf{Strictly positive definite} if $\langle Tv,v\rangle_{\mathcal{H}}\geq M \|v\|_{\mathcal{H}}^{2} $, for some $M \geq 0$ and all $v \in \mathcal{H}$.
	\end{enumerate}
	
	For the construction of the RKHS from a  positive definite kernel $K: X \times X \to \mathcal{L}(\mathcal{H})$ (we are not assuming any topology on $X$), consider the vector space 
	$$
	H_{K}:= span\{ z \in X \to [K_{x}v](z):= K(z,x)v \in \mathcal{H}  \}
	$$
	and the inner product on $H_{K}$  that satisfies $\langle K_{x}v, K_{y}u \rangle_{\mathcal{H}_{K}}:= \langle v, K(x,y)u\rangle_{\mathcal{H}}$. The completion of the normed space $(H_{K}, \|\cdot \|_{\mathcal{H}_{K}})$ is denoted by $\mathcal{H}_{K}$ and can be taken as a subspace of the set of  functions from $X$ to $\mathcal{H}$  that contains the subspace $H_{K}$. The inner product on $\mathcal{H}_{K}$ satisfies
	$$
	\langle K_{x}v, F \rangle_{\mathcal{H}_{K}}= \langle v, F(x) \rangle_{\mathcal{H}}.
	$$
	For detailed arguments and basic properties for RKHS of Hilbert valued functions we refer \cite{Schwartzkernell}. It is important to make this construction explicitly, because several results that we prove depends on how we choose to define the RKHS.

	Recall that for a locally compact space $X$, the Banach space $C_{0}(X, \mathcal{H})$ is defined as the set of continuous functions $f: X \to \mathcal{H}$ (on the norm topology of $\mathcal{H}$), such that for every $\epsilon >0$ there exists a compact set $\mathcal{A}_{\epsilon}$ for which $ \|f(x)\|_{\mathcal{H}}< \epsilon$ for   $x \in X \setminus\mathcal{A}_{\epsilon}$, with norm given by $\sup_{x \in X}\|f(x)\|_{\mathcal{H}}$.

	\begin{defn}\label{univdef} Let $X$ be a Hausdorff space and $K: X \times X \to \mathcal{L}(\mathcal{H}) $ be an operator valued positive definite kernel. We say that the operator $K$ is:\\
		$\circ$  \textbf{Universal}, if $\mathcal{H}_{K} \subset C(X, \mathcal{H})$ and for every compact set $\mathcal{A} \subset X$, every continuous function $g:  \mathcal{A} \to \mathcal{H}$ and every $\epsilon >0$ there exists $f : X \to \mathcal{H} \in \mathcal{H}_{K}$ for which
		$$	
		\sup_{x \in \mathcal{A}} \| f(x) - g(x)  \|_{\mathcal{H}}< \epsilon.	 
		$$
		$\circ$ \textbf{Projectively universal}, if for every $v \in \mathcal{H}\setminus{\{0\}}$, the scalar valued kernel $K_{v}: X \times X \to \mathbb{C}$, given by $K_{v}(x,y):=\langle K(x,y)v,v \rangle_{\mathcal{H}} $ is universal.\\
		In addition, when $X$ is a locally compact space, we say that the operator $K$ is:\\
		$\circ$ \textbf{$C_{0}$-universal}, if $\mathcal{H}_{K}   \subset C_{0}(X, \mathcal{H})$ and for  every continuous function $g \in C_{0}(X, \mathcal{H})$ and every $\epsilon >0$ there exists $f : X \to \mathcal{H} \in \mathcal{H}_{K}$ for which
		$$	
		\sup_{x \in X} \| f(x) - g(x)  \|_{\mathcal{H}}< \epsilon.	 
		$$
		$\circ$ \textbf{Projectively $C_{0}$-universal}, if for every $v \in \mathcal{H}\setminus{\{0\}}$, the scalar valued kernel $K_{v}$ is $C_{0}$-universal. 
	\end{defn}

	On the $C_{0}$ case we always assume that $X$ is locally compact in order to avoid pathological topologies. The definition of universal and $C_{0}$-universal kernels on the operator valued context first appeared at \cite{Caponnetto2008} and  the projectively definitions are a natural step from them. The main interest on the projectively universalities is to understand under which conditions the converse of Lemma \ref{scalarvaluedprojectionsuniversality} holds, in other words, when the fact that kernel is Projectively universal ($C_{0}$-universal) implies that the operator valued kernel is universal ($C_{0}$-universal). This type of analysis has been proving a fruitful and intriguing relationship, as can be seen in \cite{guella3}, \cite{Guella2019},  \cite{neeb1998}.

	Although the definition for a positive definite kernel being universal (or $C_{0}$-universal) is simple, usually it is difficult to obtain an explicit description for the  RKHS of a kernel.  At Theorem \ref{oneintegraltodoubleintegral}  we extend Theorem $11$ of \cite{Caponnetto2008} and present a powerful  if and only if characterization for the universality ($C_{0}$-universality)  of an operator valued kernel, which is one of the main building blocks for the results on this paper.

	Sometimes the inclusion $\mathcal{H}_{K} \subset C_{0}(X, \mathcal{H})$ might also be difficult to verify but the technical condition in Theorem \ref{oneintegraltodoubleintegral} that characterizes $C_{0}$-universal kernels might be much simpler to analyze. 
	
	\begin{defn}\label{intstricdefinition} Let $X$ be a locally compact Hausdorff space, we say that a bounded positive definite kernel  $K: X \times X \to \mathcal{L}(\mathcal{H}) $ for which $\mathcal{H}_{K}\subset C(X, \mathbb{C}^{\ell})$ is \textbf{integrally strictly positive definite} if for every nonzero  $\mathcal{H}$ valued Radon measure of bounded variation $ \eta$ in $X$ ($\eta \in \mathfrak{M}(X, \mathcal{H})\setminus{\{0\}}$)
		$$
		\int_{X}\langle \int_{X} K(x,y) d\eta(x), d\eta(y)\rangle >0. 
		$$	
	\end{defn}
	
	When $\mathcal{H}=\mathbb{C}$ this definition is the one given in \cite{Sriperumbudur}. For some specific type of complex valued kernels, a good description of those who are integrally strictly positive definite were obtained in  \cite{cheney1995}, \cite{Sriperumbudur2}, \cite{Sriperumbudur}, especially  the kernels on Euclidean spaces invariant by translations (more generally on a locally compact commutative group). 
	
	If a kernel $K: X \times X \to \mathcal{L}(\mathcal{H}) $ is integrally strictly positive definite, by standard arguments of measure theory, it is possible to obtain that this property is equivalent at $\mathcal{H}_{K}$ being dense on every Banach space $L^{1}(X, \lambda, \mathbb{C}^{\ell})$, where $\lambda$ is a finite scalar valued nonnegative finite Radon measure  on $X$ and 
	$$
	L^{1}(X, \lambda, \mathbb{C}^{\ell}):=\{h: X \to \mathbb{C}^{\ell}, \quad  \int_{X}\|h(x)\|d\lambda(x) <\infty \}.$$
	
	%Although it is possible to define the concept of $L^{p}$-universality from this relation, \cite{Sriperumbudur3}, we refrain from doing so.  

	On Euclidean spaces the concept of universality can be generalized.  If $U \subset \mathbb{R}^{m}$ is an open set, a function $F : U  \to \mathcal{H} $ is an element  of  $C^{1}(U , \mathcal{H})$ if for every $1\leq i \leq m$, there exists a  continuous function $ U  \to \mathcal{H}  $, which can be proved that is unique and we denote it by $\partial^{e_{i}} F $, such that
	$$
	\lim_{h \to 0} \left \| \frac{F(x+he_{i}) -F(x)}{h} - \partial^{e_{i}} F(x) \right \|_{ \mathcal{H}}=0, \quad  x \in U.
	$$
	Recursively, we say that $F \in C^{2}( U, \mathcal{H})$ if for every $1\leq i \leq m$ the function $\partial^{i} F$ is in $C^{1}(U , \mathcal{H})$. Similar to multi variable calculus, it can be proved that $\partial^{e_{i}} \partial^{e_{j}}F = \partial^{e_{j}} \partial^{e_{i}}F $ for every $F \in C^{2}(U, \mathcal{H})$. For more information on vector valued differentiable functions we refer \cite{vecdiff}
	
	We also define the set  
	$$ 
	C^{q}_{0}(U, \mathcal{H}):= \{F \in C^{q}(U, \mathcal{H}), \quad  \partial^{\alpha}F \in C_{0}(U, \mathcal{H}), \quad  |\alpha| \leq q\}.
	$$
	
	\begin{defn}\label{Cinftyuniversality}
		Let  $U \subset \mathbb{R}^{m}$ be an open set and  $K: U \times U \to \mathcal{L}(\mathcal{H})$ be a positive definite kernel for which  $\mathcal{H}_{K} \subset C^{q}(U, \mathcal{H})$. The kernel $K$ is called \\
		$\circ$ \textbf{$C^{q}$-universal } if for every compact set $\mathcal{A} \subset U$  every function $g  \in C^{q}(U, \mathcal{H})$ and every $\epsilon >0$ there exists $f : U \to \mathcal{H} \in \mathcal{H}_{K}$ for which
		$$	
		\sup_{x \in \mathcal{A}} \sum_{|\alpha |\leq q}\| \partial^{\alpha}f(x) - \partial^{\alpha}g(x)\|_{\mathcal{H}}< \epsilon.	 
		$$
		$\circ$ \textbf{Projectively $C^{q}$-universal} if for every $v \in \mathcal{H}\setminus{\{0\}}$, the scalar valued kernel $K_{v}: U \times U \to \mathbb{C}$, given by $K_{v}(x,y):=\langle K(x,y)v,v \rangle_{\mathcal{H}} $ is $C^{q}$-universal.\\
		$\circ$ \textbf{$C_{0}^{q}$-universal } if  $\mathcal{H}_{K} \subset C_{0}(U, \mathcal{H})$ and for every function $g  \in C_{0}^{q}(U, \mathcal{H})$ and every $\epsilon >0$ there exists $f : U \to \mathcal{H} \in \mathcal{H}_{K}$ for which
		$$	
		\sup_{x \in U} \sum_{|\alpha |\leq q}\| \partial^{\alpha}f(x) - \partial^{\alpha}g(x)\|_{\mathcal{H}}< \epsilon.	 
		$$
		$\circ$ \textbf{Projectively $C^{q}_{0}$-universal}, if for every $v \in \mathcal{H}\setminus{\{0\}}$, the scalar valued kernel $K_{v}$ is $C^{q}_{0}$-universal. \\
		If the kernel $K$ is $C^{q}$-universal for every $q \in \mathbb{N}$, we say that the kernel is \textbf{$C^{\infty}$-universal}. Similarly, If the kernel $K$ is $C_{0}^{q}$-universal for every $q\in\mathbb{N}$, we say that the kernel is \textbf{$C_{0}^{\infty}$-universal}. 
	\end{defn}
	
	On the scalar valued case, the concept of $C^{q}$-universal kernel first appeared at \cite{kilmer1996} (with the terminology ``Fundamental set on $C^{q}$'') while the $C_{0}^{q}$-universal kernels appeared in \cite{Simon-Gabriel2018}.

	The notation $K \in C^{q,q}(U \times U, \mathcal{L}(\mathcal{H}) )$ means  that $K$ is jointly differentiable up to order $q$ on each coordinate and the derivatives are continuous functions from $U \times U$ to $\mathcal{L}(\mathcal{H})$. For instance, the kernel
	$$
	(x,y) \in \mathbb{R}^{m} \times \mathbb{R}^{m} \to k(x,y):= \frac{\langle x,y\rangle^{q}}{\|x\|^{2} + \|y\|^{2}} \in  \mathbb{R}
	$$ 
	is positive definite, however $k \in C^{q-1}(\mathbb{R}^{m} \times \mathbb{R}^{m})\setminus{C^{q}(\mathbb{R}^{m} \times \mathbb{R}^{m})}$.

	\section{\textbf{Operator valued universal (and related) kernels}}\label{Dual spaces and equivalent conditions for universality}
	In this section, we first present  a characterization for an operator valued positive definite kernel to be universal or $C_{0}$-universal. Later, we apply this characterization to a family of operator valued kernels, which are shown to be  well behaved, especially with respect to the scalar valued projections of the operator valued kernel.

	\subsection{\textbf{Characterization}}\label{Characterization of operator valued positive definite universal (and related) kernels}
	
	On a Banach space $\mathfrak{B}$, a subset $B$ is such that $span(B)$ is dense on $\mathfrak{B}$ if and only if  the only continuous linear functional $v \in  \mathfrak{B}^{*}$ such that $(v, b)_{(\mathfrak{B}^{*}, \mathfrak{B})} =0$ for all $b \in \mathfrak{B}$ is the zero functional. 
	
	Our proof for the characterization of universal and $C_{0}$-universal operator valued kernels relies on this simple relation from functional analysis. Being so, we make a few remarks over the dual spaces of $C(\mathcal{A}, \mathcal{H})$ ($\mathcal{A}$ is a compact Hausdorff space) and $C_{0}(X, \mathcal{H})$. Both   $C(\mathcal{A}, \mathcal{H})$ and $C_{0}(X, \mathcal{H})$ are Banach spaces on the sup norm 
	$$
	\|F\|_{C(\mathcal{A}, \mathcal{H})} := \sup_{x \in \mathcal{A}} \|F(x)\|_{\mathcal{H}}, \quad \|F\|_{C_{0}(X, \mathcal{H})} := \sup_{x \in X} \|F(x)\|_{\mathcal{H}}.
	$$

	%Naturally, we define $C(\mathcal{A}, \mathcal{H})$ as the space of continuous functions from a compact Hausdorff set  $\mathcal{A}$  to $\mathcal{H}$ (norm topology) and  $ C_{0}(X, \mathcal{H})$ as the space of continuous functions from a locally compact Hausdorff space $X$ to $\mathcal{H}$ (norm topology)  that tends to zero at the infinity point of $X$ (precisely, $F \in C_{0}(X, \mathcal{H})$ if $F \in C(X, \mathcal{H})$  and for every $\epsilon >0$ there exists a compact set $\mathcal{A}_{\epsilon} \subset X $ for which  $\|F(x)\|_{\mathcal{H}} < \epsilon$ for all $x \in X - \mathcal{A}_{\epsilon}$). 

	We recall some measure theoretical definitions that will be needed.
	
	On a Hausdorff space $Z$, the sigma algebra generated by the open sets in $Z$, which we denote by $\mathscr{B}(Z)$ is called the Borel  sigma algebra on $Z$. A finite vector valued measure $ \Lambda : \mathscr{B}(Z) \to \mathfrak{B}$ of bounded variation is called a \textbf{Radon measure} on $Z$ if:
	\begin{enumerate}
		\item [(i)] $|\Lambda|$ is outer regular on all Borel sets ($
		|\Lambda|(E) = \inf \{ |\Lambda|(U), \quad E \subset U, \quad U \text{ is open} \}
		$) 
		\item[(ii)] $|\Lambda|$ is inner regular on all open sets($|\Lambda|(U) = \sup \{ |\Lambda|(\mathcal{A}), \quad \mathcal{A} \subset U \quad \mathcal{A} \text{ is compact} \}
		$)
	\end{enumerate}  
	
	The set of all $\mathcal{H}$-valued Radon measures of bounded variation on $Z$ is denoted by $\mathfrak{M}(Z, \mathcal{H})$, and this set naturally posses a structure of a Banach space when $Z$ is compact or locally compact by defining the norm
	$$
	\| \Lambda_{1} - \Lambda_{2} \|_{\mathfrak{M}(Z, \mathcal{H})}:= | \Lambda_{1} - \Lambda_{2} |(Z).
	$$ 
	
	Standard topological arguments shows that every function $F \in C(\mathcal{A}, \mathcal{H})$(or $C_{0}(X, \mathcal{H})$) is Bochner measurable and integrable with respect to any finite Radon measure of bounded variation, more precisely, it can be approximated by the simple functions on $$span\{ w \in X \to F(x)\chi_{A}(w) \in \mathcal{H}, \quad  x \in X, A \in \mathscr{B}(X) \}.$$ The integrability is a consequence of Equation \ref{bochnerintegration}. 
	
	If $\eta \in \mathfrak{M}(\mathcal{A}, \mathcal{H})$, the linear functional $F \in C(\mathcal{A}, \mathcal{H}) \to \int_{\mathcal{A}} \langle F(x), d\eta(x)\rangle \in \mathbb{C}$ is continuous, indeed
	\begin{align*}
	|\int_{\mathcal{A}} \langle F(x), d\eta(x)\rangle| &\leq  \int_{\mathcal{A}} \|F(x)\|_{\mathcal{H}}d|\eta|(x)\leq \|F\|_{C(\mathcal{A}, \mathcal{H})} \int_{\mathcal{A}} d|\eta|(x)\\ &= \|F\|_{C(\mathcal{A}, \mathcal{H})} \|\eta\|_{\mathfrak{M}(\mathcal{A}, \mathcal{H})}.
	\end{align*}
	
	A similar argument holds for $C_{0}(X, \mathcal{H})$. The following result, which is a vector valued generalization of the Riesz-Representation Theorem, states that those are the only continuous linear functionals on $C(\mathcal{A}, \mathcal{H}) $ or $C_{0}(X, \mathcal{H})$.
	\begin{thm}\textbf{(Dinculeanu-Singer)} \label{Dinculeanu-Singer} Let $Z$ be a compact Hausdorff space or a locally compact Hausdorff space,  $L : C(Z, \mathcal{H}) \to \mathbb{C} $ ($L : C_{0}(Z, \mathcal{H}) \to \mathbb{C}$) be a continuous linear functional. Then there exists an unique Radon measure of bounded variation $\eta \in \mathfrak{M}(Z, \mathcal{H})$  for which
		$$
		L(F)= \int_{Z} \langle F(x), d\eta(x)\rangle. 
		$$	
	\end{thm} 
	
	The \textbf{support of a measure} $\eta \in \mathfrak{M}(Z, \mathcal{H})$ is defined as the support of the measure $|\eta|$, so $[supp (\eta)]^{c}$
	is  the union of all open sets $U \subset X$, for which $\eta(E)=0$, for all $E \subset U$.
	
	In \cite{carmelivitotoigoumanita2010}, it was presented the following criteria for $\mathcal{H}_{K}$ to be a subset of $C(X,\mathcal{H})$ and  $C_{0}(X,\mathcal{H})$: 
	\begin{prop}\label{rkhscontained}Let $K: X \times X \to \mathcal{L}(\mathcal{H})$ be a positive definite kernel and $\mathcal{H}_{K}$ its RKHS. Then:
		\begin{enumerate}
			\item[(i)] $\mathcal{H}_{K} \subset C(X, \mathcal{H})$ if and only if the function $x \in X \to \|K(x,x)\|_{\mathcal{L}(\mathcal{H})} \in \mathbb{C}$ is locally bounded and the function $y \in X \to K(x,y)v \in C(X, \mathcal{H})$, for all $x \in X$ and $v \in \mathcal{H}$.
			\item[(ii)] $\mathcal{H}_{K} \subset C_{0}(X, \mathcal{H})$ if and only if the function $x \in X \to \|K(x,x)\|_{\mathcal{L}(\mathcal{H})} \in \mathbb{C}$ is bounded and the function $y \in X \to K(x,y)v \in C_{0}(X, \mathcal{H})$, for all $x \in X$ and $v \in \mathcal{H}$.\\
			Moreover, if $\mathcal{H}_{K}\subset C(X, \mathcal{H})$ then the inclusion-restriction $I :\mathcal{H}_{K} \to  C(\mathcal{A}, \mathcal{H})$ is continuous for every compact set $\mathcal{A}\subset X$. Similarly,  if $\mathcal{H}_{K} \subset C_{0}(X, \mathcal{H})$, then the inclusion  $I :\mathcal{H}_{K} \to  C_{0}(X, \mathcal{H})$ is continuous.
		\end{enumerate}	
	\end{prop}
	
	We emphasize that a kernel does not need to be jointly continuous in order that $\mathcal{H}_{K} \subset C(X, \mathcal{H})$. For instance, the kernel $ (x,y) \in \mathbb{R}^{m} \times \mathbb{R}^{m} \to \frac{\langle x,y\rangle}{\|x\|^{2} + \|y\|^{2}} \in  \mathbb{R}$ is positive definite because $k(x,y)= \int_{[0, \infty)}\langle x,y\rangle e^{-r\|x\|^{2} - r\|y\|^{2}}dr$ and $\mathcal{H}_{k} \subset C_{0}(\mathbb{R}^{m})$ by Lemma \ref{rkhscontained}, however the kernel $k$ is not a jointly continuous function.

	At Theorem $11$ of \cite{Caponnetto2008} it is proved that if $X$ is a Hausdorff space and $\mathcal{H_{K}}\subset C(X,\mathcal{H})$, then $K$ is universal if and only if for every compact set $\mathcal{A}\subset X$, the only measure  $\eta \in \mathfrak{M}(\mathcal{A}, \mathcal{H})$ for which
	\begin{equation}\label{univform1}
	\int_{\mathcal{A}}\langle K(x,y)v,d\eta(x)\rangle=0 
	\end{equation} 
	is the zero measure. A similar result is possible for the $C_{0}$-universality. Indeed, by the comments made at the beginning of this subsection, if $\mathcal{H}_{K} \subset C_{0}(X, \mathcal{H})$, then $K$ is $C_{0}$-universal if and only if  the only measure $\eta \in \mathfrak{M}(X, \mathcal{H})$ for which
	\begin{equation}\label{c0univform1}
	\int_{X}\langle F(x),d\eta(x)\rangle=0 
	\end{equation}
	for all $F \in \mathcal{H}_{K}$ is the zero measure. Note that if Equation \ref{c0univform1} holds true for all $F \in \mathcal{H}_{K}$, then it holds for $F(x)= K(x,y)v$, for all $y \in X$ and $v \in \mathcal{H}$, and since the inclusion $I:\mathcal{H}_{K}\to  C_{0}(X, \mathcal{H})$ is continuous by Proposition \ref{rkhscontained},    the converse is also true, so if $\mathcal{H}_{K} \subset C_{0}(X, \mathcal{H})$, then $K$ is $C_{0}$-universal if and only if  the only measure $\eta \in \mathfrak{M}(X, \mathcal{H})$ for which
	\begin{equation}\label{c0univform2}
	\int_{X}\langle K(x,y)v,d\eta(x)\rangle=0 
	\end{equation}
	for all $y \in X$ and $v \in \mathcal{H}$ is the zero measure. In the following, based on these two results we prove another characterization for universality and $C_{0}$-universality that is more technically advantageous for our purposes.
	
	\begin{thm}\label{oneintegraltodoubleintegral} Let  $K: Z\times Z \to \mathcal{L}(\mathcal{H})$ be a positive semidefinite kernel for which $\mathcal{H}_{K} \subset C_{0}(Z, \mathcal{H})$ ($\mathcal{H}_{K} \subset C(Z, \mathcal{H})$), then:
		\begin{enumerate}
			\item [(i)] For every $\eta \in \mathfrak{M}(Z, \mathcal{H})$ ($\eta \in \mathfrak{M}(Z, \mathcal{H})$ of compact support) the function $K_{\eta}: Z \to \mathcal{H}$, defined as 
			$$
			\langle v, K_{\eta}(y) \rangle_{\mathcal{H}} = \int_{Z}\langle K(x,y)v,d\eta(x)\rangle
			$$
			is an element of $\mathcal{H}_{K}$ and it is also described  by the weak-Bochner integral $K_{\eta}(y)= \int_{Z}K(x,y)d\eta(x)$.
			\item [(ii)] The following equality holds
			$$\langle K_{\eta}, K_{\eta}\rangle_{\mathcal{H}_{K}}= \int_{Z} \langle \int_{Z} K(x,y) d\eta(x), d\eta(y) \rangle 
			$$
			and if $d\eta= Hd|\eta|$ is a Radon-Nikod\'ym decomposition for $\eta$, then
			$$\langle K_{\eta}, K_{\eta}\rangle_{\mathcal{H}_{K}}= \int_{Z}\int_{Z}\langle  K(x,y)H(y), H(x)\rangle_{\mathcal{H}} d|\eta|(x), d|\eta|(y). 
			$$
			\item[(iii)] The kernel $K$ is $C_{0}$-universal (universal) if and only if $\langle K_{\eta}, K_{\eta}\rangle_{\mathcal{H}_{K}}>0$ for all  $\eta \neq 0$.  	
		\end{enumerate}	
	\end{thm}
	When $\mathcal{H}= \mathbb{C}$ this result can be found at \cite{Sriperumbudur3}. We remark that when $\mathcal{H}= \mathbb{C}$, by the way we defined the vector integrals we get
	$$\langle K_{\eta}, K_{\eta}\rangle_{\mathcal{H}_{K}}= \int_{Z}\int_{Z} K(x,y) d\overline{\eta}(x) d\eta(y) 
	$$
	while on \cite{Sriperumbudur3}, the inner product is equal to  $\int_{Z}\int_{Z} K(x,y) d\overline{\eta}(y) d\eta(x)$. The  criteria for universality are equivalent, the only  difference is that the association of $\eta$ to $K_{\eta}$ in \cite{Sriperumbudur3} occurs (on our notation) from $\eta$ to $K_{\overline{\eta}}$.

	In some applications, the inclusion-restriction $I: \mathcal{H}_{K} \to C(\mathcal{A},\mathcal{H} )$ being compact for every compact set $\mathcal{A} \subset X$  is another desirable property, because under this hypothesis and the properties of RKHS every continuous (not necessarily linear) function  $T : C(\mathcal{A},\mathcal{H} ) \to \mathbb{R} $ admits a minimizer (on $\mathcal{H}_{K}$) for every closed and bounded subset of $\mathcal{H}_{K}$. Below we present a criteria for when this property occurs based on the Arzel\`{a}-Ascoli Theorem,  which states that given a compact  Hausdorff set $\mathcal{A}$,  a closed subset $B$ of $C(\mathcal{A}, \mathcal{H})$ is compact if and only if 
	\begin{enumerate}
		\item[$\circ$] $B$ is \textbf{equicontinuous}: For every $\epsilon >0$ and $x \in X$ there is open set  $U_{x}$ that contains $x$  for which $\|F(x) - F(y)\|_{\mathcal{H}}< \epsilon $ for all $F \in B$ and $y \in U_{x}$.
		\item[$\circ$]  $B$ is \textbf{pointwise relatively compact}: The set $\{F(y), F \in B   \} \subset \mathcal{H}$ has compact closure on the norm topology of $\mathcal{H}$.
	\end{enumerate}
	
	\begin{lem}\label{compacinclusion} Let $Z$ be a locally compact space and $K: Z\times Z \to \mathcal{L}(\mathcal{H})$ be a positive definite kernel. Suppose that $\mathcal{H}_{K} \subset C(Z, \mathcal{H})$, then 
		\begin{enumerate}
			\item[(i)] Every bounded set $B \subset \mathcal{H}_{K}$ is equicontinuous if and only if the kernel $K : Z \times Z \to \mathcal{L}(\mathcal{H})$ is  continuous on the norm topology of $\mathcal{L}(\mathcal{H})$. 
			\item[(ii)] If $K(x,x)$ is a trace class operator for every $x \in Z$, then every bounded set $B \subset \mathcal{H}_{K}$ is  pointwise relatively compact. 
		\end{enumerate}
		In particular, by mixing those two results we obtain that if $K(x,x)$ is a trace class operator for every $x \in Z$, then
		\begin{enumerate}
			\item[(iii)] The  continuous inclusion-restriction $I: \mathcal{H}_{K} \to C(\mathcal{A}, \mathcal{H})$ is compact  for  every compact set $\mathcal{A} \subset Z$ if and only if the kernel $K : Z \times Z \to \mathcal{L}(\mathcal{H})$ is  continuous on the norm topology of $\mathcal{L}(\mathcal{H})$. Under this hypothesis, every closed and bounded set of $\mathcal{H}_{K}$ (restricted to $\mathcal{A}$) is a compact set of $C(\mathcal{A}, \mathcal{H})$.
			\item[(iv)] If $\mathcal{H}_{K} \subset C_{0}(Z, \mathcal{H})$, then the set
			$$
			B_{\mathcal{A}}: =B \cap \overline{span\{K_{x}v, \quad , x \in \mathcal{A}, v \in \mathcal{H} \}  }_{\mathcal{H}_{K}}
			$$
			has compact closure on $C_{0}(Z, \mathcal{H})$ for every bounded set $B$ of $\mathcal{H}_{K}$ and compact set $\mathcal{A} \subset Z$ if and only if the kernel $K : Z \times Z \to \mathcal{L}(\mathcal{H})$ is continuous on the norm topology of $\mathcal{L}(\mathcal{H})$. Under this hypothesis, if $B$ is  closed  then $B_{\mathcal{A}}$ is a compact set of $C_{0}(Z, \mathcal{H})$.
		\end{enumerate}
	\end{lem}

	\subsection{\textbf{A well behaved family of operator valued kernels}}\label{examplesuniversalkernels}
	The following $2$ examples shows that universality on matrix valued kernels in Euclidean spaces does not satisfy the converse of the scalar valued projections property not even for radial kernels.  
	
	\begin{ex} \label{examplenonuniversalityradial}For every $m \in \mathbb{N}$ and $ \ell \geq 2$, there exists a continuous function $F: [0, \infty) \to M_{\ell}(\mathbb{C})$ for which the kernel
		$$
		(x,y) \in 	\mathbb{R}^{m} \times \mathbb{R}^{m} \to F(\|x-y\|) \in   M_{\ell}(\mathbb{C}),
		$$
		is positive definite and  for every $v \in \mathbb{C}^{\ell}\setminus{\{0\}}$ the scalar valued positive definite kernel 
		$$
		(x,y) \in 	\mathbb{R}^{m} \times \mathbb{R}^{m} \to \langle   F(\|x-y\|)v,v \rangle  \mathbb{C}
		$$
		is $C_{0}^{\infty}$-universal, but the matrix valued kernel is not universal.
	\end{ex}
	
	Another example is the Gaussian type kernel
	
	\begin{ex}\label{exampleconverseuniversality} The matrix valued kernel
		$$K(x,y)=\begin{bmatrix}
		e^{-\|x-y\|^{2}} & e^{-\| x+2w-y\|^{2}}\\
		e^{-\|x-y-2w\|^{2}} & e^{-\| x-y\|^{2}}\\	
		\end{bmatrix}, \quad w \neq 0$$
		is  positive definite but it is not strictly positive definite (hence, it is not universal) and all of scalar valued projections of this kernel are $C_{0}^{\infty}$-universal kernels.
	\end{ex}
	
	However, it is proved in \cite{neeb1998} an operator valued version of Bochner's Theorem, more precisely, if  $F: \mathbb{R}^{m} \to \mathcal{L}(\mathcal{H})$ is an ultraweakly continuous function, then  the  operator valued kernel $$
	(x,y) \in \mathbb{R}^{m} \times \mathbb{R}^{m} \to F(x-y) \in \mathcal{L}(\mathcal{H})
	$$
	is positive definite if and only if for every $v \in \mathcal{H}$ the scalar valued kernel 
	$$
	(x,y) \in \mathbb{R}^{m} \times \mathbb{R}^{m} \to F_{v}(x-y):=\langle F(x-y)v,v \rangle_{\mathcal{H}}\in \mathbb{C}
	$$
	is positive definite.

	In \cite{Guella2019}, it is proved that if $F: [0, \infty) \to M_{\ell}(\mathbb{C})$ is a continuous function, then the matrix valued kernel ($m\geq 2$)$$
	(x,y) \in \mathbb{R}^{m} \times \mathbb{R}^{m} \to F(\|x-y\|) \in M_{\ell}(\mathbb{C})
	$$
	is strictly positive definite if and only if for every $v \in \mathbb{C}^{\ell}$ the scalar valued kernel 
	$$
	(x,y) \in \mathbb{R}^{m} \times \mathbb{R}^{m} \to F_{v}(x-y):=\langle F(x-y)v,v \rangle\in \mathbb{C}
	$$
	is strictly positive definite (by using a Hamel basis argument and this result, it is possible to obtain the same type of equivalence for operator valued ultraweakly radial kernels).

	In some sense, the results from \cite{neeb1998}, \cite{Guella2019}  and Examples \ref{examplenonuniversalityradial}, \ref{exampleconverseuniversality}  are  related to how we can describe a scalar valued kernel of this type.

	For instance, if $f: \mathbb{R}^{m} \to \mathbb{C}$ is a continuous function that defines a positive definite kernel on Bochner's sense, then
	$$
	f(x-y)=\int_{\mathbb{R}^{m}}e^{-i(x-y)\cdot\xi}d\lambda(\xi), \quad x,y \in \mathbb{R}^{m} 
	$$
	for some scalar valued nonnegative finite Radon measure $\lambda$ on $\mathbb{R}^{m}$. Note that the kernels $e^{-i(x-y)\cdot\xi}$ are positive definite but none of them is strictly positive definite. 
	
	As for radial kernels on a fixed Euclidean space, if $f: [0, \infty) \to \mathbb{C}$ is a continuous function that defines a positive definite radial kernel on $\mathbb{R}^{m}$  then
	$$
	f(\|x-y\|)=\int_{[0,\infty)}\Omega_{m}(r\|x-y\|)d\lambda(r), \quad x,y \in \mathbb{R}^{m}
	$$
	for some scalar valued nonnegative finite Radon measure $\lambda$ on $[0,\infty)$ \cite{wend}, where 
	$$
	\Omega_{m}(\|x\|)= \frac{1}{Vol(S^{m-1})}\int_{S^{m-1}}e^{-ix\cdot\xi}d\xi, \quad x \in \mathbb{R}^{m}.
	$$
	Note that the kernels $\Omega_{m}(r\|x-y\|)$ are strictly positive definite for all $r>0$ (when $m\geq 2$) by \cite{zucastell}, but they are not universal (we prove the last affirmation at Lemma \ref{radialfini}). 
	
	On the other hand, for radial kernels on  all Euclidean spaces, if $f: [0, \infty) \to \mathbb{C}$ is a continuous function that defines a positive definite radial kernel on all Euclidean spaces  by \cite{schoenb} we have that
	\begin{equation}\label{pdonallEuclid}
	f(\|x-y\|)=\int_{[0,\infty)}e^{-r\|x-y\|^{2}}d\lambda(r), \quad x,y \in \mathbb{R}^{m}
	\end{equation}
	for some scalar valued nonnegative finite Radon measure $\lambda$ on $[0,\infty)$. The Gaussian kernels $e^{-r\|x-y\|^{2}}$ are $C^{\infty}_{0}$-universal for all $r>0$ by \cite{Simon-Gabriel2018}.

	At Theorem \ref{universalityintegralkernel} we present a general setting  where several properties over the scalar valued projections of the kernel implies that the operator valued kernel also satisfies this property.    
	
	\begin{thm}\label{universalityintegralkernel} Let $X$ and $\Omega$ be Hausdorff spaces, with $\Omega$ being locally compact and  $p : \Omega \times X \times X \to \mathbb{C}$ be a bounded continuous  function such that the kernel $$
		(x,y) \in X\times X \to p_{\omega}(x,y):=p(\omega, x,y)
		$$
		is  positive definite for every $w \in \Omega$. Given a Radon nonnegative finite operator $\mathcal{L}(\mathcal{H})$ valued measure $\Lambda: \mathscr{B}(\Omega) \to \mathcal{L}(\mathcal{H})$  such that $\Lambda$ admits a  Radon-Nikod\'ym decomposition as $d\Lambda= Gd\lambda$,  consider the operator valued kernel
		$$	
		P : X\times X \to \mathcal{L}(\mathcal{H}), \quad P(x,y)= \int_{\Omega}p_{w}(x,y) d\Lambda(w),	
		$$
		and for every $v \in \mathcal{H}\setminus{\{0\}}$ the scalar valued kernels  and the scalar valued nonnegative finite Radon measures
		$$
		P_{v} : X\times X \to \mathbb{C}, \quad P_{v}(x,y)= \langle P(x,y)v, v \rangle_{\mathcal{H}} 
		$$ 
		$$
		\Lambda_{v}: \mathscr{B}(\Omega) \to \mathbb{C}, \quad    
		\Lambda_{v}(A)= \langle \Lambda(A)v, v \rangle_{\mathcal{H}}.  	
		$$
		Then:
		\begin{enumerate}
			\item[(i)] The operator valued kernel  $P$ and the scalar valued kernels $P_{v}$ are well defined and positive definite.
			\item[(ii)] The function $P : X\times X \to \mathcal{L}(\mathcal{H})$ is bounded and continuous on the norm topology of $\mathcal{L}(\mathcal{H})$. In particular, $\mathcal{H}_{P} \subset C(X, \mathcal{H})$.
			\item[(iii)] If the kernel  $p_{\omega}$ is strictly positive definite for every $w \in \Omega$, then the operator valued kernel $P$ is strictly positive definite if and only if  all scalar valued projections $P_{v}$ are strictly positive definite.
			\item[(iv)] If the kernel  $p_{\omega}$ is universal for every $w \in \Omega$, then the operator valued kernel $P$ is universal if and only if  all the scalar valued kernels $P_{v}$ are universal. 
			\item[(v)] If $X$ is a locally compact space and the kernel  $p_{\omega}$ is integrally strictly positive definite for every $w \in \Omega$, then the operator valued kernel $P$ is integrally strictly positive definite if and only if  all the scalar valued kernels $P_{v}$ are integrally strictly positive definite. 
		\end{enumerate}
		Additionally, the equivalences in $(iii)$, $(iv)$ and $(v)$  are also equivalent at $\Lambda(\Omega)$ being a positive definite operator.
	\end{thm}
	
	As a direct consequence of Theorem \ref{universalityintegralkernel}  to a specific context on Euclidean spaces we have that:
	
	\begin{cor}\label{compcsuppor} Let  $f:\mathbb{R}^{m} \to \mathbb{C}$ be a continuous function for which the kernel $$
		(x,y) \in \mathbb{R}^{m}\times \mathbb{R}^{m} \to k(x,y):= f(x-y) \in \mathbb{C}
		$$ 
		is positive definite. If $\Lambda $ is a finite Radon nonnegative $\mathcal{L}(\mathcal{H})$ valued measure defined on $(0, \infty)$ that admits a Radon-Nikod\'ym decomposition $d\Lambda= Gd\lambda$, the operator valued kernel 
		$$
		(x, y) \in \mathbb{R}^{m} \times \mathbb{R}^{m} \to P(x,y):=\int_{(0, \infty)}f(w(x -y))d\Lambda(r) \in   \mathcal{L}(\mathcal{H}),
		$$
		is  positive definite, $\mathcal{H}_{P} \subset C(\mathbb{R}^{m}, \mathcal{H})$ and the following statements are true
		\begin{enumerate}
			\item If the kernel $k$ is strictly positive definite, the operator valued kernel $P$ is strictly positive definite if and only if $\Lambda((0, \infty))$ is positive definite.
			\item If the kernel $k$ is universal, the operator valued kernel $P$ is universal if and only if $\Lambda((0, \infty))$ is positive definite.
			\item If the kernel $k$ is integrally strictly positive definite, the operator valued kernel $P$ is integrally strictly positive definite if and only if $\Lambda((0, \infty))$ is positive definite.
		\end{enumerate}
	\end{cor}
	
	%As a special case of Proposition \ref{compcsuppor}, we obtain that all operator valued radial kernels that are positive on all Euclidean spaces (under a trace class condition over the function at the point $0$) or that are connected to the Askey class of positive definite radial kernels, the behavior of the scalar valued projections of the kernel implies a similar behavior to the operator valued kernel. First, we need a few comments and definitions.

	From this Corollary we  obtain a characterization for several families of  trace class valued positive definite radial kernels on Euclidean spaces.
	
	\begin{thm}\label{universaleuclidian2}Let   $F: [0, \infty) \to \mathcal{L}(\mathcal{H})$ be an ultraweakly continuous function for which   $F(0) \in \mathcal{L}(\mathcal{H})$ is  a trace class operator and the kernel $K_{F,v}: \mathbb{R}^{m}\times \mathbb{R}^{m}\to \mathbb{C}$  given by	$$ K_{F,v}(x,y):= \langle F(\|x-y\|)v, v\rangle_{\mathcal{H}}$$
		is positive definite for every $m \in \mathbb{N}$  and $v \in \mathcal{H}$. Then the kernel 
		$$
		K_{F} : \mathbb{R}^{m} \times \mathbb{R}^{m} \to F(\|x-y\|) \in \mathcal{L}(\mathcal{H})
		$$
		is positive definite for every $m \in \mathbb{N}$, $\mathcal{H}_{K_F} \subset C(\mathbb{R}^{m}, \mathcal{H})$ and the following are equivalent
		\begin{enumerate}
			\item[(i)] The kernel $K$ is strictly positive definite. 
			\item[(ii)] The kernel $K$ is universal.
			\item[(iii)] For every $v \in \mathcal{H}\setminus\{{0}\}$ the function $t \in [0, \infty) \to \langle F(t)v, v\rangle_{\mathcal{H}} $ is non constant.
		\end{enumerate}
		Moreover, $\mathcal{H}_{K_{F}} \subset C_{0}(\mathbb{R}^{m}, \mathcal{H})$ if and only if for every $v \in \mathcal{H}\setminus\{{0}\}$ the function $t \in [0, \infty) \to \langle F(t)v, v\rangle_{\mathcal{H}} \in C_{0}([0, \infty))$. Under this additional hypothesis, the following are equivalent
		\begin{enumerate}
			\item[(i)] The kernel $K$ is strictly positive definite. 
			\item[(ii)] The kernel $K$ is $C_{0}$-universal.
			\item[(iii)] For every $v \in \mathcal{H}\setminus\{{0}\}$ the function $t \in [0, \infty) \to \langle F(t)v, v\rangle_{\mathcal{H}} $ is nonzero.
		\end{enumerate}	
	\end{thm}
	
	In \cite{schoenb}, besides the integral representation at Equation \ref{pdonallEuclid} for the continuous positive definite radial kernels on all Euclidean spaces, it is also proved  a very useful equivalence: a continuous function $f:[0, \infty) \to \mathbb{R}$ admits an integral representation as  Equation \ref{pdonallEuclid} if and only if the function $g(t):=f(\sqrt{t})$ is  \textbf{completely monotone}, that is $g \in C^{\infty}((0,\infty))$ and $(-1)^{n}g^{(n)}(t) \geq 0$ for every $t>0$ and $n\in \mathbb{N}$.
	
	Although the class of completely monotone functions has some very important examples,  it has the problem that any function that satisfies it is highly regular and do not have compact support, which might be troublesome in some applications, as argued in \cite{compcsimgraph}. A continuous function $f: [0, \infty)\to \mathbb{R}$ is called \textbf{$\ell$-times completely monotone} ($\ell \geq2$) if  
	\begin{enumerate}
		\item[(i)] $f \in C^{(\ell-2)}((0, \infty))$;
		\item[(ii)] f is nonnegative;
		\item[(iii)] $(-1)^{(\ell-2)}f^{(\ell -2)(t)}$ is convex;
		\item[(iv)] $\lim_{t \to \infty} f(t)$ exists;
	\end{enumerate} 
	
	Note that a function is completely monotone  if and only if is $\ell$-times completely monotone for every $\ell\geq 2$.  In \cite{askey1}, \cite{askey2} it is proved that for a $\ell$-times completely monotone function $f$,  the kernel $(x,y) \in \mathbb{R}^{m} \times \mathbb{R}^{m} \to f(\|x-y\|) \in \mathbb{R}$ is positive definite for $m \leq 2\ell -3$. Moreover,  a $\ell$-times completely monotone function $f$ can be  represented as 
	$$
	f(t)= \int_{[0, \infty)}(1-rt)_{+}^{\ell -1}d\lambda(r)   
	$$
	for some finite scalar valued nonnegative finite Radon measure $\lambda$ on $[0,\infty)$. Conversely, every function with this representation is also a $\ell$-times completely monotone function and the representation is unique  as proved in \cite{williamson}.

	%Another example of Proposition \ref{compcsuppor} is an operator valued generalization of the Askey class of positive definite kernels.
	
	\begin{thm}\label{universaleuclidianaskey}Let $\ell \in \{ 2,3, \cdots, \infty \}$,    $F: [0, \infty) \to \mathcal{L}(\mathcal{H})$ be an ultraweakly continuous function for which  $F_{v}$ is $\ell$-times completely monotone for every $v \in \mathcal{H}$ and $F(0) \in \mathcal{L}(\mathcal{H})$ is  a trace class operator. Then, the kernel 	
		$$
		(x,y) \in 	\mathbb{R}^{m} \times \mathbb{R}^{m} \to  K_{F}(x,y):= F(\|x-y\|) \in   \mathcal{L}(\mathcal{H}),
		$$
		is positive definite for  $m \leq 2\ell -3$, $\mathcal{H}_{K_{F}} \subset C(\mathbb{R}^{m}, \mathcal{H})$ and the following are equivalent
		\begin{enumerate}
			\item[(i)] The kernel  $K_{F}$ is strictly positive definite. 
			\item[(ii)] The kernel $K_{F}$ is universal.
			\item[(iii)] For every $v \in \mathcal{H}\setminus\{{0}\}$ the function $t \in [0, \infty) \to \langle F(t)v, v\rangle_{\mathcal{H}} $ is non constant.
		\end{enumerate}
		Moreover, $\mathcal{H}_{K_{F}} \subset C_{0}(\mathbb{R}^{m}, \mathcal{H})$ if and only if for every $v \in \mathcal{H}\setminus\{{0}\}$ the function $t \in [0, \infty) \to \langle F(t)v, v\rangle_{\mathcal{H}} \in C_{0}([0, \infty))$. Under this addtional hypothesys, the following are equivalent
		\begin{enumerate}
			\item[(i)] The kernel $K_{F}$ is strictly positive definite. 
			\item[(ii)] The kernel $K_{F}$ is $C_{0}$-universal.
			\item[(iii)] For every $v \in \mathcal{H}\setminus\{{0}\}$ the function $t \in [0, \infty) \to \langle F(t)v, v\rangle_{\mathcal{H}} $ is nonzero.
		\end{enumerate}	
	\end{thm}
	
	We emphasize that the kernel that we define on a $\ell$-times completely monotone function ($g(\|x-y\|)$, on \cite{askey1}, \cite{askey2}) and the one we define for a completely monotone function ($g(\|x-y\|^{2})$, by reescaling \cite{schoenb}) are different. In particular, the set of kernels defined  by completely monotone functions using $\|x-y\|$, instead of $\|x-y\|^{2}$ is a proper subset from the Schoenberg result. A characterization of this space and further information can be found in \cite{gneiting}.

	%In \cite{schoenb}, it is a proved a characterization of the positive definite radial kernels of a fixed dimension as well, precisely, if $m \in \mathbb{N}$ and  $f: [0, \infty) \to \mathbb{R}$ is a continuous function,   the kernel $
	%(x,y) \in \mathbb{R}^{m} \times \mathbb{R}^{m} \to f(\|x-y\|) \in \mathbb{R}
	%$
	%is positive definite if and only if 
	%$$
	%f(\|x-y\|)= \int_{[0, \infty)} \Omega_{m}(r\|x-y\|)d\lambda(r),
	%$$
	%where $\lambda$ is a scalar valued nonnegative finite Radon measure on $[0,\infty)$ and $\Omega_{m}:[0, \infty) \to \mathbb{R}$ satisfy
	%$$
	%  \Omega_{m}(\|x\|)= \frac{1}{Vol(S^{m-1})}\int_{S^{m-1}} e^{-ix\xi}d\xi, \quad x \in \mathbb{R}^{m}.
	%$$
	%The kernel  $(x,y) \in \mathbb{R}^{m} \times \mathbb{R}^{m} \to \Omega_{m}(\|x-y\|) \in \mathbb{R}$ is strictly positive definite($m\geq 2$) but it is not universal, however, the kernel
	%$$
	%(x,y) \in \mathbb{R}^{m-1} \times \mathbb{R}^{m-1} \to \Omega_{m}(\|x-y\|) \in \mathbb{R}
	%$$
	%is $C^{\infty}$ universal (we prove the last affirmation at Lemma \ref{radialfini}) and as a direct consequence we have the following consequence on the operator valued setting.  
	
	For the final example of this section, we need the following interesting result.
	
	\begin{lem}\label{radialfini}Let $m\geq 2$. The kernel
		$$
		(x,y) \in \mathbb{R}^{m} \times \mathbb{R}^{m} \to \Omega_{m}^{m}(x,y):=\Omega_{m}(\|x-y\|) \in \mathbb{R}
		$$	
		is not universal and  the kernel
		$$
		(x,y) \in \mathbb{R}^{m-1} \times \mathbb{R}^{m-1} \to \Omega_{m}^{m-1}(x,y):=\Omega_{m}(\|x-y\|) \in \mathbb{R}
		$$
		is $C^{\infty}$-universal  but it is not  $C_{0}$-universal.\\
		In particular, if a nonzero function $f:[0, \infty) \to \mathbb{R}$ is an element of $C^{2q}([0, \infty))$ and the kernel
		$$
		(x,y) \in \mathbb{R}^{m}\times \mathbb{R}^{m} \to f(\|x-y\|) \in \mathbb{R}
		$$
		is positive definite, then the kernel
		$$
		(x,y) \in \mathbb{R}^{m-1}\times \mathbb{R}^{m-1} \to f(\|x-y\|) \in \mathbb{R}
		$$
		is $C^{q}$-universal.
	\end{lem}

	\begin{thm}\label{universaleuclidian3}Let $m \geq 2$,   $F: [0, \infty) \to \mathcal{L}(\mathcal{H})$ be an ultraweakly continuous function for which the kernel $(x,y)\in \mathbb{R}^{m}\times \mathbb{R}^{m}\to  \langle F(\|x-y\|)v,v \rangle_{\mathcal{H}}  \in \mathbb{C}$
		is positive definite  for every $v \in \mathcal{H}$ and $F(0) \in \mathcal{L}(\mathcal{H})$ is  a trace class operator. Then, the kernel 
		$$
		(x,y)\in \mathbb{R}^{m-1}  \times \mathbb{R}^{m-1} \to K(x,y):=F(\|x-y\|) \in \mathcal{L}(\mathcal{H})
		$$	
		is such that $\mathcal{H}_{K} \subset C(\mathbb{R}^{m-1}, \mathcal{H})$ and the following are equivalent
		\begin{enumerate}
			\item[(i)] The kernel $K$ is strictly positive definite. 
			\item[(ii)] The kernel $K$ is universal.
			\item[(iii)] For every $v \in \mathcal{H}\setminus\{{0}\}$ the function $t \in [0, \infty) \to \langle F(t)v, v\rangle_{\mathcal{H}} $ is non constant.
		\end{enumerate}
	\end{thm}

	Theorem \ref{universaleuclidian2}, Theorem \ref{universaleuclidianaskey} and Theorem \ref{universaleuclidian3} still holds true if  the finite nonnegative operator valued measure that uniquely describes $F$ has finite variation and admits a Radon-Nikod\'ym decomposition. But, even if the function $F$ as above is complex valued, it is difficult to obtain the measure that describes $F$ by simply using the function $F$. By assuming that $F(0)$ is a trace class operator, the setting allow us to use Lemma \ref{radonnikodintrace}. Further, we expect that the trace class assumption will be more relevant on applications because of Lemma \ref{compacinclusion}.

	\section{\textbf{On differentiable operator valued universal (and related) kernels}}\label{On differentiable positive definite operator valued universal (and related) kernels}

	In this section, we first deal with extensions of some well known results related to differentiable positive definite kernels (matrix and scalar valued) to the operator valued context. From those results we obtain a characterization for differentiable universality on the operator valued context and we obtain a differentiable version of Theorem \ref{universalityintegralkernel}. We finish this section presenting several families of operator valued kernels for which the differentiable universalities are well behaved with respect to the scalar valued projections of the kernel, with an emphasis on radial kernels.    
	
	\subsection{\textbf{On differentiable kernels and properties of its RKHS}\label{On differentiable positive definite kernels and properties of its RKHS}}
	Throughout  this subsection, let $\mathcal{A}  \subset \mathbb{R}^{m}$ be a compact set for which $ \overline{Int(\mathcal{A})}= \mathcal{A}$. This condition is not essential for the development of this subject, but simplifies several arguments. Note that for every open set $U  \subset\mathbb{R}^{m}$, there exists a sequence of compact sets $\mathcal{A}_{n}$, $n \in \mathbb{N} $, such that $\overline{Int(\mathcal{A}_{n})}= \mathcal{A}_{n}  \subset Int(\mathcal{A}_{n+1})  \subset U $, and $\cup_{n \in \mathbb{N} } \mathcal{A}_{n}= U$.  We define
	$$   
	C^{q}(\mathcal{A}, \mathcal{H}):= \{ F: \mathcal{A} \to \mathcal{H}\mid  F   \in C^{q}(Int(\mathcal{A}), \mathcal{H}),  \alpha \in \mathbb{Z}_{+}^{m}, |\alpha|\leq q,  \partial^{\alpha}F=F_{\alpha}\in C(\mathcal{A}, \mathcal{H}) \}
	$$
	Because $ \overline{Int(\mathcal{A})}= \mathcal{A}$, the functions  $F_{\alpha}$ are unique and we write $\partial^{\alpha} F(x)$ even for points in $\mathcal{A}\setminus{  Int(\mathcal{A})}$. In particular, $C^{q}(\mathcal{A}, \mathcal{H})$ is a vector space and admits a natural complete norm by setting
	$$
	\|F\|_{C^{q}(\mathcal{A}, \mathcal{H})}:= \sum_{\alpha \in \mathbb{Z}_{+}^{m}}^{|\alpha| \leq q}\sup_{x \in \mathcal{A}}\|\partial^{\alpha}F(x)\|_{\mathcal{H}}= \sum_{\alpha \in \mathbb{Z}_{+}^{m}}^{|\alpha| \leq q}\|\partial^{\alpha}F\|_{C(\mathcal{A}, \mathcal{H})}.
	$$
	Similarly, for an open set $U \subset \mathbb{R}^{m}$, the vector space $C^{q}_{0}(U, \mathcal{H})$ admits a natural complete norm by setting
	$$
	\|F\|_{C_{0}^{q}(U, \mathcal{H})}:= \sum_{\alpha \in \mathbb{Z}_{+}^{m}}^{|\alpha| \leq q}\sup_{x \in U}\|\partial^{\alpha}F(x)\|_{\mathcal{H}}= \sum_{\alpha \in \mathbb{Z}_{+}^{m}}^{|\alpha| \leq q}\|\partial^{\alpha}F\|_{C_{0}(U, \mathcal{H})}.
	$$

	First we prove a differentiable version of Proposition \ref{rkhscontained}. This result is a generalization  of Theorem $2.11$ in \cite{Michelli2014} to the operator valued setting. 
	\begin{prop}\label{diffimersion} Let $K : U \times U \to \mathcal{L}(\mathcal{H})$ be a positive definite kernel. Then $ \mathcal{H}_{K} \subset  C^{q}(U, \mathcal{H}) $ if and only if the functions
		$$
		y \in U \to K(x,y)v \in C^{q}(U, \mathcal{H}), \quad  \text{ for every }  x \in U, \quad  v \in \mathcal{H}
		$$  
		$$
		x \in U \to \partial_{2}^{\beta}[K(x,y)v]= \partial^{\beta}_{y}K_{y}v \in C^{q}(U, \mathcal{H}), \quad  \text{ for every }  y \in U, \quad  v \in \mathcal{H},\quad  |\beta |\leq q.
		$$ 
		and for every compact set $\mathcal{A} \subset U$ 
		$$ \| \partial_{1}^{\alpha}[\partial_{2}^{\beta}[K(x,y)v]]\|_{\mathcal{H}} \leq M_{\mathcal{A}}\|v\|, \quad x,y \in \mathcal{A}, \quad |\alpha|, |\beta|\leq q, \quad v \in \mathcal{H} $$  
		Additionally, the following properties are satisfied
		\begin{enumerate}   
			\item[(i)] For every $y \in U$ , $v \in \mathcal{H}$ and   $|\beta| \leq q$, $\partial^{\beta}_{y}K_{y}v\in \mathcal{H}_{K}$.
			\item[(ii)]	If $F \in \mathcal{H}_{K}$, then  $F\in C^{q}(U , \mathcal{H} ) $ and $
			\langle v, \partial^{\alpha} F(x)\rangle_{\mathcal{H}}= \langle \partial^{\alpha}_{x}K_{x}v , F \rangle_{\mathcal{H}_{K}} 
			$.
			\item[(iii)] The inclusion-restriction  $I : \mathcal{H}_{K} \to C^{q}(\mathcal{A} , \mathcal{H} )  $ is continuous.
			\item[(iv)]For every $v_{\alpha} \in \mathcal{H}$, $|\alpha|\leq q$, the  matrix valued  kernel 
			$$
			(x,y) \in\mathbb{R}^{m}\times\mathbb{R}^{m}  \to \langle \partial_{1}^{\alpha} [\partial_{2}^{\beta}[K(x,y)v_{\beta}]], v_{\alpha} \rangle_{\mathcal{H}} \in M_{\ell}(\mathbb{C})
			$$
			is positive definite.
		\end{enumerate}	
	\end{prop}
	
	We emphasize that a kernel does not need to be jointly differentiable in order that $\mathcal{H}_{K} \subset C^{q}(X, \mathcal{H})$. For instance, the kernel 
	$$
	(x,y) \in \mathbb{R}^{m} \times \mathbb{R}^{m} \to k(x,y):= \frac{\langle x,y\rangle^{q}}{\|x\|^{2} + \|y\|^{2}} \in  \mathbb{R}
	$$ is positive definite because $k(x,y)= \int_{[0, \infty)}\langle x,y\rangle^{k} e^{-r\|x\|^{2} - r\|y\|^{2}}dr$ and $\mathcal{H}_{k} \subset C_{0}^{q}(\mathbb{R}^{m})$ by Proposition \ref{diffimersion}, however  $k\notin C^{q,q}(\mathbb{R}^{m} \times \mathbb{R}^{m})$. Also, the linear operator $\partial_{1}^{\alpha}[\partial_{2}^{\beta}K(x,y)]$ is well defined and is an element of $\mathcal{L}(\mathcal{H})$  by Proposition  \ref{diffimersion}, however we do not expect  that the operator valued kernel is separably differentiable, but we have not found an example for this.
	
	Often we encounter stronger conditions compared to the ones at Proposition \ref{diffimersion}, for instance that the kernel $K \in C^{q,q}(U\times U, \mathcal{L}(\mathcal{H}))$. On this case, it holds that 
	$$
	\partial_{1}^{\alpha} [\partial_{2}^{\beta}[K(x,y)v]]=[\partial_{1}^{\alpha} \partial_{2}^{\beta}K(x,y)]v.
	$$

	\begin{prop}\label{diffimersion0} Let $K : U \times U \to \mathcal{L}(\mathcal{H})$ be a positive definite kernel. Then $\mathcal{H}_{K} \subset C^{q}_{0}(U, \mathcal{H})$ if and only if the functions
		$$
		y \in U \to K(x,y)v \in C^{q}_{0}(U, \mathcal{H}), \quad  \text{ for every }  x \in U, \quad  v \in \mathcal{H}
		$$  
		$$
		x \in U \to \partial_{2}^{\beta}[K(x,y)v]= \partial^{\beta}_{y}K_{y}v \in C^{q}_{0}(U, \mathcal{H}), \quad  \text{ for every }  y \in U, \quad  v \in \mathcal{H},\quad  |\beta |\leq q.
		$$
		and           
		$$
		\| \partial_{1}^{\alpha}[\partial_{2}^{\beta}[K(x,y)v]] \|_{\mathcal{H}} \leq M\|v\|, \quad |\alpha|, |\beta| \leq q , \quad x,y \in U, \quad  v \in \mathcal{H}.
		$$  
		Additionally, the following properties are satisfied
		\begin{enumerate}
			\item[(i)] For every $x \in U$ , $v \in \mathcal{H}$ and   $|\beta| \leq q$, $\partial^{\beta}_{y}K_{y}v\in \mathcal{H}_{K}$.   
			\item[(ii)]	If $F \in \mathcal{H}_{K}$, then  $F\in C^{q}_{0}(U , \mathcal{H} ) $ and $
			\langle v, \partial^{\alpha} F(x)\rangle_{\mathcal{H}}= \langle \partial^{\alpha}_{x}K_{x}v , F \rangle_{\mathcal{H}_{K}}. 
			$
			\item[(iii)] The inclusion $I : \mathcal{H}_{K} \to C^{q}_{0}(U, \mathcal{H} )  $ is continuous.
			\item[(iv)]For every $v_{\alpha} \in \mathcal{H}$, $|\alpha|\leq q$, the  matrix valued  kernel 
			$$
			(x,y) \in\mathbb{R}^{m}\times\mathbb{R}^{m}  \to \langle \partial_{1}^{\alpha} [\partial_{2}^{\beta}[K(x,y)v_{\beta}]], v_{\alpha} \rangle_{\mathcal{H}} \in M_{\ell}(\mathbb{C})
			$$
			is positive definite.
		\end{enumerate}	
	\end{prop}

	An immediate generalization of  the Arzel\`{a}-Ascoli Theorem to the differentiable setting is that a closed subset $B$ of $C^{q}(\mathcal{A}, \mathcal{H})$ is compact if and only if 
	\begin{enumerate}
		\item[$\circ$] $B$ is \textbf{$C^{q}$-equicontinuous}: For every $\epsilon >0$ and $x \in \mathcal{A}$ there is open set  $U_{x}$ that contains $x$  for which $\|\partial^{\alpha}F(x) - \partial^{\alpha}F(y)\|_{\mathcal{H}}< \epsilon $ for all $F \in B$, $y \in U_{x}$ and $|\alpha|\leq q$.
		\item[$\circ$]  $B$ is \textbf{$C^{q}$-pointwise relatively compact}: The set $\{\partial^{\alpha}F(y), F \in B, |\alpha|\leq q   \} \subset \mathcal{H}$ has compact closure on the norm topology of $\mathcal{H}$.
	\end{enumerate}

	Below we prove a differentiable version of Lemma \ref{compacinclusion}.

	\begin{lem}\label{compacinclusiondiff} Let $U\subset \mathbb{R}^{m}$ be an open set and $K: U\times U \to \mathcal{L}(\mathcal{H})$ be a positive definite kernel. Suppose that $\mathcal{H}_{K} \subset C^{q}(U, \mathcal{H})$, then 
		\begin{enumerate}
			\item[(i)] Every bounded set $B \subset \mathcal{H}_{K}$ is $C^{q}$-equicontinuous if and only if $K \in C^{q,q}(U\times U, \mathcal{L}(\mathcal{H}))$.
			\item[(ii)] If $(u,v) \in \mathcal{H}\times \mathcal{H} \to \langle  [\partial^{\alpha}_{1}[\partial^{\alpha}_{2}K(x,x)u]], v \rangle_{\mathcal{H}} \in \mathbb{C}$ is a trace class operator for every $x \in U$ and $|\alpha |\leq q$, then for every  bounded set $B \subset \mathcal{H}_{K}$ and $y \in U$, the set $\{\partial^{\alpha}F(y), \quad  F \in B  , |\alpha| \leq q\} \subset \mathcal{H}$ is  $C^{q}$-pointwise relatively compact.
		\end{enumerate}
		In particular, by mixing those two results we obtain that if $ \langle  [\partial^{\alpha}_{1}[\partial^{\alpha}_{2}K(x,x)u]], v \rangle_{\mathcal{H}}$ is a trace class operator for every $x \in U$ and $|\alpha |\leq q$, then
		\begin{enumerate}
			\item[(iii)] The  continuous inclusion-restriction $I: \mathcal{H}_{K} \to C^{q}(\mathcal{A}, \mathcal{H})$ is compact  for  every compact set $\mathcal{A}= \overline{Int(\mathcal{A})} \subset U$ if and only if $K \in C^{q,q}(U\times U, \mathcal{L}(\mathcal{H}))$. Under this hypothesis, every closed and bounded set of $\mathcal{H}_{K}$ (restricted to $\mathcal{A}$) is a compact set of $C^{q}(\mathcal{A}, \mathcal{H})$.
			\item[(iv)] If $\mathcal{H}_{K} \subset C_{0}^{q}(U, \mathcal{H})$, then the set
			$$
			B_{\mathcal{A}}: =B \cap \overline{span\{K_{x}v, \quad , x \in \mathcal{A}, v \in \mathcal{H} \}  }_{\mathcal{H}_{K}}
			$$
			has compact closure on $C_{0}(U, \mathcal{H})$ for every bounded set $B$ of $\mathcal{H}_{K}$ and compact set $\mathcal{A} \subset U$ if and only if  $K \in C^{q,q}(U\times U, \mathcal{L}(\mathcal{H}))$. Under this hypothesis, if $B$ is  closed  then $B_{\mathcal{A}}$ is a compact set of $C_{0}^{q}(U, \mathcal{H})$.
		\end{enumerate}
	\end{lem}

	\begin{defn} Let $U \subset \mathbb{R}^{m}$ be an open set, $\Omega$  a locally compact Hausdorff space,  $p : \Omega \times U \times U \to \mathbb{C} \in C^{0, q, q}(\Omega\times U \times U)$ and $\lambda$ be scalar valued nonnegative finite Radon measure on $\Omega$. We say that $p$ is \textbf{$C^{q}$-Dominated with respect to $\lambda$} if there exists a $\lambda$ integrable function $h: \Omega \to \mathbb{C}$ such that 
		$$
		|\partial_{1}^{\alpha}\partial_{2}^{\beta}p_{w}(x,y)| \leq |h(w)|
		$$
		for all $x,y \in X$ and $|\alpha|, |\beta| \leq q$. 
	\end{defn}
	Note that on the previous definition,  if $p_{w}$ is a positive definite kernel for all $w$,  then it is sufficient to analyze the case $x=y$ and $\alpha=\beta$, because the matrix valued kernel presented at Lemma \ref{diffimersion} is positive definite. The interest on the previous definition is the following Lemma, where we obtain an explicit description of the derivatives of an  operator valued  kernel.

	\begin{lem}\label{diffkernelfourtrans} Let $\Omega$ be a locally compact space, $U \subset \mathbb{R}^{m}$ an open set, $p : \Omega \times U \times U \to \mathbb{C} \in C^{0,q,q}$ such that the kernels $$
		(x,y) \in U \times U \to p_{w}(x,y):=p(w,x,y) 
		$$ are positive definite for all $w \in \Omega$, and $\Lambda : \mathscr{B}(\Omega) \to \mathcal{L}(\mathcal{H})$ be a Radon nonnegative finite operator valued measure that admits a Radon-Nikod\'ym decomposition as $d\Lambda= G d\lambda$. Suppose in addition that the function $p$  is $C^{q}$-dominated with respect to $|\Lambda|$, then the kernels
		$$
		(x,y) \in \mathbb{R}^{m} \times \mathbb{R}^{m} \to P_{\alpha, \beta}(x,y):=\int_{\Omega} \partial^{\alpha}_{1}\partial^{\beta}_{2}p_{w}(x,y)d\Lambda(w) \in \mathcal{L}(\mathcal{H}), \quad  |\alpha|, |\beta|\leq q
		$$	
		are well defined, $P_{\alpha, \beta}\in C^{q-|\alpha|, q-|\beta|}( \mathbb{R}^{m}\times \mathbb{R}^{m}, \mathcal{L}(\mathcal{H}))$ and
		$$
		\partial^{\alpha}_{1}\partial^{\beta}_{2}P(x,y)= P_{\alpha, \beta}(x,y), \quad |\alpha|, |\beta|\leq q.
		$$
	\end{lem}

	On the special case where $p_{w}(x,y)= e^{-iw(x-y)}$, usually it is defined $h(w):=\|w\|^{2q}$. Actually, it can be proved that the integral of $|h|$ with respect to the finite measure $\|G\|_{\mathcal{L}(\mathcal{H})}d\lambda$  is finite if and only if the integrals of the functions $|\partial_{1}^{\alpha}\partial_{2}^{\alpha}e^{-iw(x-y)}|= |w^{2\alpha}|$ with respect to the finite measure $\|G(w)\|_{\mathcal{L}(\mathcal{H})}d\lambda$ are finite, for all $|\alpha|=p$. In other words, the fact that the kernels $P_{\alpha, \beta}$ are well defined (in the sense that are defined by a Bochner integral) implies the differentiation properties, which does not hold on the general setting. On the following result we prove a similar property for the kernels  presented at Corollary \ref{compcsuppor} based on the main result of \cite{gneiting2}.
	
	\begin{lem}\label{diffform}Let $f: \mathbb{R}^{m}\to \mathbb{C}$ be a  non constant function in $C^{2q}(\mathbb{R}^{m})$ that defines a positive definite kernel on the Bochner's sense. If $\lambda$ is a scalar valued nonnegative finite Radon measure on $[0,\infty)$, then the kernel
		$$ (x,y) \in \mathbb{R}^{m} \times \mathbb{R}^{m} \to p(x,y):= \int_{[0,\infty)}f(w(x-y))d\lambda(w) \in \mathbb{C} $$ 
		is an element of $C^{q,q}(\mathbb{R}^{m} \times \mathbb{R}^{m})$ if and only if 
		$$
		\int_{[0, \infty)} w^{2q}d\lambda(w) < \infty.
		$$
		Moreover 
		$$
		\partial^{\alpha}_{1}\partial^{\beta}_{2}p(x,y)= \int_{[0, \infty)}\partial^{\alpha}_{x}\partial^{\beta}_{y}f(w(x-y))d\lambda(w)
		$$
		and $\partial^{\alpha}_{1}\partial^{\beta}_{2}p(x,x)=(-1)^{|\beta|}f^{(\alpha +\beta)}(0) 	\int_{[0, \infty)} w^{|\alpha| +|\beta|}d\lambda(w)$ 
	\end{lem}

	\subsection{\textbf{Characterization of differentiable universality}} \label{Sufficient conditions for operator valued positive definite differentiable universal (and related) kernels}
	
	The analysis of $C^{q}$-universal  and $C^{q}_{0}$-universal kernels follows a similar path as the universality on $C(U, \mathcal{H})$ and $C_{0}(U,\mathcal{H})$. The linear operator
	$$
	F\in  C^{q}(\mathcal{A}, \mathcal{H}) \to \partial F := ( \partial^{\alpha}F )_{|\alpha| \leq q}  \in   \prod_{|\alpha| \leq q}  C(\mathcal{A}, \mathcal{H})
	$$
	is a continuous isometry (on the product space we are defining the sum norm). If $T : C^{q}(\mathcal{A}, \mathcal{H}) \to \mathbb{C}$ is a continuous operator, by a very famous consequence of the Hanh-Banach Theorem there exists a continuous linear operator  $J : \prod_{|\alpha| \leq q}  C(\mathcal{A}, \mathcal{H}) \to \mathbb{C}$ such that $J(\partial (F)) = T(F)$ for all $F \in C^{q}(\mathcal{A}, \mathcal{H}))$. 
	
	But $J( (\phi_{\alpha})_{|\alpha| \leq q} ) =  \sum_{|\alpha| \leq q} J(\phi_{\alpha}) $, where  we made the abuse of notation $(\phi_{\alpha})_{\beta}=: \delta_{\alpha, \beta} \phi_{\alpha}$. In other words,  $\phi_{\alpha} \in  \prod_{|\alpha| \leq q} C(\mathcal{A}, \mathcal{H})$, but except for the coordinate $\alpha$ it is the zero function on $C(\mathcal{A}, \mathcal{H})$. By the choice of the norms we are working, $J(\phi_{\alpha})$ is a continuous linear functional defined on $C(\mathcal{A}, \mathcal{H})$, by Theorem \ref{Dinculeanu-Singer}, there exists a measure $\eta_{\alpha} \in \mathfrak{M}(\mathcal{A}, \mathcal{H})$ for which
	$$  
	J(\phi_{\alpha})= \int_{\mathcal{A}} \langle \phi_{\alpha}(x), d\eta_{\alpha}(x) \rangle. 
	$$
	In particular, we obtain that
	$$
	T(F)= J(\partial(F))= J((\partial^{\alpha}F)_{|\alpha| \leq q})= \sum_{|\alpha| \leq q} J(\partial^{\alpha} F)= \sum_{|\alpha| \leq q} \int_{\mathcal{A}} \langle \partial^{\alpha}F(x), d\eta_{\alpha}(x) \rangle.
	$$
	And conversely, if $\eta_{\alpha} \in \mathfrak{M}(\mathcal{A}, \mathcal{H}) $, $\alpha \in \mathbb{Z}_{+}^{m}$ and $|\alpha| \leq q$, the linear operator
	$$
	T(F)= \sum_{|\alpha| \leq q} \langle \partial^{\alpha}F(x), d\eta_{\alpha}(x) \rangle.
	$$
	is continuous. A similar reasoning can be done for the continuous linear functionals in $C_{0}(U, \mathcal{H})$. This proves the following Corollary of Theorem \ref{Dinculeanu-Singer}.

	\begin{cor}\label{Dinculeanu-Singer-diff}  Let $\mathcal{A} \subset U$ be a compact set for which $ \overline{Int(\mathcal{A})}= \mathcal{A}$ ($U \subset \mathbb{R}^{m}$ an open set) and $L : C^{q}(\mathcal{A}, \mathcal{H}) \to \mathbb{C} $ ($L : C_{0}^{q}(U, \mathcal{H}) \to \mathbb{C}$) be a continuous linear functional. Then for each $\alpha \in \mathbb{Z}_{+}^{m}$ and $|\alpha| \leq q$ there exists a Radon measure of bounded variation $\eta_{\alpha} \in \mathfrak{M}(\mathcal{A}, \mathcal{H})$ ($\mathfrak{M}(U, \mathcal{H})$)  for which
		$$
		L(F)= \sum_{|\alpha| \leq q} \int_{\mathcal{A}} \langle \partial^{\alpha}F(x), d\eta_{\alpha}(x) \rangle, \quad L(F)=\sum_{|\alpha| \leq q} \int_{U} \langle \partial^{\alpha}F(x), d\eta_{\alpha}(x) \rangle.
		$$	
	\end{cor}  
	In order to simplify the notation, we say that $\eta\in \mathfrak{M}^{q}(U, \mathcal{H})$ if $\eta= (\eta_{\alpha})_{|\alpha| \leq q}$ and each $\eta_{\alpha} \in \mathfrak{M}(U, \mathcal{H})$ (similar for a compact set $\mathcal{A}$). For an element $\eta \in \mathfrak{M}^{q}(U, \mathcal{H})$ ($\eta \in \mathfrak{M}^{q}(\mathcal{A}, \mathcal{H})$) and a function $F \in C^{q}_{0}(U, \mathcal{H})$ ($F \in C^{q}(\mathcal{A}, \mathcal{H})$) we use the following symbology two simplify the writing
	$$
	\int \langle \partial F(x), d\eta(x)\rangle :=\sum_{|\alpha| \leq q} \int \langle \partial F^{\alpha}(x), d\eta_{\alpha}(x)\rangle.
	$$
	
	We remark that Corollary \ref{Dinculeanu-Singer-diff} does not mean that  $\mathfrak{M}^{q}(U, \mathcal{H})$($\mathfrak{M}^{q}(\mathcal{A}, \mathcal{H})$) is the actual dual space of $C_{0}^{q}(U, \mathcal{H})$($C^{q}(\mathcal{A}, \mathcal{H})$), on an isometric sense. This can be seen by the fact that $\{ \partial F, F \in  C_{0}^{q}(U, \mathcal{H})\}$ is not dense in $\prod_{|\alpha|\leq q}C_{0}(U, \mathcal{H})$ (similar for a compact set $\mathcal{A}$).%
	
	As a consequence, when we prove a result regarding $C^{q}$-universality or $C_{0}^{q}$-universality, an additional step, compared to Theorem $11$ of \cite{Caponnetto2008}, is necessary. The next Lemma is a differentiable version of  Theorem $11$ of \cite{Caponnetto2008}. 
	
	\begin{lem}\label{addstepunidiff} Let $K: U \times U \to \mathcal{L}(\mathcal{H})$ be a positive definite kernel such that  $\mathcal{H}_{K}\subset C_{0}^{q}(U, \mathcal{H})$ ($\mathcal{H}_{K}\subset C^{q}(U, \mathcal{H}) $). The kernel $K$ is $C_{0}^{q}$-universal ($C^{q}$-universal) if and only if a measure $\eta \in \mathfrak{M}^{q}(U, \mathcal{H})$ (for every compact set $\mathcal{A}$ for which $\mathcal{A}= \overline{Int(\mathcal{A})}$ and $\eta \in \mathfrak{M}^{q}(\mathcal{A}, \mathcal{H})$) such that 
		$$ 
		\int\langle \partial_{1}(K(x,y)v),d\eta(x) \rangle =0 	
		$$	 
		for all $y \in U$ and $v \in \mathcal{H}$ satisfies
		$$
		\int \langle \partial F,d\eta(x) \rangle =0 	
		$$
		for all $F \in C_{0}^{q}(U, \mathcal{H})$ (for all $F \in C^{q}(\mathcal{A}$)).\end{lem}
	
	The next Theorem is a differentiable version of  Theorem \ref{oneintegraltodoubleintegral}

	\begin{thm}\label{oneintegraltodoubleintegralcq} Let $K: U \times U \to \mathcal{L}(\mathcal{H})$ be a positive definite kernel such that  $\mathcal{H}_{K}\subset C_{0}^{q}(U, \mathcal{H})$ ($\mathcal{H}_{K}\subset C^{q}(U, \mathcal{H}) $). Then
		\begin{enumerate}
			\item [(i)]For every $\eta=(\eta_{\alpha})_{|\alpha|\leq q} \in \mathfrak{M}^{q}(U, \mathcal{H})$ (for any compact $\mathcal{A}\subset U$ with $\overline{Int(\mathcal{A})}=\mathcal{A}$ and $ \eta \in \mathfrak{M}^{q}(\mathcal{A}, \mathcal{H})$) the function $K_{\eta} : U \to \mathcal{H}$, defined as 
			$$
			y \in U   \to K_{\eta}(y):=\sum_{|\alpha|\leq q}  \int  \partial_{1}^{\alpha}K(x,y) d\eta_{\alpha}(x)  \in \mathcal{H}
			$$
			is an element of $\mathcal{H}_{K}$.
			\item[(ii)] For every $|\beta| \leq q$ it holds that 
			$$
			\partial^{\beta}K_{\eta}(y)= \sum_{|\alpha|\leq q}\int \partial_{1}^{\alpha}[\partial_{2}^{\beta}K(x,y)] d\eta_{\alpha}(x) 
			$$
			\item [(iii)] The following equality holds
			$$\langle K_{\eta}, K_{\eta}\rangle_{\mathcal{H}_{K}}=
			\sum_{|\alpha|,|\beta| \leq q}\int \langle \int \partial_{1}^{\alpha}[\partial_{2}^{\beta}K(x,y)] d\eta_{\alpha}(x), d\eta_{\beta}(y) \rangle , 
			$$
			and if $d\eta_{\alpha}= H_{\alpha}d|\eta|$ is a Radon-Nikod\'ym decomposition for all measures $\eta_{\alpha}$, then
			$$\langle K_{\eta}, K_{\eta}\rangle_{\mathcal{H}_{K}}=
			\sum_{|\alpha|,|\beta| \leq q} \int   \int \langle \partial_{1}^{\alpha}[\partial_{2}^{\beta}K(x,y)] H_{\beta}(y), H_{\alpha}(x)\rangle_{\mathcal{H}}  d|\eta|(x) d|\eta|(y).  
			$$
			
			\item [(iv)] The kernel is $C_{0}^{q}$-universal ($C^{q}$-universal) if and only if $\langle K_{\eta}, K_{\eta} \rangle > 0 $ for all $\eta \in \mathfrak{M}^{q}(U, \mathcal{H})$ ($ \eta \in \mathfrak{M}^{q}(\mathcal{A}, \mathcal{H})$) such that the continuous  linear functional in $C^{q}_{0}(U, \mathcal{H})$ ($C^{q}(\mathcal{A}, \mathcal{H})$)    
			$$
			\int\langle \partial F(x), d\eta(x) \rangle= \sum_{|\alpha| \leq q} \int\langle \partial^{\alpha} F(x), d\eta_{\alpha}(x) \rangle
			$$	
			is nonzero.	
		\end{enumerate}	
	\end{thm}
	
	Similar to the $C_{0}$-universal case, to verify if $\mathcal{H}_{K} \subset C_{0}(U, \mathcal{H})$ might be very complicated, but the analysis of the double sum  of the double integrals at relation $(iii)$ in Theorem \ref{oneintegraltodoubleintegralcq} can be easier. Then, if a positive definite kernel $K: U\times U \to \mathcal{L}(\mathcal{H}) \in C^{q,q}(U\times U, \mathcal{L}(\mathcal{H}))$   is such that the functions  
	$$
	x \in U \to \|\partial^{\alpha}_{1} \partial^{\alpha}_{2}K(x,x)\|_{\mathcal{L}(\mathcal{H})} \in \mathbb{C}
	$$
	are bounded for every $|\alpha|\leq q$, we say that $K$ is \textbf{$C^{q}$-integrally strictly positive definite} if a measure $\eta \in \mathfrak{M}^{q}(U, \mathcal{H})$ that satisfies
	$$
	\sum_{|\alpha|,|\beta| \leq q}\int \langle \int \partial_{1}^{\alpha}[\partial_{2}^{\beta}K(x,y)] d\eta_{\alpha}(x), d\eta_{\beta}(y) \rangle=0
	$$
	the continuous  linear functional in $C^{q}_{0}(U, \mathcal{H})$ 
	$$
	F \to \int\langle \partial F(x), d\eta(x) \rangle= \sum_{|\alpha| \leq q} \int\langle \partial^{\alpha} F(x), d\eta_{\alpha}(x) \rangle \in \mathbb{C}
	$$	
	is zero.	
	
	%Although not being  clear initially, those two double integrals have the same value. Indeed, if $H_{n}: X \to \mathcal{H}$ is a sequence of simple functions such that $\int\| H(x) -H_{n}(x)\|d|\eta|(x) \to 0$. Then   
	%$$
	%\int_{X}  \int_{X}\langle K(x,y)H_{n}(x),H_{n}(y) \rangle_{\mathcal{H}}d|\eta|(y)  d|\eta|(x) \to \int_{X}  \int_{X}\langle K(x,y)H(x),H(y) \rangle_{\mathcal{H}}d|\eta|(y)  d|\eta|(x)
	%$$
	%$$
	%\int_{X} \int_{X} \langle   K(x,y)H_{n}(y),H_{n}(x) \rangle_{\mathcal{H}} d|\eta|(x)  d|\eta|(y) \to \int_{X} \int_{X} \langle   K(x,y)H(y),H(x) \rangle_{\mathcal{H}} d|\eta|(x)  d|\eta|(y)
	%$$
	%and also, if $H_{n}= \sum_{\mu} v_{\mu}\chi_{A_{\mu}}$, then
	%$$
	%\int_{X}  \int_{X}\langle K(x,y)H_{n}(x),H_{n}(y) \rangle_{\mathcal{H}}d|\eta|(y)  d|\eta|(x)= \sum_{\mu, \nu }\int_{A_{u}}  \int_{A_{\nu}}\langle K(x,y)v_{\mu},v_{\nu} \rangle_{\mathcal{H}}d|\eta|(y)  d|\eta|(x)
	%$$
	%$$
	%\int_{X} \int_{X} \langle   K(x,y)H_{n}(y),H_{n}(x) \rangle_{\mathcal{H}} d|\eta|(x)  d|\eta|(y) = \sum_{\mu, \nu}\int_{A_{\nu}} \int_{A_{\mu}} \langle   K(x,y)v_{\nu},v_{\mu} \rangle_{\mathcal{H}} d|\eta|(x)  d|\eta|(y).
	%$$
	%Precisely, the $(\mu, \nu)$ term on the first double sum is equal to the $(\nu, \mu)$ term on the second sum.

	Now we specialize Theorem \ref{universalityintegralkernel}  to the differentiable setting.

	\begin{thm}\label{Cpuniversalityintegralkernel} Let $U \subset \mathbb{R}^{m}$ be an open set, $\Omega$  a locally compact Hausdorff space,  $p : \Omega \times U \times U \to \mathbb{C} \in C^{0, q, q}$ such that the kernel $$
		(x,y) \in U\times U \to p_{\omega}(x,y):=p(\omega, x,y), 
		$$  is positive definite for all $w \in \Omega$. Given a Radon nonnegative finite operator $\mathcal{L}(\mathcal{H})$ valued measure $\Lambda: \mathscr{B}(\Omega) \to \mathcal{L}(\mathcal{H})$  such that $\Lambda$ admits a  Radon-Nikod\'ym decomposition as $d\Lambda= Gd\lambda$, assume that $p$ is $C^{q}$-Dominated with respect to $\|G\|_{\mathcal{L}(\mathcal{H})}d\lambda= d|\Lambda|$. Consider the operator valued kernel
		$$	
		P : U\times U \to \mathcal{L}(\mathcal{H}), \quad P(x,y)= \int_{\Omega}p_{w}(x,y) d\Lambda(w),	
		$$
		for every $v \in \mathcal{H}\setminus{\{0\}}$ the scalar valued kernels  and the scalar valued nonnegative finite Radon measures
		$$
		P_{v} : U\times U \to \mathbb{C}, \quad P_{v}(x,y)= \langle P(x,y)v, v \rangle_{\mathcal{H}} 
		$$ 
		$$
		\Lambda_{v}: \mathscr{B}(\Omega) \to \mathbb{C}, \quad    
		\Lambda_{v}(A)= \langle \Lambda(A)v, v \rangle_{\mathcal{H}}.  	
		$$
		Then:
		\begin{enumerate}
			\item[(i)] The function $P : U \times U \to \mathcal{L}(\mathcal{H})$  is well defined and belongs to $C^{q,q}(U \times U, \mathcal{L}(\mathcal{H}))$. In particular, $\mathcal{H}_{P} \subset C^{q}(U, \mathcal{H})$.
			\item[(iii)] If the kernel $p_{\omega}(x,y)$ is $C^{q}$-universal for every $w \in \Omega$, then the operator valued kernel $P$ is $C^{q}$-universal universal if and only if   all the scalar valued kernels $P_{v}$ are $C^{q}$-universal.
			\item[(iii)] If the kernel $p_{\omega}(x,y)$ is $C^{q}$-integrally strictly positive definite for every $w \in \Omega$, then the operator valued kernel $P$ is $C^{q}$ integrally strictly positive definite if and only if all the scalar valued kernels $P_{v}$ are $C^{q}$ integrally strictly positive definite.
		\end{enumerate}
		Additionally, the equivalences in $(ii)$ and $(iii)$ are also equivalent at all scalar valued measures $\Lambda_{v}$ being nonzero.
	\end{thm}
	
	\subsection{\textbf{Kernels on Euclidean spaces}\label{Kernels on Euclidean spaces}}

	In this subsection, we obtain several consequences of Theorem \ref{Cpuniversalityintegralkernel} for several trace class valued families of positive definite kernels on Euclidean spaces, with an emphasis on radial kernels. 
	
	First, we obtain a version of Corollary \ref{compcsuppor} to the differentiable setting.
	
	\begin{cor}\label{Cpuniversalityintegralkernelrestricmeasure} Let $f  \in C^{2q}(\mathbb{R}^{m})$ be a function for which the kernel
		$$	
		(x,y) \in \mathbb{R}^{m}\times \mathbb{R}^{m} \to k(x,y):=f(x-y) \in \mathbb{C}	
		$$
		is positive definite. Given a Radon nonnegative finite operator $\mathcal{L}(\mathcal{H})$ valued measure $\Lambda: \mathscr{B}((0,\infty)) \to \mathcal{L}(\mathcal{H})$  such that $\Lambda((0,\infty))$ is a trace class operator, consider the operator valued kernel
		$$	
		P : \mathbb{R}^{m}\times \mathbb{R}^{m} \to \mathcal{L}(\mathcal{H}), \quad P(x,y)= \int_{(0,\infty)}f(w(x-y)) d\Lambda(w),	
		$$
		and for every $v \in \mathcal{H}\setminus{\{0\}}$ the scalar valued kernels  and the scalar valued nonnegative finite Radon measures
		$$
		P_{v} : \mathbb{R}^{m}\times \mathbb{R}^{m} \to \mathbb{C}, \quad P_{v}(x,y)= \langle P(x,y)v, v \rangle_{\mathcal{H}} 
		$$ 
		$$
		\Lambda_{v}: \mathscr{B}((0,\infty)) \to \mathbb{C}, \quad    
		\Lambda_{v}(A)= \langle \Lambda(A)v, v \rangle_{\mathcal{H}}.  	
		$$
		If $(Tr P) \in C^{2q}(\mathbb{R}^{m})$ then
		\begin{enumerate}
			\item[(i)] The kernel $P \in C^{q,q}(\mathbb{R}^{m}\times \mathbb{R}^{m}, \mathcal{L}(\mathcal{H}))$ and $\mathcal{H}_{P} \subset C^{q}(\mathbb{R}^{m}, \mathcal{H})$.
			\item[(iii)] If the kernel $k$ is $C^{q}$-universal then the operator valued kernel $P$ is $C^{q}$-universal universal if and only if  all scalar valued kernels $P_{v}$ are $C^{q}$-universal. 
			\item[(iii)] If the kernel $k$ is $C^{q}$-integrally strictly positive definite, then the operator valued kernel $P$ is $C^{q}$-integrally strictly positive definite if and only if  all scalar valued  kernels $P_{v}$ are $C^{q}$-integrally strictly positive definite. 
		\end{enumerate}
		Additionally, the equivalences in $(ii)$ and $(iii)$ are also equivalent at all scalar valued measures $\Lambda_{v}$ being nonzero.
	\end{cor}

	Again, we emphasize that Corollary \ref{Cpuniversalityintegralkernelrestricmeasure} is also valid if $\Lambda$ is a measure of bounded variation that admits a Radon-Nikod\'ym decomposition $d\Lambda=Gd\lambda$ and $\int w^{2q}\|G(w)\|d\lambda< \infty$. On the proof of the Corollary   \ref{Cpuniversalityintegralkernelrestricmeasure} this property is obtained from the hypothesis that $(Tr P) \in C^{2q}(\mathbb{R}^{m})$.

	We finish this subsection with three examples of families of operator valued positive definite radial kernels  where the differentiable universalities are preserved by the scalar valued projections of the kernel as a consequence of Corollary \ref{Cpuniversalityintegralkernelrestricmeasure}. The first one is a continuation of Theorem \ref{universaleuclidian2}.
	
	\begin{thm}\label{universaleuclidian2cq}Let $q \in \mathbb{Z}_{+}$,  $F: [0, \infty) \to \mathcal{L}(\mathcal{H})$ be an ultraweakly continuous function for which $F(0)$ is a trace class operator, $(Tr F) \in C^{2q}([0,\infty))$ and the kernels 	
		$$
		(x,y) \in 	\mathbb{R}^{m} \times \mathbb{R}^{m} \to  \langle F(\|x-y\|)v ,v \rangle_{\mathcal{H}} \in   \mathbb{C},
		$$
		are positive definite for every $m \in \mathbb{N}$, $v \in \mathcal{H}$. Then the kernel
		$$
		(x,y) \in \mathbb{R}^{m} \times \mathbb{R}^{m} \to K_{F}(x,y):= F(\|x-y\|) \in \mathcal{L}(\mathcal{H})
		$$ 
		is positive definite, $\mathcal{H}_{K_{F}} \subset C^{q}(\mathbb{R}^{m}, \mathcal{H})$ and the following are equivalent
		\begin{enumerate}
			\item[(i)] The kernel $K_{F}$ is strictly positive definite. 
			\item[(ii)] The kernel $K_{F}$ is $C^{q}$-universal.
			\item[(iii)] For every $v \in \mathcal{H}\setminus\{{0}\}$ the function $t \in [0, \infty) \to \langle F(t)v, v\rangle_{\mathcal{H}} $ is non constant.
		\end{enumerate}
		Moreover, $\mathcal{H}_{K_{F}} \subset C_{0}(\mathbb{R}^{m}, \mathcal{H})$ if and only if for every $v \in \mathcal{H}\setminus\{{0}\}$ the function $t \in [0, \infty) \to \langle F(t)v, v\rangle_{\mathcal{H}} \in C_{0}([0, \infty))$. Under this additional hypothesys, the following are equivalent
		\begin{enumerate}
			\item[(i)] The kernel $K_{F}$ is strictly positive definite. 
			\item[(ii)] The kernel $K_{F}$ is $C_{0}^{q}$-universal.
			\item[(iii)] For every $v \in \mathcal{H}\setminus\{{0}\}$ the function $t \in [0, \infty) \to \langle F(t)v, v\rangle_{\mathcal{H}} $ is nonzero.
		\end{enumerate}	
	\end{thm}
	
	The second is the operator valued generalization of the Askey class, a  continuation of Theorem \ref{universaleuclidianaskey}.
	
	\begin{thm}\label{universaleuclidian2cqaskey}Let  $F: [0, \infty) \to \mathcal{L}(\mathcal{H})$ be an ultraweakly continuous function for which $F(0)$ is a trace class operator, $(Tr F) \in C^{(2q)}([0, \infty))$, $F_{v}$ is $\ell$-times completely monotone for every $v \in \mathcal{H}$ and $q \leq (\ell -2)/2$. Then the kernel 
		$$
		(x,y) \in 	\mathbb{R}^{m} \times \mathbb{R}^{m} \to  K_{F}(x,y):=F(\|x-y\|) \in   \mathcal{L}(\mathcal{H}),
		$$
		is positive definite for  $m \leq 2\ell -3 $, $\mathcal{H}_{K_{F}} \subset C^{q}(\mathbb{R}^{m}, \mathcal{H})$ and the following are equivalent
		\begin{enumerate}
			\item[(i)] The kernel $K_{F}$ is strictly positive definite. 
			\item[(ii)] The kernel $K_{F}$ is $C^{q}$-universal.
			\item[(iii)] For every $v \in \mathcal{H}\setminus\{{0}\}$ the function $t \in [0, \infty) \to \langle F(t)v, v\rangle_{\mathcal{H}} $ is non constant.
		\end{enumerate}
		Moreover, $\mathcal{H}_{K_{F}} \subset C_{0}(\mathbb{R}^{m}, \mathcal{H})$ if and only if for every $v \in \mathcal{H}\setminus\{{0}\}$ the function $t \in [0, \infty) \to \langle F(t)v, v\rangle_{\mathcal{H}} \in C_{0}([0, \infty))$. Under this additional hypothesys, the following are equivalent
		\begin{enumerate}
			\item[(i)] The kernel $K_{F}$ is strictly positive definite. 
			\item[(ii)] The kernel $K_{F}$ is $C_{0}^{q}$-universal.
			\item[(iii)] For every $v \in \mathcal{H}\setminus\{{0}\}$ the function $t \in [0, \infty) \to \langle F(t)v, v\rangle_{\mathcal{H}} $ is nonzero.
		\end{enumerate}	
	\end{thm}
	
	And the last result  is a differentiable version of Theorem \ref{universaleuclidian3}.
	
	\begin{thm}\label{universaleuclidian32}Let $m\geq2$, $F: [0, \infty) \to \mathcal{L}(\mathcal{H})$ be an ultraweakly continuous function for which $F(0)$ is a trace class operator, $(Tr F) \in C^{(2q)}([0, \infty))$, and  the kernel $(x,y) \in 	\mathbb{R}^{m} \times \mathbb{R}^{m} \to  \langle F(\|x-y\|)v , v\rangle_{\mathcal{H}}$ is positive definite for every $v \in \mathcal{H}$. Consider the kernel 
		$$
		(x,y) \in 	\mathbb{R}^{m-1} \times \mathbb{R}^{m-1} \to  P(x,y):=F(\|x-y\|) \in   \mathcal{L}(\mathcal{H}),
		$$  
		then $P$ is positive definite, $\mathcal{H}_{P}\subset C^{q}(\mathbb{R}^{m}, \mathcal{H})$ and  the following statements are equivalent
		\begin{enumerate}
			\item[(i)] The kernel $P$ is strictly positive definite. 
			\item[(ii)] The kernel $P$ is $C^{q}$-universal.
			\item[(iii)] For every $v \in \mathcal{H}\setminus\{{0}\}$ the function $t \in [0, \infty) \to \langle F(t)v, v\rangle_{\mathcal{H}} $ is non constant.
		\end{enumerate}
	\end{thm}

	%Before, we presented examples for the fact that on an operator valued positive definite radial kernel, the fact that all scalar valued projections of the kernel are universal kernels does not imply that the operator valued kernel is universal. The next result tell us that this is not true for some radial kernel that are positive definite on all Euclidean spaces.  

	\section{\textbf{Proofs}}

	\subsection{\textbf{Section \ref{Dual spaces and equivalent conditions for universality}}}    	   
	\begin{proof}[\textbf{Proof of Theorem \ref{oneintegraltodoubleintegral} }]We focus all the arguments on the $C_{0}$ case, being the other similar. Note that for every $\eta \in \mathfrak{M}(Z, \mathcal{H})$ the linear functional $$F \in \mathcal{H}_{K} \to \int _{Z}\langle F(x), d\eta(x) \rangle \in \mathbb{C}$$ is continuous. Indeed, it is the composition of the  continuous operators  $ I: \mathcal{H}_{K} \to C_{0}(Z, \mathcal{H}) $ (Proposition \ref{rkhscontained}) with the integral of a measure in $\mathfrak{M}(Z, \mathcal{H})$(Theorem \ref{Dinculeanu-Singer}). But then, the Riesz Representation Theorem for Hilbert spaces implies that there exists a function $K_{\eta} \in \mathcal{H}_{K}$ for which $\int _{Z}\langle F(x), d\eta(x) \rangle = \langle F,K_{\eta}  \rangle_{\mathcal{H}_{K}}$ for all $F \in \mathcal{H}_{K}$. In particular 
		$$
		\langle v, K_{\eta}(y)\rangle_{\mathcal{H}}=\langle K_{y}v, K_{\eta} \rangle_{\mathcal{H}_{K}} = \int_{Z}\langle K(x,y)v,d\eta(x)\rangle.
		$$
		The weak-Bochner integral characterization  $K_{\eta}(y)= \int_{Z}K(x,y)d\eta(x)$ holds because  the linear functional $ v \in \mathcal{H} \to \int_{Z}\langle K(x,y)v , d\eta(x) \rangle \in \mathbb{C} $ is continuous. This proves $(i)$.\\
		As for the proof of $(ii)$,  by Lemma \ref{manipulation} we have that  
		\begin{align*}
		\langle K_{\eta}, K_{\eta}\rangle_{\mathcal{H}_{K}}&= \int_{Z} \langle K_{\eta}(x), d\eta(x) \rangle=\int_{Z} \langle K_{\eta}(x), H(x) \rangle_{\mathcal{H}}d|\eta|(x) \\ 
		&=\int_{Z} \overline{\langle H(x), K_{\eta}(x) \rangle_{\mathcal{H}}}d|\eta|(x) = \int_{Z} \overline{ \left [ \int_{Z}\langle K(y,x)H(x),d\eta(y) \rangle \right ]} d|\eta|(x)\\
		&=\int_{Z} \overline{ \left [ \int_{Z}\langle K(y,x)H(x),H(y) \rangle_{\mathcal{H}}d|\eta|(y) \right ]} d|\eta|(x)\\
		&= \int_{Z}  \int_{Z}\langle H(y) ,K(y,x)H(x)\rangle_{\mathcal{H}}d|\eta|(y)  d|\eta|(x)\\
		&=\int_{Z} \int_{Z} \langle   K(x,y)H(y),H(x) \rangle_{\mathcal{H}} d|\eta|(x)  d|\eta|(y)
		\end{align*}
		Where $d\eta= Hd|\eta|$. Since $H$ is Bochner integrable and $x \in Z \to \|K(x,x)\|_{\mathcal{H}} \in \mathbb{C}$ is a bounded function, it is possible to reverse the order of integration by Fubinni-Tonelli.\\
		Finally $(iii)$ holds true because  $\mathcal{H}_{K}$ is a reproducing kernel Hilbert space, so $\langle K_{\eta}, K_{\eta}\rangle_{\mathcal{H}_{K}}=0 $ if and only if $K_{\eta}(y)=0$ for all $y \in Z$,  Theorem $11$ of \cite{Caponnetto2008} for the universal case and Equation \ref{c0univform2} for the $C_{0}$-universal case  concludes the proof.\end{proof}
	
	\begin{proof}[\textbf{Proof of Lemma \ref{compacinclusion}}] If the kernel is continuous  and $B \subset \mathcal{H}_{K}$ is a bounded set, then the equicontinuity follows from the inequality.
		$$
		\|F(x) - F(y)\|_{\mathcal{H}}\leq \|F\|_{\mathcal{H}_{K}}(\|K(x,x) -K(x,y) -K(y,x) +K(y,y)  \|_{\mathcal{L}(\mathcal{H})})^{1/2}.
		$$
		Conversely, fix arbitrary $x,y \in Z$. By Proposition \ref{rkhscontained}, there are open sets $U$ and $V$ of $Z$, that contains $x$ and $y$ respectively, for which $\sup_{z\in U \cup V}\|K(z,z)\|_{\mathcal{L}(\mathcal{H})}<\infty$. Then, the set $B:=\{ K_{z}v, \quad  z\in U \cup V, \|v\|_{\mathcal{H}}=1 \} \subset \mathcal{H}_{K}$  is bounded and by the hypothesis  is equicontinuous. In particular, for every $\epsilon >0$ there exists open sets $U_{x}$ and $U_{y}$ of $Z$, that contains $x$ and $y$ respectively, for which
		$$
		\|K(w,z)v- K(x,z)v \|_{\mathcal{H}}=\|[K_{z}v](w) - [K_{z}v](x)  \|_{\mathcal{H}}< \epsilon,\quad  w \in U_{x}, z\in U \cup V, \|v\|_{\mathcal{H}}=1  
		$$   	
		$$
		\|K(w,z)v- K(y,z)v \|_{\mathcal{H}}=\|[K_{z}v](w) - [K_{z}v](y)  \|_{\mathcal{H}}< \epsilon,\quad  w \in U_{y}, z\in U \cup V, \|v\|_{\mathcal{H}}=1  
		$$ 	
		Then
		$$
		\|K(x,y) - K(x^{\prime}, y^{\prime})\|_{\mathcal{L}(\mathcal{H})} \leq \|K(x,y) - K(x^{\prime}, y)\|_{\mathcal{L}(\mathcal{H})} + \|K(y,x^{\prime}) - K(y^{\prime}, x^{\prime})\|_{\mathcal{L}(\mathcal{H})} < 2\epsilon, 	
		$$	
		for all $x^{\prime} \in U \cap U_{x}$ and $y \in V \cap U_{y}$, which proves relation $(i)$. As for the proof of $(ii)$, note that
		$$
		|\langle e_{\mu}, F(y) \rangle_{\mathcal{H}} | \leq \|F\|_{\mathcal{H}_{K}} \sqrt{\langle K(y,y)e_{\mu}, e_{\mu} \rangle_{\mathcal{H}}}, 
		$$
		since $\sum_{\mu \in \mathcal{I}} (\sqrt{\langle K(y,y)e_{\mu}, e_{\mu} \rangle_{\mathcal{H}}})^{2}= Tr(K(y,y))< \infty$, the set $\{F(y), F \in B   \} \subset \mathcal{H}$ has compact closure on the norm topology of $\mathcal{H}$.\\
		Now we prove $(iii)$. By definition, the inclusion-restriction is  a compact operator if and only if for every bounded set $B \subset \mathcal{H}_{K}$ (restricted to $\mathcal{A}$) has compact closure on  $C(\mathcal{A}, \mathcal{H})$. By the Arzel\`{a}-Ascoli Theorem this occurs if and only if the set $B$(restricted to $\mathcal{A}$) is equicontinuous and pointwise relatively compact, the conclusion follows from $(i)$, $(ii)$ and the locally compact assumption on $Z$. The second part of $(iii)$ is a consequence from the fact  that  closed sets on $\mathcal{H}_{K}$(restricted to $\mathcal{A}$) are closed on $C(\mathcal{A}, \mathcal{H})$.\\ 
		The proof of $(iv)$ is similar to the previous ones,  we just emphasize that the importance of the compact set $\mathcal{A}$  is to ensure that $B_{\mathcal{A}}$ is equicontinuous at the infinity point of $Z$. \end{proof}

	\begin{proof}[\textbf{Proof of Example \ref{examplenonuniversalityradial}}]
		Indeed, let $\phi_{1}, \phi_{2} \in C^{\infty}_{c}(\mathbb{R}^{m})$ be two non zero radial functions that are linearly independent. Define the kernel $K: \mathbb{R}^{m}\times \mathbb{R}^{m}  \to M_{2}(\mathbb{C})$ by
		$$
		K(x,y) := \int_{\mathbb{R}^{m}}e^{-i(x-y)\cdot\xi}\begin{pmatrix}|\widehat{\phi_{2}}(\xi)|^{2} & -\widehat{\phi_{2}}(\xi)\widehat{\overline{\phi_{1}}}(\xi)\\ -\widehat{\phi_{1}}(\xi)\widehat{\overline{\phi_{2}}}(\xi) & |\widehat{\phi_{1}}(\xi)|^{2} \end{pmatrix} d\xi.
		$$ 	
		The kernel $K$ is well defined, continuous, it is  radial because $\phi_{1}, \phi_{2}$ are radial functions and is positive definite because the matrix 
		$$
		\begin{pmatrix}|\widehat{\phi_{2}}(\xi)|^{2} & -\widehat{\phi_{2}}(\xi)\widehat{\overline{\phi_{1}}}(\xi)\\ -\widehat{\phi_{1}}(\xi)\widehat{\overline{\phi_{2}}}(\xi) & |\widehat{\phi_{1}}(\xi)|^{2} \end{pmatrix}
		$$
		is positive semidefinite for every $\xi \in \mathbb{R}^{m}$. For every $v \in \mathbb{C}^{2}\setminus{\{0\}}$ the kernel $K_{v}$ is $C_{0}^{\infty}$-universal because   
		$$
		K_{v}(x-y)= \int_{\mathbb{R}^{m}}e^{-i(x-y)\cdot\xi}|v_{1}\widehat{\phi_{2}}(\xi) -v_{2}\widehat{\phi_{1}}(\xi) |^{2}d\xi, 
		$$ 
		being the $C^{\infty}_{0}$-universality  a consequence of Theorem $17$ of \cite{Simon-Gabriel2018} since the scalar valued nonnegative finite Radon measure   $|v_{1}\widehat{\phi_{2}}(\xi) -v_{2}\widehat{\phi_{1}}(\xi) |^{2}d\xi$ has positive measure on any open set of $\mathbb{R}^{m}$ and all of its moments are finite. Finally, the matrix valued kernel is not universal, because the scalar valued nonnegative finite Radon measures $d\eta_{1}(\xi):= \phi_{1}(-\xi)d\xi$ and $d\eta_{2}(\xi):= \phi_{2}(-\xi)d\xi$ have compact support and if $\eta= (\eta_{1}, \eta_{2})$
		\begin{align*}  
		\int_{\mathbb{R}^{m}}\langle \int_{\mathbb{R}^{m}}&K(x,y)d\eta(x),d\eta(y) \rangle = \sum_{\mu, \nu =1}^{2} \int_{\mathbb{R}^{m}}\int_{\mathbb{R}^{m}}F_{\mu, \nu}(x-y)d\eta_{\mu}(y)\overline{d\eta_{\nu}}(x)\\
		&\int_{\mathbb{R}^{m}}  \left (  |\widehat{\phi_{2}}|^{2}|\widehat{\phi_{1}}|^{2} -  \widehat{\phi_{2}}\widehat{\overline{\phi_{1}}}\widehat{\phi_{1}}\widehat{\overline{\phi_{2}}} - \widehat{\phi_{2}}\widehat{\overline{\phi_{1}}}\widehat{\phi_{1}}\widehat{\overline{\phi_{2}}} +|\widehat{\phi_{1}}|^{2}|\widehat{\phi_{2}}|^{2}                         \right )d\xi =0
		\end{align*}\end{proof}

	\begin{proof}[\textbf{Proof of Example \ref{exampleconverseuniversality}}] Since for every $x \in \mathbb{R}^m$
		$$
		e^{-\|x\|^{2}}:=\int_{\mathbb{R}^{m}}e^{-ix\cdot\xi}\frac{1}{2^{m}\pi^{m/2}}e^{-\|\xi\|^{2}/4}d\xi, 	
		$$	
		we have that 
		$$
		K(x,y) = \int_{\mathbb{R}^{m}}e^{-i(x-y)\cdot\xi}\left [\frac{1}{2^{m}\pi^{m/2}}e^{-\|\xi\|^{2}/4} \begin{pmatrix} 1 & e^{-iw\cdot\xi}\\ e^{iw\cdot\xi} & 1  \\\end{pmatrix}	 d\xi\right ].
		$$	
		The function
		$$
		\xi \in \mathbb{R}^{m} \to G(\xi):=\frac{1}{2^{m}\pi^{m/2}}e^{-\|\xi\|^{2}/4} \begin{pmatrix} 1 & e^{-i2w\cdot\xi}\\ e^{i2w\xi} & 1  \end{pmatrix} \in M_{2}(\mathbb{C})
		$$
		is continuous and positive semidefinite on every point. By relation $(i)$ at Theorem \ref{universalityintegralkernel}, the kernel $K$ is positive definite. Choosing the points $x_{1}=0$, $x_{2} = 2w$, the matrix
		$$
		[K(x_{\mu}, x_{\nu})]_{\mu, \nu =1}^{2}= \begin{pmatrix} 1 & e^{-\|2w\|^{2}} & e^{-\|2w\|^{2}} & 1\\ e^{-\|2w\|^{2}} & 1 & e^{-\|4w\|^{2}}& e^{-\|2w\|^{2}} \\ e^{-\|2w\|^{2}} & e^{-\|4w\|^{2}} & 1 & e^{-\|2w\|^{2}} \\ 1 &e^{-\|2w\|^{2}} & e^{-\|2w\|^{2}} & 1 \end{pmatrix}  \in M_{4}(\mathbb{C})
		$$
		is not positive definite, which implies that the kernel $K$ is not strictly positive definite, hence it is not universal. The set $\mathcal{H}_{K}$ is a subset of $C_{0}(\mathbb{R}^{m},\mathbb{C}^{2} )$, because it satisfies the conditions at Proposition \ref{rkhscontained}. But for every $v \in \mathbb{C}^{2}\setminus{\{0\}}$
		$$
		K_{v}(x -y) =  \int_{\mathbb{R}^{m}}e^{-i(x-y)\cdot\xi} \frac{1}{2^{m}\pi^{m/2}}e^{-\|\xi\|^{2}/4} |v_{1}e^{-iw\cdot\xi} + v_{2}e^{iw\cdot\xi} |^{2} d\xi,
		$$
		note that  the complex valued nonnegative finite measure $e^{-\|\xi\|^{2}/4} |v_{1}e^{-iw\cdot\xi} + v_{2}e^{iw\cdot\xi} |^{2} d\xi$ has positive measure on any open subset of $\mathbb{R}^{m}$ (by Lemma $6.7$ of \cite{wend}) and all of its moments are finite, which by Theorem $17$ of \cite{Simon-Gabriel2018}, the kernel $K_{v}$ is $C^{\infty}_{0}$ universal. 
	\end{proof}

	Below we prove a technical result in order to simplify  some arguments on the proof of Theorem \ref{universalityintegralkernel}

	\begin{lem}\label{almostaschurproduct} Let $X$  be a locally compact Hausdorff space and $p : X \times X\to \mathbb{C}$ be a continuous and bounded positive definite kernel. Given a scalar valued nonnegative finite Radon measure $\lambda$ and a Borel Bochner measurable function $H: X\to \mathcal{H}$ that is integrable with respect to $\lambda$ on $X$. Then, for  every positive semidefinite operator $T \in \mathcal{L}(\mathcal{H})$, we have that
		$$	
		\int_{X}\int_{X} p(x,y)\langle TH(y), H(x) \rangle_{\mathcal{H}}  d\lambda(x)d\lambda(y)\geq 0		
		$$	
	\end{lem}
	
	\begin{proof}	First, note that this double integral exists because 
		$$|p(x,y)\langle TH(y), H(x) \rangle_{\mathcal{H}} | \leq M \|T\|_{\mathcal{L}(\mathcal{H})}\|H(x)\|_{\mathcal{H}}\|H(y)\|_{\mathcal{H}}. $$ 
		Let $T^{1/2}$ be the positive semidefinite square root of the operator $T$, then by adding coordinates
		\begin{align*}
		\int_{X}\int_{X}& p(x,y)\langle TH(y), H(x) \rangle_{\mathcal{H}}  d\lambda(x)d\lambda(y)(y)\\
		& = \int_{X}\int_{X} p(x,y)\langle T^{1/2}H(y), T^{1/2}H(x) \rangle_{\mathcal{H}}  d\lambda(x)d\lambda(y)(y)\\
		& = \int_{X}\int_{X} p(x,y) \sum_{\mu \in \mathcal{I}} (T^{1/2}H(y))_{\mu}  \overline{(T^{1/2}H(y))_{\mu}}   d\lambda(x)d\lambda(y)(y)\\
		& = \sum_{\mu \in \mathcal{I}}\int_{X}\int_{X} p(x,y)\overline{\lambda_{\mu}}(x)d\lambda_{\mu}(y)\geq 0
		\end{align*}
		where $d\lambda_{\mu}(y)= (T^{1/2}H(y))_{\mu}d\lambda(y)$ is a finite scalar valued Radon measure on $X$ and the last inequality follows because for every $\mu \in \mathcal{I}$ it holds that $\langle K_{\lambda_{\mu}}, K_{\lambda_{\mu}}\rangle_{\mathcal{H}_{p}}\geq 0 $ by the scalar valued version of relation $(ii)$ at Theorem \ref{oneintegraltodoubleintegral}\end{proof}

	Before proving Theorem \ref{universalityintegralkernel}, we need a scalar valued version of it.
	
	\begin{thm}\label{universalityintegralkernelscalar} Let $X$ and $\Omega$  be Hausdorff spaces,  with $\Omega$ being locally compact and \\ $p : \Omega \times X \times X \to \mathbb{C}$ be a bounded continuous function such that the kernel 
		$$
		(x,y) \in X\times X \to p_{\omega}(x,y):=p(\omega, x,y)
		$$ 
		is positive definite for every $w \in \Omega$. Given a scalar valued nonnegative finite Radon measure $\lambda$ on $\Omega$, consider the  kernel
		$$	
		P : X\times X \to \mathbb{C}, \quad P(x,y)= \int_{\Omega}p_{w}(x,y) d\lambda(w). 	  	
		$$
		Then:
		\begin{enumerate}
			\item[(i)] The  kernel  $P$ is positive definite.
			\item[(ii)] The kernel $P$ is continuous, and in particular  $\mathcal{H}_{K} \subset C(X)$.
			\item[(iii)] If the kernel  $p_{\omega}$ is strictly positive definite for every $w \in \Omega$, then the  kernel $P$ is strictly positive definite if and only if the measure $\lambda$ is nonzero. 
			\item[(iv)] If the kernel  $p_{\omega}$ is universal for every $w \in \Omega$, then the  kernel $P$ is universal if and only if the measure $\lambda$ is nonzero.
			\item[(v)] If the kernel  $p_{\omega}$ is integrally strictly positive definite for every $w \in \Omega$, then the  kernel $P$ is integrally strictly positive definite if and only if the measure $\lambda$ is nonzero. 
		\end{enumerate}
	\end{thm}
	\begin{proof}
		Relation $(i)$ is immediate and relation $(ii)$ can be proved by the same arguments as the proof of relation $(ii)$ in Theorem \ref{universalityintegralkernel}.\\
		Relation $(iii)$ is valid because if $x_{1}, \ldots , x_{n}$ are distinct points in $X$ and $c_{1}, \ldots , c_{n} \in \mathbb{C}$, but not all of them are null, we have that for every $w \in \Omega$, $ \sum_{i,j=1}^{n}c_{i}\overline{c_{j}}p_{w}(x_{i}, x_{j})>0$, and this is a continuous function on the variable $w$. But then, if 
		$$
		0=\sum_{i,j=1}^{n}c_{i}\overline{c_{j}}P(x_{i}, x_{j})= \int_{\Omega}\sum_{i,j=1}^{n}c_{i}\overline{c_{j}}p_{w}(x_{i}, x_{j})d\lambda(w)
		$$
		then $\lambda$ is the zero measure.\\
		Now we prove $(iv)$. Let $\eta \in \mathfrak{M}(X)$ be a scalar valued nonnegative finite Radon measure of compact support. Then
		\begin{align*}
		0&=\int_{X}\int_{X}P(x,y)d\eta(y)d\overline{\eta}(x)=\int_{X}\int_{X}\int_{\Omega}p_{w}(x,y)d\lambda(w)d\eta(y)d\overline{\eta}(x)\\
		&=\int_{\Omega}\int_{X}\int_{X}p_{w}(x,y)d\eta(y)d\overline{\eta}(x)d\lambda(w).
		\end{align*}
		The integral and the change of order are possible, because $\eta$ has compact support and \\ $\int_{\Omega}\int_{X}\int_{X}|p_{w}(x,y)|d|\eta|(x)d|\eta|(y)d\lambda(w)<\infty$.  But since $p_{w}$ is an universal kernel for all $w \in \Omega$, if  $\eta$ is non zero, we have that
		$$	
		\int_{X}\int_{X} p_{w}(x,y)d\eta(y)d\overline{\eta}(x)> 0, \quad w \in \Omega		
		$$
		but then, the triple integral is zero if and only if the measure $\lambda$ is zero, which proves our claim.\\
		The proof of $(v)$ is identical to the proof of $(iv)$.
	\end{proof}
	
	On the following result, we  prove that if an operator valued positive definite kernel has a property, then all of its scalar valued projections also have the same property.

	\begin{lem}\label{scalarvaluedprojectionsuniversality}Let $X$ be a Hausdorff space and $K: X \times X \to \mathcal{L}(\mathcal{H})$ be an operator valued strictly positive definite (universal) kernel. Then for every $v  \in \mathcal{H}\setminus{\{0\}}$ the scalar valued kernel
		$$
		(x,y) \in X \times X \to K_{v}(x,y):=\langle K(x,y)v,v \rangle_{\mathcal{H}} \in \mathbb{C} 
		$$	
		is strictly positive definite (universal). In addition, if $X$ is locally compact and the operator valued kernel $K$ is $C_{0}$-universal (integrally strictly positive definite) then the complex valued kernels $K_{v}$ are $C_{0}$-universal (integrally strictly positive definite).  
	\end{lem}

	\begin{proof}
		We focus the proof on the universality case, being the other arguments similar.  Being the kernel $K$ universal, if $\eta \in \mathfrak{M}(X, \mathcal{H})$ has compact support and
		$$
		\int_{X}\langle \int_{X}K(x,y)d\eta(x),d\eta(y)\rangle=0 \in \mathbb{C} 
		$$	
		then $\eta$ is the zero measure. Then, given a nonzero vector $v \in \mathcal{H}$, if  $\lambda \in \mathfrak{M}(X)$ has compact support and satisfies
		$$
		\int_{X}\int_{X}K_{v}(x,y)d\lambda(y)d\overline{\lambda}(x)=0,
		$$
		define the $\mathcal{H}$ valued Radon measure of bounded variation $d\eta= vd\lambda \in \mathfrak{M}(X, \mathcal{H})$, and  by Lemma \ref{manipulation}
		$$
		0=\int_{X}\int_{X}K_{v}(x,y)d\lambda(y)d\overline{\lambda}(x)=\int_{X}\langle K(x,y)v,v \rangle_{\mathcal{H}}d\lambda(y)d\overline{\lambda}(x)= \int_{X}\langle \int_{X}K(x,y)d\eta(x),d\eta(y) \rangle.
		$$	
		This can only occur if $\lambda$ is the zero measure, then the kernel $K_{v}$ is universal.  
	\end{proof}

	\begin{proof}[\textbf{Proof of Theorem \ref{universalityintegralkernel}}]  \label{ProofofTheoremuniversalityintegralkernel}
		In order to simplify some expressions, we suppose that the function $p$ is bounded by $1$. We skip the proof of $(i)$ because the argument is similar to the one we use in $(iii)$.
		As for $(ii)$, we focus on the continuity, being the boundedness an easier and similar argument. We have that
		\begin{align*}
		|\langle [P(x,y)-P(x^{\prime}, y^{\prime})]u, v \rangle_{\mathcal{H}}|&= |\langle \int_{\Omega}[p_{w}(x,y) - p_{w}(x^{\prime},y^{\prime}) ]  \langle G(w)u, v  \rangle_{\mathcal{H}} d\lambda(w) | \\
		&\leq  \|u\|_{\mathcal{H}}\|v\|_{\mathcal{H}} \int_{\Omega}|p_{w}(x,y) - p_{w}(x^{\prime},y^{\prime})|\|G(w)\|_{\mathcal{L}(\mathcal{H})}  d\lambda(w)
		\end{align*}
		Since $\Lambda$ is a Radon measure on $\Omega$, for every $\epsilon >0$ there exists a compact set $\mathcal{C} \subset \Omega$ for which $|\Lambda|(\Omega - \mathcal{C}  )< \epsilon$.  On the other hand, since $\mathcal{C}\times \{x\}\times \{y\}$ is a compact set and $p$ is continuous, there exists open neighborhoods of $x$ and of $y$ for which $|p_{w}(x,y) - p_{w}(x^{\prime},y^{\prime})| < \epsilon$ for all  $x^{\prime }$, $y^{\prime}$ on the open neighborhoods of $x$ and $y$  respectively and $ w \in \mathcal{C}$. Gathering these information, we obtain that  
		$$
		\|P(x,y) - P(x^{\prime}, y^{\prime})\|_{\mathcal{L}(\mathcal{H})} <  \epsilon(|\Lambda|(\mathcal{C}) + 2) 
		$$ 
		which proves the continuity on the operator norm. The fact that $\mathcal{H}_{K} \subset C(X, \mathcal{H})$ is a direct consequence of the continuity on the operator norm and Proposition \ref{rkhscontained}.\\ 
		As for $(iii)$, a measure $\Lambda_{v}$ is nonzero if and only if $P_{v}$ is strictly positive definite by Theorem \ref{universalityintegralkernelscalar}. Now, suppose that all measures $\Lambda_{v}$ are nonzero and  let $x_{1}, \ldots , x_{n}$ be distinct points in $X$ and $v_{1}, \ldots , v_{n} \in \mathcal{H}$ such that
		$$
		0=\sum_{i,j=1}^{n} \langle P(x_{i},x_{j})v_{i}, v_{j} \rangle_{\mathcal{H}} = \sum_{i,j=1}^{n}  \int_{\Omega}p_{w}(x_{i}, x_{j})  \langle G(w) v_{i}, v_{j} \rangle_{\mathcal{H}} d\lambda(w) $$
		But since
		$$
		\sum_{i,j=1}^{n}  p_{w}(x_{i}, x_{j}) \langle G(w) v_{i}, v_{j} \rangle_{\mathcal{H}} \geq 0, \quad w \in \Omega
		$$
		and the measure $\lambda$ is nonnegative,  this double sum must be equal to $0$ almost everywhere on $\lambda$. The matrix $[ \langle G(w) v_{i}, v_{j} \rangle_{\mathcal{H}}]_{i,j=1}^{n}$ is positive semidefinite by Lemma \ref{simplifyoperatornonegativemeasure}. After using an argument involving the Gram representation of this matrix and the fact that the kernel $p_{w}$ is  strictly positive definite  for every $w \in \Omega$, we obtain that
		$$   
		\langle G(w) v_{i}, v_{i} \rangle_{\mathcal{H}} =0, \quad 1\leq i \leq n 
		$$
		almost everywhere on $\lambda$, and then $   
		\langle \Lambda(A)v_{i}, v_{i}\rangle_{\mathcal{H}}= \int_{A} \langle G(w) v_{i}, v_{i} \rangle_{\mathcal{H}} d\lambda(w)=0
		$ for every Borel measurable set $A$, so we must have that all vectors $v_{i}$ are zero and  the kernel is strictly positive definite. The remaining implication is a consequence of Lemma \ref{scalarvaluedprojectionsuniversality}.\\
		The proof of $(iv)$ is a bit longer. A measure $\Lambda_{v}$ is nonzero if and only if $P_{v}$ is universal by Theorem \ref{universalityintegralkernelscalar}. Now  assume that all measures $\Lambda_{v}$ are nonzero. Let $\eta \in \mathfrak{M}(X, \mathcal{H})$ be a finite Radon measure with bounded variation and compact support such that  
		$$
		\int_{X}\langle  \int_{X} P(x,y) d\eta(x), d\eta(y)\rangle =0 \in \mathbb{C} 
		$$  
		Since every Hilbert valued measure of bounded variation admits a Radon-Nikod\'ym decomposition, (explained at the comments after Definition \ref{Radon Nykodin Property}), if a measure $\eta$  satisfy this requirements, there exists a Bochner measurable function $H: X\to \mathcal{H}$, Bochner integrable with respect to $|\eta|$ such that $\eta = H d|\eta |$. By Lemma \ref{manipulation} and the arguments presented on the proof of Theorem \ref{oneintegraltodoubleintegral} we have that  
		$$
		0=\int_{\Omega}\int_{X} \int_{X}p_{w}(x,y)\langle G(w)H(y), H(x)\rangle_{\mathcal{H}}d|\eta|(x)  d|\eta|(y)d\lambda(w)
		$$
		The kernel defined by $p_{w}$ is positive  definite and bounded for every $w \in \Omega$ and the function $H$ is Bochner integrable with respect to $|\eta|$, so by Lemma \ref{almostaschurproduct},  for every $w \in \Omega$ 
		\begin{equation}\label{double}
		\int_{X} \int_{X}p_{w}(x,y)\langle G(w)H(y), H(x)\rangle_{\mathcal{H}}d|\eta|(x)  d|\eta|(y)\geq 0.
		\end{equation}
		Since $\lambda$  is a nonnegative measure, the double integral in Equation \ref{double} is zero   $\lambda$  almost everywhere. Let $G^{1/2}(w)$ be the unique positive semidefinite square root of the operator $G(w)$, then
		\begin{align*}
		0&=\int_{X} \int_{X}p_{w}(x,y)\langle G(w)H(y), H(x)\rangle_{\mathcal{H}}d|\eta|(x)  d|\eta|(y)\\
		&= \int_{X} \int_{X}p_{w}(x,y)\langle G^{1/2}(w)H(y), G^{1/2}(w)H(x)\rangle_{\mathcal{H}}d|\eta|(x)  d|\eta|(y).
		\end{align*}
		Adding coordinates, we get that
		\begin{align*}
		0&=\int_{X} \int_{X}p_{w}(x,y)\langle G(w)H(y), H(x)\rangle_{\mathcal{H}}d|\eta|(x)  d|\eta|(y)\\
		&= \int_{X} \int_{X}p_{w}(x,y)\left (\sum_{\mu \in \mathcal{I}}[G^{1/2}(w)H(y)]_{\mu}\overline{[G^{1/2}(w)H(x)]_{\mu}} \right )d|\eta|(x)  d|\eta|(y)\\
		&= \sum_{\mu \in \mathcal{I}} \int_{X} \int_{X}p_{w}(x,y)[G^{1/2}(w)H(y)]_{\mu}\overline{[G^{1/2}(w)H(x)]_{\mu}} d|\eta|(x)  d|\eta|(y).
		\end{align*}
		The third equality is a consequence of the Lebesgue Dominated Convergence  Theorem for nonnegative complex valued measures. But then, we must have that
		$$
		\int_{X} \int_{X}p_{w}(x,y)[G^{1/2}(w)H(y)]_{\mu}\overline{[G^{1/2}(w)H(x)]_{\mu}} d|\eta|(x)  d|\eta|(y)= 0, \quad \mu \in \mathcal{I}
		$$
		$\lambda$ almost everywhere (all the remaining equalities and properties we present that depends on $w$ holds $\lambda$ almost everywhere, but we do not specify that to simplify the reading). But since $[G^{1/2}(w)H(y)]_{\mu}d|\eta|$ is a scalar valued  finite Radon measure of compact support on $X$, and
		\begin{align*}
		0&=\int_{X} \int_{X}p_{w}(x,y)[G^{1/2}(w)H(y)]_{\mu}\overline{[G^{1/2}(w)H(x)]_{\mu}} d|\eta|(x)  d|\eta|(y) \\
		&\int_{X} \int_{X}p_{w}(x,y) d([G^{1/2}(w)H(y)]_{\mu}|\eta|(y))  d\overline{([G^{1/2}(w)H(x)]_{\mu}|\eta|)}(x), \quad \mu \in \mathcal{I}
		\end{align*}
		the fact that $p_{w}$ is an universal kernel for every $w \in \Omega$ and Theorem \ref{oneintegraltodoubleintegral} on the complex valued case, implies that the measure $[G^{1/2}(w)H]_{\mu}d|\eta|$ is the zero measure. In particular, for every Borel measurable set $A \subset X$ and $x \in X$ 
		$$
		\int_{A}\langle G(w)H(y), H(x) \rangle_{\mathcal{H}}d|\eta|(y)=  \sum_{\mu \in \mathcal{I}} \overline{[G^{1/2}(w)H(x)]_{\mu}} \int_{A} [G^{1/2}(w)H(y)]_{\mu}d|\eta|(y)=0.
		$$
		But on the other hand, since $H$ is Bochner integrable with respect to $|\eta|$, it is valid that
		$$
		\langle G(w)\eta(A), H(x)\rangle_{\mathcal{H}}=\int_{A}\langle G(w)H(y), H(x) \rangle_{\mathcal{H}}d|\eta|(y)=0,
		$$
		and then
		$$
		\langle G(w)\eta(A), \eta(A)\rangle_{\mathcal{H}}=\int_{A}\langle G(w)\eta(A), H(x)\rangle_{\mathcal{H}}d|\eta|(x)=0.
		$$
		Finally, integrating this over $\Omega$, we obtain that
		$$
		\langle \Lambda(\Omega)\eta(A), \eta(A)\rangle_{\mathcal{H}}= \int_{\Omega}\langle G(w)\eta(A), \eta(A)\rangle_{\mathcal{H}}d\lambda(w)= 0. 
		$$
		By the hypothesis on the operator $\Lambda(\Omega)$, we must have that $\eta(A)=0 \in \mathcal{H}$, since $A$ was an arbitrary Borel measurable set, the measure $\eta$ is the zero measure, which implies that the kernel $P$ is universal by Theorem \ref{oneintegraltodoubleintegral}. The remaining implication is a consequence of Lemma \ref{scalarvaluedprojectionsuniversality}\\
		The proof of $(v)$ follows the same arguments as the proof of $(iv)$, without the assumption that $\eta$ has compact support.\end{proof}

	\begin{proof}[\textbf{Proof of Corollary \ref{compcsuppor}}]The function
		$$	
		p(w,x,y)= f(w(x -y)) \in C([0, \infty)\times \mathbb{R}^{m}\times \mathbb{R}^{m}),	 
		$$
		is bounded by $f(0)$ and for every $w \in (0, \infty)$ the kernel $p_{w}$ is  positive definite, so the statements on the Corollary are  a direct consequence of Theorem \ref{universalityintegralkernel}.
	\end{proof}  
	
	Before proving Theorem \ref{universaleuclidian2} and Theorem \ref{universaleuclidianaskey}, we need a result that connect the scalar valued projections of an operator valued function and the function itself, concerning the $\ell$-times completely monotone property in a similar way  as presented in \cite{neeb1998}.  
	
	\begin{lem}\label{uniquelcomplmon} Let $F:[0, \infty) \to \mathcal{L}(\mathcal{H})$ be an  ultraweakly continuous function and $\ell \in \{2, \ldots , \infty \}$. Then $F_{v}$ is $l$-times completely monotone for all $v \in \mathcal{H}$ if and only if there exists a weak$^{*}$  nonnegative finite  $\mathcal{L}(\mathcal{H})$ valued measure $\Lambda$ for which  the weak$^{*}$ integral representation holds 
		$$
		F(t)= \int_{[0, \infty)} h_{\ell}(rt)d\Lambda(r)	
		$$	
		where $h_{\ell}(t):=(1-t)^{\ell-1}_{+}$ for $\ell \in \mathbb{N}$ and $h_{\infty}(t):= e^{-rt}$. The representation is unique.
	\end{lem} 
	
	\begin{proof} If $F$ admits the weak$^{*}$ integral representation, then for every $v \in \mathcal{H}$
		$$
		F_{v}(t)= \int_{[0,\infty)}h_{\ell}(rt)d\Lambda_{v}(r),
		$$	
		and since $\Lambda_{v}$ is a scalar valued nonnegative finite Radon measure on  $[0, \infty)$, $F_{v}$ is a $\ell$-times completely monotone function. Note that since the representation is unique for scalar valued measures, $\Lambda_{v}$ is uniquely defined for every $v \in \mathcal{H}$, but the operator valued measure $\Lambda$ is also uniquely defined by the scalar valued projection measures $\Lambda_{v}$, $v \in \mathcal{H}$, so $\Lambda$ is also uniquely defined.\\
		Conversely, if all scalar valued projections of the function $F$ are $\ell$-times completely monotone, then for every $v \in \mathcal{H}$ there exists an unique scalar valued nonnegative finite Radon measure $\Lambda_{v}$ for which  
		$$
		F_{v}(t):=\langle F(t)v, v \rangle_{\mathcal{H}}= \int_{[0,\infty)}h_{\ell}(rt)d\Lambda_{v}(r).
		$$	
		Define $\Lambda : \mathscr{B}([0,\infty)) \to \mathcal{L}(\mathcal{H})$ by 
		$$
		4\langle \Lambda(A)u, v \rangle_{\mathcal{H}}:= \Lambda_{u+v}(A) - \Lambda_{u-v}(A) -i(\Lambda_{u-iv}(A) - \Lambda_{u+iv}(A)),
		$$
		then $\Lambda$ is bounded by $F(0)$, is an weak$^{*}$ nonnegative finite $\mathcal{L}(\mathcal{H})$ valued measure and
		$$
		4\int_{[0, \infty)} h_{\ell}(rt)d\Lambda_{u,v}(r)=  F_{u+v}(t) -F_{u-v}(t)-i(F_{u-iv}(t)-F_{u+iv}(t))= 4\langle F(t) u,v \rangle_{\mathcal{H}}. 
		$$\end{proof}
	
	With an similar argument it is also possible to prove a characterization of the ultraweakly continuous positive definite operator valued radial kernels on Euclidean spaces. 
	
	\begin{lem}\label{uniquelcomplmonradial} Let $F:[0, \infty) \to \mathcal{L}(\mathcal{H})$ be an  ultraweakly continuous function and $m \in \mathbb{N}$. The kernel
		$$
		(x,y) \in \mathbb{R}^{m} \times \mathbb{R}^{m} \to F(\|x-y\|) \in \mathcal{L}(\mathcal{H})
		$$
		is positive definite if and only if there exists a weak$^{*}$  nonnegative finite  $\mathcal{L}(\mathcal{H})$ valued measure $\Lambda$ for which  the weak$^{*}$ integral representation holds 
		$$
		F(t)= \int_{[0, \infty)} \Omega_{m}(rt)d\Lambda(r).	
		$$	
		The representation is unique. Additionally, the kernel being positive definite is also equivalent at every kernel $ (x,y) \in \mathbb{R}^{m} \times \mathbb{R}^{m} \to \langle F(\|x-y\|)v,v\rangle_{\mathcal{H}}  \in \mathbb{C}
		$, $v \in \mathcal{H}\setminus{\{0\}}$ being positive definite.\end{lem}

	\begin{proof}[\textbf{Proof of Theorem \ref{universaleuclidian2}}] If the stated kernel is positive definite, then by Lemma \ref{uniquelcomplmon}  there exists an ultraweakly nonnegative $\mathcal{L}(\mathcal{H})$ valued Radon measure $\Lambda$ on $[0, \infty)$ for which the ultraweakly representation is valid
		\begin{equation}\label{operatorgaussian}
		F_{v}(\|x-y\|)= \langle F(\|x-y\|)v,v \rangle_{\mathcal{H}}= \int_{[0, \infty)}e^{-r\|x-y\|^{2}}d\Lambda_{v}(r),\quad  x,y \in \mathbb{R}^{m}.	  
		\end{equation}
		Note that $F(0) = \int_{[0, \infty)}d\Lambda(r) = \Lambda([0, \infty))$ is a trace class operator and  also  that the  function 
		$$
		p: (0, \infty) \times \mathbb{R}^{m} \times \mathbb{R}^{m} \to p(r, x,y) = :e^{-r\|x-y\|^{2}} \in \mathbb{C}
		$$ is continuous  and  bounded.  The kernel $p_{r}$ is $C_{0}^{\infty}$-universal for every  $r>0$ by \cite{Simon-Gabriel2018}. In particular, by Lemma \ref{radonnikodintrace},  $\Lambda$ is a countably additive operator valued  measure that admits a Radon-Nikod\'ym decomposition $d\Lambda= Gd\Lambda_{T}$ and the integral  \ref{operatorgaussian} can be reformulated as the Bochner integral $	F(\|x-y\|)= \int_{[0, \infty)}e^{-r\|x-y\|^{2}}G(r)d\Lambda_{T}(r)$.\\
		After this analysis, we begin the proof.\\
		($ii\to i $) If $F$ defines an universal kernel that is positive definite  then $F$ defines a strictly positive definite kernel by Theorem \ref{oneintegraltodoubleintegral}.\\
		($i \to iii$) If $F$ defines a strictly positive definite kernel then $F_{v}(\|x-y\|)$ is also a  strictly positive definite kernel  for every $v \in \mathcal{H}\setminus{\{0\}}$ by Lemma \ref{scalarvaluedprojectionsuniversality}, but since 
		$$ 
		F_{v}(\|x-y\|)= \int_{[0, \infty)}e^{\|x-y\|^{2}r}d\langle \Lambda(r)v,v \rangle_{\mathcal{H}},\quad  x,y \in \mathbb{R}^{m}	  
		$$
		this can only happen if the complex valued measure $\Lambda_{v}:=\langle \Lambda v, v\rangle_{\mathcal{H}}$ is not concentrated at $0$, or equivalently, the function $F_{v}$ is nonconstant, \cite{sun}.\\
		($iii \to ii$) Finally, if each function $F_{v}$ is non constant, then the measure $\Lambda_{v}$ is such that $\Lambda_{v}((0, \infty))>0$  by \cite{sun} for all $v \in \mathcal{H}\setminus{\{0\}}$, but then 
		$\Lambda((0, \infty))$, which is a trace class positive semidefinite operator must be positive definite, applying Theorem \ref{universalityintegralkernel}, we obtain that the kernel 
		$$
		(x,y) \in \mathbb{R}^{m} \times \mathbb{R}^{m} \to \int_{(0, \infty)}e^{\|x-y\|^{2}r}d\Lambda(r)= F(\|x-y\|) - \Lambda(\{0\})
		$$	
		is universal, and since adding a constant positive semidefinite operator does not change this fact (this can be proved directly from the definition), the kernel defined by $F$ is universal. \\
		Now we focus on the second part of the Theorem. Since $\lim_{\|y\| \to \infty}F_{v}(x-y)= \Lambda_{v}(\{0\}) $, if $\mathcal{H}_{K} \subset C_{0}(\mathbb{R}^{m}, \mathcal{H})$ then $\Lambda(\{0\})=0$ and $F_{v} \in C_{0}([0, \infty))$. Conversely, if $F_{v} \in C_{0}([0, \infty))$ for every $v \in \mathcal{H}\setminus{\{0\}}$, the same limit implies that $\Lambda(\{0\})=0$.  Note that
		\begin{align*}
		\|F(x-y)\|_{\mathcal{L}(\mathcal{H})}&=\|\int_{(0,\infty)}e^{-r\|x-y\|^{2}}G(r)d\Lambda_{T}(r)\|_{\mathcal{L}(\mathcal{H})}\\
		&\leq \int_{(0, \infty)}e^{-r\|x-y\|^{2}}\|G(r)\|_{\mathcal{L}(\mathcal{H})}d\Lambda_{T}(r) \to 0 , \quad \|y\| \to \infty 
		\end{align*}
		which is a stronger property compared to the ones we need to apply  Proposition \ref{rkhscontained}  to obtain that $\mathcal{H}_{K} \subset C_{0}(\mathbb{R}^{m},\mathcal{H})$. 
		The proof for the remaining equivalences follows by the same path as the first one, by using  relation $(v)$ in Theorem \ref{universalityintegralkernel}. 
	\end{proof}
	
	\begin{proof}[\textbf{Proof of Theorem \ref{universaleuclidianaskey}}] %The fact that the ultraweakly finite nonnegative $\mathcal{L}(H)$ valued measure $\Lambda$ at Lemma \ref{uniquelcomplmon} admits a Radon-Nikod\'ym decomposition $d\Lambda= Gd\Lambda_{T}$  because $\Lambda(0)$ is a trace class operator is similar to the one we presented at the proof of Theorem \ref{universaleuclidian2}. 
		The function
		$$
		p: (0, \infty) \times \mathbb{R}^{m} \times \mathbb{R}^{m} \to p(r, x,y) = :(1-r\|x-y\|)^{l-1}_{+}\in \mathbb{C}
		$$ is continuous and  bounded. The kernel $p_{r}$ is $C_{0}$-universal for every  $r>0$, because the kernel is positive definite and the function $x \in \mathbb{R}^{m} \to (1-r\|x\|)^{l-1}_{+} \in C_{c}(\mathbb{R}^{m})$, then by Corollary $10$ of \cite{Sriperumbudur}, we obtain the $C_{0}$-universality.\\
		The remaining statements of this Theorem are proved by the same arguments as those in Theorem \ref{universaleuclidian2}. \end{proof}

	\begin{proof}[\textbf{Proof of Lemma \ref{radialfini}}]
		If $\phi \in C_{c}^{\infty}(\mathbb{R}^{m})$ is a nonzero function, then $\psi:= \phi + \Delta \phi$ is also nonzero  and its Fourier transform satisfies $\hat{\psi}(x)= \hat{\phi} - \|x\|^{2}\hat{\phi}(x)= (1-\|x\|^{2})\hat{\phi}(x)$. Define the scalar valued finite measure $d\eta:= \psi d\xi$ of compact support, then
		\begin{align*}
		\int_{\mathbb{R}^{m}} \int_{\mathbb{R}^{m}}\Omega^{m}_{m}(\|x-y\|)d\eta(y)d\overline{\eta}(x)&=\frac{1}{Vol(S^{m-1})}\int_{\mathbb{R}^{m}} \int_{\mathbb{R}^{m}}\int_{S^{m-1}}e^{-i(x-y)\cdot\xi}d\xi d\eta(y)d\overline{\eta}(x)\\
		&= \frac{1}{Vol(S^{m-1})}\int_{S^{m-1}}|\hat{\eta}(\xi)|^{2}d\xi\\
		&= \frac{1}{Vol(S^{m-1})}\int_{S^{m-1}}(1-\|\xi\|^{2})^{2}|\hat{\phi}(\xi)|^{2}d\xi=0.
		\end{align*}	
		Hence, the kernel $\Omega_{m}^{m}$ is not universal. By the recurrence relation
		$$
		\Omega_{m}(t)= \frac{2\Gamma(m/2)}{\Gamma(1/2)\Gamma((m-1)/2)}\int_{(0,1)} \Omega_{m-1}(rt)(1-r^{2})^{-1/2}r^{m-2}dr 
		$$  
		proved in \cite{Emilio}, we have that
		\begin{align*}
		\Omega_{m}^{m-1}(\|x-y\|)&= \frac{2\Gamma(m/2)}{\Gamma(1/2)\Gamma((m-1)/2)}\int_{(0,1)} \Omega_{m-1}^{m-1}(r\|x-y\|)(1-r^{2})^{-1/2}r^{m-2}dr\\
		&= \frac{2\Gamma(m/2)}{Vol(S^{m-2})\Gamma(1/2)\Gamma((m-1)/2)} \int_{\mathbb{R}^{m-1}}e^{-i(x-y)\cdot\xi}(1-\|\xi\|^{2})_{+}^{-1/2}\|\xi\|^{m-2}d\xi.
		\end{align*}
		So, the support of $(1-\|\xi\|^{2})_{+}^{-1/2}\|\xi\|^{m-2}d\xi$	is the   closed unit ball $B[0,1]$ on $\mathbb{R}^{m-1}$, Theorem  $18$ in \cite{Simon-Gabriel2018} implies that the kernel $\Omega_{m}^{m-1}$ is $C^{\infty}$-universal, while Theorem $9$ at \cite{Sriperumbudur} implies that the kernel $\Omega_{m}^{m-1}$ is not $C_{0}$-universal.\\
		For the second part,  since $f:[0, \infty) \to \mathbb{R}$  is a nonconstant function and defines a positive definite radial kernel on $\mathbb{R}^{m}$, then there exists a  complex valued nonnegative finite measure $\lambda$ on $[0, \infty)$ for which
		$$
		f(\|x-y\|)= \int_{[0, \infty)} \Omega^{m}_{m}(r\|x-y\|)d\lambda(r), \quad x,y \in \mathbb{R}^{m}, 
		$$
		with $\lambda((0, \infty))>0$. Additionally,  being $f\in C^{2q}([0, \infty))$, by \cite{gneiting2}  we get that $\int_{[0, \infty)}w^{2q}d\lambda(w)<\infty$, and by the arguments of Lemma \ref{diffform}, we obtain that $p^{m}_{r}(x,y):=\Omega^{m-1}_{m}(r\|x-y\|)$ is $C^{q}$-dominated with respect to $\lambda$ on $\mathbb{R}^{m-1}$. Since for every $r>0$ this kernel is $C^{q}$-universal on $\mathbb{R}^{m-1}$, as a consequence of  Theorem \ref{Cpuniversalityintegralkernelcomplvalue} and the fact that $\lambda((0, \infty))>0$, we obtain that the kernel  
		the kernel
		$$
		(x,y) \in \mathbb{R}^{m-1}\times \mathbb{R}^{m-1} \to f(\|x-y\|)=\int_{[0, \infty)} \Omega^{m-1}_{m}(r\|x-y\|)d\lambda(r) \in \mathbb{R}
		$$
		is $C^{q}$-universal.
	\end{proof}
	
	\begin{proof}[\textbf{Proof of Theorem \ref{universaleuclidian3}} ]
		By  Lemma \ref{radialfini} the kernel 
		$$(0, \infty)\times\mathbb{R}^{m-1}\times \mathbb{R}^{m-1} \to  p(r,x,y):= \Omega_{m}^{m-1}(r\|x-y\|) $$	
		is continuous, bounded and for every fixed $r>0$ it is $C^{\infty}$-universal. 	The remaining statements of this Theorem are proved by the same arguments as those in Theorem \ref{universaleuclidian2}.
	\end{proof}

	\subsection{\textbf{Section \ref{On differentiable positive definite operator valued universal (and related) kernels}}}

	\begin{proof}[\textbf{Proof of Lemma \ref{diffimersion}}] We prove the case $q=1$, the general case follows by an induction argument. Suppose that $\mathcal{H}_{K} \subset C^{1}(U, \mathcal{H})$. Let $(h_{n})_{n \in \mathbb{N}}$ be a sequence of nonzero real numbers that converges to $0$, then for every fixed $y \in U$ and $v \in \mathcal{H}$ the function 
		$$
		z \in U \to  [\Delta^{e_{i}}_{h_{n}}K_{y}v](z):=[K(z, y+h_{n}e_{i})v - K(z,y)v]/h_{n} \in \mathcal{H}
		$$
		is an element of $ \mathcal{H}_{K}$. Since 
		\begin{equation}\label{diffmersionform}
		\langle \Delta^{e_{i}}_{h_{n}}K_{y}v, F \rangle_{\mathcal{H}_{K}}= \langle v,(F( y+h_{n}e_{i})- F( y))/h_{n} \rangle_{\mathcal{H}} \to  \langle v, \partial^{e_{i}}F(y) \rangle_{\mathcal{H}} 
		\end{equation}
		by the Uniform Boundedness principle we obtain that 
		$$
		\sup_{y \in \mathcal{A}}\sup_{n \in \mathbb{N}} \sup_{\|v\|=1}(\|\Delta^{e_{i}}_{h_{n}}K_{y}v\|_{\mathcal{H}_{K}})^{2}:= M^{e_{i}}_{\mathcal{A}}<\infty.
		$$
		Since on a RKHS every bounded sequence admits a subsequence that converges weakly and pointwise to a function on the RKHS, we obtain that the function  
		$$ 
		z\in U \to \partial_{y}^{e_{i}}[K_{y}v](z)=\partial_{2}^{e_{i}}[K(z,y)v] \in \mathcal{H}
		$$
		exists and is an element of $\mathcal{H}_{K}$ and this settles relation $(i)$. Moreover 
		\begin{equation}\label{diffmersionform2}
		\sup_{\|v\|=1}\sup_{y\in \mathcal{A}}(\| \partial_{y}^{e_{i}}[K_{y}v]\|_{\mathcal{H}_{K}})^{2}\leq M^{e_{i}}_{\mathcal{A}}.
		\end{equation}
		Note that if $F \in \mathcal{H}_{K}$, by Equation \ref{diffmersionform} and the weak convergence  $\Delta^{e_{i}}_{h_{n}}K_{y}v \rightharpoonup \partial_{y}^{e_{i}}[K_{y}v]$ on $\mathcal{H}_{K}$, we obtain relation $(ii)$. Since $\mathcal{H}_{K} \subset C^{1}(U, \mathcal{H})$ we can differentiate  $\partial_{2}^{e_{i}}[K(x,y)v]$ on the variable $x$, while the weak convergence implies that 
		$$
		\langle u, \partial_{1}^{e_{j}}[\partial_{2}^{e_{i}}[K(x,y)v]]\rangle_{\mathcal{H}}=  \langle \partial_{x}^{e_{j}}[K_{x}u] , \partial_{y}^{e_{i}}[K_{y}v]\rangle_{\mathcal{H}_{K}}.
		$$
		From the previous equality we obtain relation $(iv)$. The triple supremum  is also a consequence of the previous equality together with Equation \ref{diffmersionform2} and $\sup_{x\in \mathcal{A}}(\|K(x,x)\|_{\mathcal{L}(\mathcal{H})})^{1/2}:= M_{\mathcal{A}}^{0}< \infty$ from Proposition \ref{rkhscontained}. The inclusion-restriction is continuous because if $F \in \mathcal{H}_{K}$ 
		$$
		|\langle v, \partial^{e_{i}}F (x)\rangle _{\mathcal{H}}|= |\langle \partial^{e_{i}}_{x}K_{x}v, F \rangle _{\mathcal{H}_{K}}|\leq M^{e_{i}}_{\mathcal{A}} \|F\|_{\mathcal{H}_{K}},
		$$  
		and then $\|F\|_{C^{1}(\mathcal{A}, \mathcal{H})} \leq ( \sum_{i =1}^{m}M^{e_{i}}_{\mathcal{A}} +  M_{\mathcal{A}}^{0})\|F\|_{\mathcal{H}_{K}}$. \\
		It only remains to prove the converse of the first assertion. By the mean value inequality and the hypothesis we have that
		\begin{align*}
		\|K(z, x +he_{i})v& - K(z, x)v - h\partial_{x}^{e_{i}}[K_{x}v](z) \|_{\mathcal{H}}\\
		&\leq |h| \sup_{w \in [x, x+he_{i}]}\|  \partial_{w}^{e_{i}}[K_{w}v](z) - \partial_{x}^{e_{i}}[K_{x}v](z) \|_{\mathcal{H}}\leq 2|h|\|v\|M_{\mathcal{A}}.
		\end{align*}
		for every $z \in \mathcal{A}$, $x \in Int(\mathcal{A})$ and  small enough $|h|$. Similarly, 
		\begin{align*}
		\|\partial_{y}^{e_{i}}[K_{y}v](x + he_{i})& -\partial_{y}^{e_{i}}[K_{y}v](x) - h\partial_{1}^{e_{i}}[\partial_{2}^{e_{i}}[K(x,y)v]] \|_{\mathcal{H}}\\
		&\leq |h| \sup_{w \in [x, x+he_{i}]}\|\partial_{1}^{e_{i}}[\partial_{2}^{e_{i}}[K(w,y)v]] -\partial_{1}^{e_{i}}[\partial_{2}^{e_{i}}[K(x,y)v]] \|_{\mathcal{H}}\leq 2|h|\|v\|M_{\mathcal{A}}.
		\end{align*}              
		for every $y \in \mathcal{A}$, $x \in Int(\mathcal{A})$ and  small enough $|h|$. So
		\begin{equation}\label{counterargumentmatrixdeformation}
		| \langle \Delta_{h}^{e_{i}}K_{x}v, \Delta^{e_{i}}_{h}K_{x}v \rangle_{\mathcal{H}_{K}} - \langle \partial_{1}^{e_{j}}[\partial_{2}^{e_{i}}[K(x,x)v]], v\rangle_{\mathcal{H}}|\leq 6\|v\|^{2}M_{\mathcal{A}},
		\end{equation}
		then any sequence $\Delta^{e_{i}}_{h_{n}}K_{x}v$ is bounded in $\mathcal{H}_{K}$, by the hypothesis it converges pointwise to $\partial_{x}^{e_{i}}K_{x}v$, so $\Delta^{e_{i}}_{h_{n}}K_{x}v \rightharpoonup \partial_{x}^{e_{i}}[K_{x}v]\in \mathcal{H}_{K}$. In order to prove that an arbitrary $F \in \mathcal {H}_{K}$ is differentiable, let $\phi_{n} \in H_{K}$ be a sequence that converges to $F$ on the $\mathcal{H}_{K}$ norm. Since
		$$
		|\langle F(x) - \phi_{n}(x), v  \rangle_{\mathcal{H}}| \leq \|F - \phi_{n}\|_{\mathcal{H}_{K}} \|K_{x}v\|_{\mathcal{H}_{K}} 
		$$
		the function $F$ is continuous. Similarly
		$$
		| \langle \Delta_{h}^{e_{i}}F(x) -   \Delta_{h}^{e_{i}}\phi_{n}(x), v   \rangle_{\mathcal{H}} | \leq \|F - \phi_{n} \|_{\mathcal{H_{K}}} \|  \Delta_{h}^{e_{i}}K_{x}v \|_{\mathcal{H}_{K}}
		$$
		so the function $F$ is differentiable, finally
		$$
		| \langle \partial^{e_{i}}F(x) -    \partial^{e_{i}}\phi_{n}(x), v   \rangle_{\mathcal{H}} | = | \langle \partial^{e_{i}}_{x}K_{x}v, F - \phi_{n}   \rangle_{\mathcal{H}_{K}} | \leq \|F - \phi_{n} \|_{\mathcal{H_{K}}} \|   \partial^{e_{i}}K_{x}v \|_{\mathcal{H}_{K}}
		$$
		which concludes that $F \in C^{1}(U, \mathcal{H})$.\end{proof}
	\begin{rem}
		At Equation \ref{counterargumentmatrixdeformation} it might occur that
		$$
		\lim_{h \to 0}  \langle \Delta_{h}^{e_{i}}K_{x}v, \Delta^{e_{i}}_{h}K_{x}v \rangle_{\mathcal{H}_{K}} \neq  \langle \partial_{1}^{e_{j}}[\partial_{2}^{e_{i}}[K(x,x)v]], v\rangle_{\mathcal{H}}.
		$$
		For instance, if $k(x,y)= x^{2}y^{2}/(x^{2} + y^{2})$ then 
		$$
		\frac{1}{h^{2}}[k(h,h) - k(h,0)-k(0,h)+k(0,0)]= \frac{1}{2}
		$$
		while
		$$
		\partial_{2}[K(x,0)]=0 \text{ for all } x \in \mathbb{R} \longrightarrow \partial_{1}[\partial_{2}[K(0,0)]]=0. 
		$$
	\end{rem}

	\begin{proof}[\textbf{Proof of Proposition \ref{diffimersion0}}] Suppose that  $\mathcal{H}_{K} \subset C^{q}_{0}(U, \mathcal{H})$. On the proof of Proposition \ref{diffimersion} it is proved that $\partial_{x}^{\alpha}K_{x}v \in \mathcal{H}_{K} \subset C_{0}(U, \mathcal{H})$, for every $|\alpha| \leq q$.  The inequality involving $\partial_{1}^{\alpha}[\partial_{2}^{\beta}[K(x,y)v]]$ can be proved in a similar way as  the proof of Lemma \ref{diffimersion}, by  applying the Uniform Boundedness principle to the whole set $U$ (instead of just a compact set). The converse holds because the approximations $\phi_{n} \in H_{K} \subset C_{0}(U, \mathcal{H})$ for a function $F \in \mathcal{H}_{K}$(and its derivatives) is uniform on the whole set $U$ (instead of being uniform at all compact sets of $U$).\\
		Relation $(i)$ and $(ii)$ and $(iv)$ were already proved at Proposition \ref{diffimersion}.\\
		The inclusion $I : \mathcal{H}_{K} \to C_{0}^{q}(U, \mathcal{H})$ is continuous because
		$$	 
		|\langle \partial^{\alpha}F(x), v \rangle_{\mathcal{H}_{K}} | \leq \|F\|_{\mathcal{H}_{k}} \|\partial^{\alpha}_{x}K_{x}v \|_{\mathcal{H}_{K}}\leq \|F\|_{\mathcal{H}_{k}} M^{1/2} \|v\|
		$$	   
		for all $|\alpha| \leq q$, which proves $(ii)$.
	\end{proof}
	
	\begin{proof}{\textbf{Proof of Lemma \ref{compacinclusiondiff} }} If $K \in C^{q,q}(U\times U, \mathcal{L}(\mathcal{H}))$  and $B \subset \mathcal{H}_{K}$ is a bounded set, then the equicontinuity follows from the inequality
		\begin{align*}
		|\langle v, &\partial^{\alpha}F(z) - \partial^{\alpha}F(y)\rangle_{\mathcal{H}}| \leq \|F\|_{\mathcal{H}_{K}}\| \partial^{\alpha}_{1}[K(\cdot , z  )v] - \partial^{\alpha}_{1}[K(\cdot , y  )v]  \|_{\mathcal{H}_{K}}\\
		&\leq  \|F\|_{\mathcal{H}_{K}} \|v\|_{\mathcal{H}}(\| \partial^{\alpha}_{1}\partial^{\alpha}_{2}K(z , z  )  -\partial^{\alpha}_{1}\partial^{\alpha}_{2}K(z , y  )-\partial^{\alpha}_{1}\partial^{\alpha}_{2}K(y , z  )+\partial^{\alpha}_{1}\partial^{\alpha}_{2}K(y , y  )  \|_{\mathcal{L}(\mathcal{H})} )^{1/2
		}
		\end{align*}        
		For the converse we prove the case $q=1$, the general case follows by an induction argument. Since every bounded set of $\mathcal{H}_{K}$ is $C^{q}$-equicontinuous, the same set must be equicontinuous by definition, in particular  relation $(i)$ in Lemma \ref{compacinclusion} implies that  the kernel $K : U\times U \to \mathcal{L}(\mathcal{H})$ is continuous on the norm topology of $\mathcal{L}(H)$.\\ 
		For any compact set $\mathcal{A}\subset U$, the set $B:=\{ \partial_{z}^{\beta}K_{z}v, \quad  z \in \mathcal{A}, \|v\|=1, |\beta|\leq 1  \} $ is bounded on $\mathcal{H}_{K}$ by Proposition \ref{diffimersion}. From the $C^{q}$-equicontinuity we get that for every $x \in Int(\mathcal{A})$ there exists an open set $U_{x}$ that contains $x$ and
		$$
		\|\partial^{\alpha}_{1}[\partial_{2}^{\beta}[K(w,z)v]]- \partial^{\alpha}_{1}[\partial_{2}^{\beta}[K(x,z)v]] \|_{\mathcal{H}}< \epsilon,\quad  w \in U_{x}, z\in \mathcal{A}, \|v\|_{\mathcal{H}}=1, |\alpha|, |\beta|\leq 1. 
		$$ 
		By a similar argument as the one in the proof of Lemma \ref{compacinclusion}, we get that
		\begin{equation}\label{trick}
		\|\partial^{\alpha}_{1}[\partial_{2}^{\beta}[K(x,y)v]] - \partial^{\alpha}_{1}[\partial_{2}^{\beta}[K(x^{\prime}, y^{\prime})v]]\|_{\mathcal{H}}   < 2\epsilon, 	
		\end{equation}	
		for all $x^{\prime} \in Int(\mathcal{A})\cap U_{x}$, $y^{\prime} \in Int(\mathcal{A}) \cap U_{y}$ , $|\alpha|,|\beta|\leq 1$ and $\|v\|=1$.
		Then  by the mean value inequality 
		$$
		\| [\Delta^{e_{i}}_{h,2}K(x,y) - \Delta^{e_{i}}_{h^{\prime},2}K(x,y)]v \|_{\mathcal{H}} \leq 2 \sup_{w \in V_{2} } \| \partial^{e_{i}}_{2}[K(x,y)v]- \partial^{e_{i}}_{2}[K(x,w)v]      \|_{\mathcal{H}} 
		$$
		where $V_{2} := [y, y+he_{i}]\cup[y, z+ y^{\prime}e_{i}]$, this inequality together with Equation \ref{trick} proves that $\Delta^{e_{i}}_{h,2}K(x,y)$ is a Cauchy sequence (on $h$), which proves that $\partial^{e_{i}}_{2}K(x,y) $ exists. It is jointly continuous because Equation \ref{trick} for $\alpha =0$ and $\beta = e_{i}$  holds for every $\|v\|=1$. Similarly, by the mean value inequality 
		$$
		\| [\Delta^{e_{j}}_{h,1}\partial^{e_{i}}_{2}K(x,y) - \Delta^{e_{j}}_{h^{\prime},1}\partial^{e_{i}}_{2}K(x,y)]v \|_{\mathcal{H}} \leq 2 \sup_{w \in V_{1} } \| \partial^{e_{j}}_{1}[\partial^{e_{i}}_{2}[K(x,y)v]]- \partial^{e_{j}}_{1}[\partial^{e_{i}}_{2}[K(w,y)v]]      \|_{\mathcal{H}} 
		$$
		where $V_{1} := [x, x+he_{j}]\cup[x, x+ h^{\prime}e_{j}]$, this inequality together with Equation \ref{trick} proves that $\Delta^{e_{j}}_{h,1}\partial_{2}^{e_{i}}K(x,y)$ is a Cauchy sequence (on $h$), which proves that $\partial^{e_{j}}_{1}\partial_{2}^{e_{i}}K(x,y)$ exists. It is jointly continuous because Equation \ref{trick} for $\alpha =e_{j}$ and $\beta = e_{i}$ holds for every $\|v\|=1$. This settles $(i)$. \\
		As for $(ii)$, note that
		$$
		|\langle e_{\mu}, \partial^{\alpha}F(y)\rangle_{\mathcal{H}}| \leq \|F\|_{\mathcal{H}_{K}}  \sqrt{\langle  [\partial^{\alpha}_{1}[\partial^{\alpha}_{2}K(y,y)e_{\mu}]], e_{\mu} \rangle_{\mathcal{H}}},
		$$
		for every $|\alpha|\leq q$, since $ \sum_{\mu \in \mathcal{I}}(\sqrt{\langle  [\partial^{\alpha}_{1}[\partial^{\alpha}_{2}K(x,x)e_{\mu}]], e_{\mu} \rangle_{\mathcal{H}}})^{2}< \infty$, the set  $\{\partial^{\alpha}F(y), F \in B, |\alpha|\leq q   \} \subset \mathcal{H}$ has compact closure on the norm topology of $\mathcal{H}$.\\
		Now we prove $(iii)$. By definition, the inclusion-restriction is  a compact operator if and only if for every bounded set $B \subset \mathcal{H}_{K}$ (restricted to $\mathcal{A}$) has compact closure on  $C^{q}(\mathcal{A}, \mathcal{H})$. By the differentiable version of  of the Arzel\`{a}-Ascoli Theorem this occurs if and only if the set $B$(restricted to $\mathcal{A}$) is $C^{q}$-equicontinuous and $C^{q}$-pointwise relatively compact, the conclusion follows from $(i)$, $(ii)$. The second part of $(iii)$ is a consequence from the fact  that  closed sets on $\mathcal{H}_{K}$(restricted to $\mathcal{A}$) are closed on $C^{q}(\mathcal{A}, \mathcal{H})$.\\ 
		The proof of $(iv)$ is similar to the previous ones,  we just emphasize that the importance of the compact set $\mathcal{A}$  is to ensure that $B_{\mathcal{A}}$ is $C^{q}$-equicontinuous at the infinity point of $U$.
	\end{proof}

	\begin{proof}[\textbf{Proof of Lemma \ref{diffkernelfourtrans}}] Since the matrix valued kernel 
		$$
		(x,y) \in\mathbb{R}^{m}\times  \mathbb{R}^{m}\to  \partial_{1}^{\alpha} \partial_{2}^{\beta}p_{w}(x,y)\in M_{\ell}(\mathbb{C})
		$$
		is positive definite by Lemma \ref{diffimersion}, we have that 
		$$
		2|\partial_{1}^{\alpha} \partial_{2}^{\beta}p_{w}(x,y)| \leq \partial_{1}^{\alpha} \partial_{2}^{\alpha}p_{w}(x,x) + \partial_{1}^{\beta} \partial_{2}^{\beta}p_{w}(y,y) \leq 2h(w).
		$$
		By the  hypothesis we obtain  that $|\partial_{1}^{\alpha} \partial_{2}^{\beta}p_{w}(x,y)|\|G(w)\|_{\mathcal{L}(\mathcal{H})}$ is $\lambda$ integrable and  Equation \ref{bochnerintegration} implies that the kernels $K_{\alpha, \beta}$ are well defined.\\
		As for the differentiability of the kernels, we prove the case $\alpha = e_{i}$ and $\beta=0$, being the other cases proved by an induction argument based on the one we present. Note that
		\begin{align*}
		&\|\Delta_{h,1}^{e_{i}}K_{0,0}(x,y) - K_{e_{i},0}(x,y)   \|_{\mathcal{L}(\mathcal{H})}\\
		&=  \|\int_{\Omega}\left [\frac{p_{w}(x+he_{i}, y) - p_{w}(x,y)}{h} - \partial_{1}^{e_{i}}p_{w}(x,y)    \right ] G(w) d\lambda(w)\|_{\mathcal{L}(\mathcal{H})}\\
		&\leq \int_{\Omega} | \frac{p_{w}(x+he_{i}, y) - p_{w}(x,y)}{h} - \partial_{1}^{e_{i}}p_{w}(x,y) | \|G(w)\|_{\mathcal{L}(\mathcal{H})}	 d\lambda(w).
		\end{align*} 
		By the mean value inequality 
		$$|-\partial_{1}^{e_{i}}p_{w}(x,y) +\frac{p_{w}(x+he_{i}, y) - p_{w}(x,y)}{h} | \leq \sup_{t\in [x,x+h]}|\partial_{1}^{e_{i}}p_{w}(x,y)- \partial_{1}^{e_{i}}p_{w}(t,y)| \leq 2h(w) $$
		while $[p_{w}(x+he_{i}, y) - p_{w}(x,y)]/h$ converges  to $\partial_{1}^{e_{i}}p_{w}(x,y)$, so the Lebesgue Dominated Convergence Theorem implies that
		$$
		\int_{\Omega} | \frac{p_{w}(x+he_{i}, y) - p_{w}(x,y)}{h} - \partial_{1}^{e_{i}}p_{w}(x,y) | \|G(w)\|_{\mathcal{L}(\mathcal{H})}	 d\lambda(w) \to 0
		$$ 
		which proves our claim.	
	\end{proof}
	
	\begin{proof}[\textbf{Proof of Lemma \ref{diffform}} ]
		First, we prove the converse. Indeed, since $f \in C^{2q}(\mathbb{R}^{m})$, then the kernel
		$$
		(w,x,y) \in [0,\infty)\times \mathbb{R}^{m} \times \mathbb{R}^{m} \to f(w(x-y))
		$$ 
		is an element of $C^{0,q,q}([0,\infty)\times \mathbb{R}^{m} \times \mathbb{R}^{m})$, and it can be easily deducted that $$\partial^{\alpha}_{x}\partial^{\beta}_{y}f(w(x-y))=w^{|\alpha|}(-w)^{|\beta|}[\partial^{\alpha + \beta}f](w(x-y)).$$
		But since $f(w(x-y))$ is a positive definite kernel, the matrix valued kernel $[\partial^{\alpha}_{x}\partial^{\beta}_{y}f(w(x-y))]_{|\alpha|, |\beta|\leq q}$  is positive definite by Lemma \ref{diffimersion}, and then
		$$
		|\partial^{\alpha}_{x}\partial^{\beta}_{y}f(w(x-y))|\leq w^{2|\alpha|}|\partial^{2\alpha}f(0)|  + w^{2|\beta|}|\partial^{2\beta}f(0)|.
		$$
		By choosing $h(w)= M w^{2q}$, for a suitable nonnegative number $M$, we can apply Lemma \ref{diffkernelfourtrans} on the scalar valued case and obtain the result.\\
		On the other hand, suppose that $p \in C^{q,q}(\mathbb{R}^{m}\times \mathbb{R}^{m})$. Then by the hypothesis
		$$
		[p(x+he_{j}, x+he_{j}) - p(x, x+he_{j}) - p(x+he_{j},x) - p(x,x)]/h^{2} \to \partial^{e_{j}}_{1}\partial^{e_{j}}_{2}p(x,x),  
		$$
		but
		\begin{align*}
		p(x+he_{j}, x+he_{j}) - p(x, x+he_{j})& - p(x+he_{j},x) - p(x,x)\\
		& = \int_{[0,\infty)} 2f(0) - f(hwe_{j}) - f(hwe_{j})d\lambda(w).
		\end{align*}
		Note that $2f(0) - f(hwe_{j}) - f(-hwe_{j}) \geq 0$ (because $f$ defines a positive definite kernel on Bochner's sense)  and $[2f(0) - f(hwe_{j}) - f(-hwe_{j})]/h^{2} \to -\partial^{2e_{j}} [f(wx)](0)= w^{2}[-\partial^{2e_{j}}f(0)] $.\\
		If $-\partial^{2e_{j}}f(0) = 0 $ for $1\leq j \leq m$, then by the positivity of the matrix valued kernel $[\partial^{\alpha}_{x}\partial^{\beta}_{y}f(w(x-y))]_{|\alpha|, |\beta|\leq q}$, we obtain that $\partial^{e_{j}}f=0$ for $1\leq j \leq m$, and consequently $f$ is a constant function, which is an absurd by the hypothesis.  
		The integrability of $w^{2}$ with respect to $\lambda$ is then a consequence of the Fatou' s Lemma.\\
		The general case follows by an induction on this argument. Recall that any nonconstant polynomial on an Euclidean space is not a bounded function, however $f$ is a bounded function.   
	\end{proof}
	
	\begin{proof}[\textbf{Proof of Lemma \ref{addstepunidiff}}] We prove the $C^{q}$-universal case, being the other proof simpler. If $F\in \mathcal{H}_{K}$, there exists a sequence $F_{n} \in H_{K}:=span\{K_{x}v, x \in U , v \in \mathcal{H}\}$ for which $F_{n} \to F$ on the norm  of $\mathcal{H}_{K}$. Since the  the inclusion-restriction $I: \mathcal{H}_{K} \to C^{q}(\mathcal{A}, \mathcal{H})$ is continuous we have that $\|F_{|\mathcal{A}}-(F_{n})_{|\mathcal{A}}\|_{C^{q}(\mathcal{A}, \mathcal{H})} \leq M \|F-F_{n}\|_{\mathcal{H}_{K}}$, for some $M>0$, and then  $[\mathcal{H}_{K}]_{|\mathcal{A}}$ is dense in $C^{q}(\mathcal{A}, \mathcal{H})$ if and only if $[H_{K}]_{|\mathcal{A}}$ is dense in $C^{q}(\mathcal{A}, \mathcal{H})$.\\
		From functional analysis, we know that $[H_{K}]_{|\mathcal{A}}$ is dense in $C^{q}(\mathcal{A}, \mathcal{H})$ if and only if the  only continuous linear functional $T: C^{q}(\mathcal{A}, \mathcal{H}) \to \mathbb{C} $ for which $T([K_{y}v]_{|\mathcal{A}})=0$ for all $x \in U$ and $v \in \mathcal{H}$ is the zero functional. By Corollary \ref{Dinculeanu-Singer-diff}, we can describe a continuous linear functional $T$ in $C^{q}(\mathcal{A}, \mathcal{H})$ by 
		$$
		T([K_{y}v]_{|\mathcal{A}})=\sum_{|\alpha|\leq q} \int_{\mathcal{A}}\langle \partial_{x}^{\alpha}K(x,y)v, d\eta_{\alpha}(x) \rangle
		$$ 
		where $\eta= (\eta_{\alpha})_{|\alpha|\leq q} \in \mathfrak{M}^{q}(\mathcal{A}, \mathcal{H})$ and this linear functional is zero if and only if 
		$$
		T(\psi)=\sum_{|\alpha|\leq q} \int_{\mathcal{A}}\langle \partial^{\alpha}\psi, d\eta_{\alpha}(x) \rangle=0
		$$ 	 
		for all $\psi \in C^{q}(\mathcal{A}, \mathcal{H})$.\end{proof}
	
	\begin{proof}[\textbf{Proof of Theorem \ref{oneintegraltodoubleintegralcq}}]We focus the arguments on the $C^{q}$-universal case, being the other case similar.\\
		First, note that for every $\eta =(\eta_{\alpha})_{|\alpha|\leq q} \in \mathfrak{M}^{q}(\mathcal{A}, \mathcal{H})$, where $\mathcal{A}\subset U$ is a compact set that satisfies $\overline{Int(\mathcal{A})}=\mathcal{A}$, the linear functional 
		$$
		F \in \mathcal{H}_{K} \to \sum_{|\alpha|\leq q}\int _{\mathcal{A}}\langle \partial^{\alpha}F(x), d\eta_{\alpha}(x) \rangle \in \mathbb{C} 
		$$ is continuous. Indeed, it is the composition of the inclusion-restriction  $I: \mathcal{H}_{K} \to C^{q}(\mathcal{A}, \mathcal{H})$, proved to be continuous at Lemma \ref{diffimersion},  with a continuous linear functional of   $C^{q}(\mathcal{A}, \mathcal{H})$ by Corollary \ref{Dinculeanu-Singer-diff}.\\
		But then, the Riesz Representation Theorem for Hilbert spaces implies that there exists a function $K_{\eta} \in \mathcal{H}_{K}$ for which 
		$$
		\langle F,K_{\eta}  \rangle_{\mathcal{H}_{K}} = \sum_{|\alpha|\leq q}\int _{\mathcal{A}}\langle \partial^{\alpha}F(x), d\eta_{\alpha}(x) \rangle \in \mathbb{C}  
		$$
		for all $F \in \mathcal{H}_{K}$, and in particular
		\begin{align*}
		\langle v, K_{\eta}(y)\rangle_{\mathcal{H}}&= \langle K_{y}v, K_{\eta} \rangle_{\mathcal{H}_{K}}=  \sum_{|\alpha|\leq q}\int _{\mathcal{A}}\langle \partial^{\alpha}_{1}[K(x,y)v], d\eta_{\alpha}(x) \rangle\\
		&= \langle v, \sum_{|\alpha|\leq q}\int _{\mathcal{A}} \partial^{\alpha}_{1}K(x,y) d\eta_{\alpha}(x) \rangle_{\mathcal{H}}
		\end{align*}
		The last equality makes sense, because  the function $x \in U \to \partial^{\alpha}_{1}K(x,y) \in \mathcal{L}(\mathcal{H})$ is weak-Bochner integrable with respect to $\eta_{\alpha}$, since by Proposition \ref{diffimersion0} 
		$$
		| \int _{\mathcal{A}} \langle\partial^{\alpha}_{1}K(x,y)v, d\eta_{\alpha}(x) \rangle| \leq  \int _{\mathcal{A}} \|\partial^{\alpha}_{1}K(x,y)v\|_{\mathcal{H}}d|\eta_{\alpha}|(x)\\
		\leq M|\eta_{\alpha}|(\mathcal{A}) \|v\|_{\mathcal{H}}.
		$$
		This settles $(i)$.\\
		%We remark that on the $C_{0}^{q}$-universal case the  integral
		%$$
		%\int _{U} \|\partial^{\alpha}_{x}K(x,y)\|_{\mathcal{L}(\mathcal{H})}d|\eta_{\alpha}|(x)
		%$$
		%is finite because by the positivity of the matrix valued operator kernel at relation $(iv)$ in Lemma \ref{diffimersion}  
		%$$
		%\|\partial^{\alpha}_{x}K(x,y)\|_{\mathcal{L}(\mathcal{H})} \leq \|\partial^{\alpha}_{1}\partial^{\alpha}_{2}K(x,x)\|_{\mathcal{L}(\mathcal{H})}\|K(y,y)\|_{\mathcal{L}(\mathcal{H})},
		%$$
		%together with  relation $(v)$ on the same Lemma. 
		As for the proof of $(ii)$, by relation $(ii)$ in Lemma \ref{diffimersion}, we have that
		$$
		\langle v, \partial^{\beta}K_{\eta}(y) \rangle_{\mathcal{H}}= \langle \partial^{\beta}K_{y}v, K_{\eta}\rangle_{\mathcal{H}_{K}}=\sum_{|\alpha|\leq q}\int_{\mathcal{A}}\langle \partial_{1}^{\alpha}[\partial_{2}^{\beta}[K(x,y)v]] , d\eta_{\alpha}(x)\rangle
		$$
		and the weak-Bochner integral exists by a similar argument used at relation $(i)$.\\ 
		Now we prove $(iii)$. The first equality is immediate. For the second equality, note that 
		\begin{align*}
		\langle K_{\eta}, K_{\eta}\rangle_{\mathcal{H}_{K}}&= \sum_{|\alpha|\leq q} \langle \int_{\mathcal{A}}\partial^{\alpha} K_{\eta}(x), d\eta_{\alpha}(x) \rangle  =\sum_{|\alpha|\leq q}\int_{\mathcal{A}} \langle \partial^{\alpha} K_{\eta}(x), H_{\alpha}(x) \rangle_{\mathcal{H}}d|\eta|(x) \\ 
		&=\sum_{|\alpha|\leq q}\int_{\mathcal{A}} \overline{\langle H_{\alpha}(x), \partial^{\alpha} K_{\eta}(x) \rangle_{\mathcal{H}} }d|\eta|(x) \\
		&=\sum_{|\alpha|\leq q}\int_{\mathcal{A}} \overline{\langle \sum_{|\beta|\leq q}\int_{\mathcal{A}} \partial_{1}^{\beta}[ \partial_{2}^{\alpha} [K(y,x)  H_{\alpha}(x)]] , d\eta_{\beta}(y) \rangle }d|\eta|(x)\\
		&= \sum_{|\alpha|, |\beta|\leq q}\int_{\mathcal{A}}\overline{ \int_{\mathcal{A}}\langle \partial_{1}^{\beta}[ \partial_{2}^{\alpha} [K(y,x)  H_{\alpha}(x)]] , H_{\beta}(y) \rangle_{\mathcal{H}}}d|\eta|(y) d|\eta|(x)\\
		&= \sum_{|\alpha|, |\beta|\leq q}\int_{\mathcal{A}} \int_{\mathcal{A}} \langle H_{\beta}(y), \partial_{1}^{\beta}[ \partial_{2}^{\alpha} [K(y,x)  H_{\alpha}(x)]]  \rangle_{\mathcal{H}}d|\eta|(y) d|\eta|(x)\\
		& = \sum_{|\alpha|,|\beta| \leq q} \int_{\mathcal{A}}   \int_{\mathcal{A}} \langle \partial_{1}^{\alpha}[\partial_{2}^{\beta}[K(x,y) H_{\beta}(y)]], H_{\alpha}(x)\rangle_{\mathcal{H}}  d|\eta|(x) d|\eta|(y).
		\end{align*}
		Since the functions $H_{\alpha}$ are Bochner integrable with respect to $|\eta|$ and \\$x \in \mathcal{A} \to \|\partial_{y}^{\beta} \partial_{x}^{\alpha}K(y,x)\|_{\mathcal{L}(\mathcal{H})} \in \mathbb{C}$ are bounded functions, it is possible to reverse the order of integration by Fubinni-Tonelli.\\
		Finally $(iv)$ holds true because  $\mathcal{H}_{K}$ is a reproducing kernel Hilbert space, so $\langle K_{\eta}, K_{\eta}\rangle_{\mathcal{H}_{K}}=0 $ if and only if $K_{\eta}(y)=0$ for all $y \in U$, and in this case	
		$$
		0=\langle v, K_{\eta}(y) \rangle_{\mathcal{H}} = \langle K_{y}v, K_{\eta} \rangle_{\mathcal{H}_{K}}= \sum_{|\alpha|\leq q} \int\langle \partial_{1}^{\alpha}K(x,y)v, d\eta_{\alpha}(x) \rangle	
		$$
		and the conclusion follows from Lemma \ref{addstepunidiff}. 
	\end{proof}

	Before proving Theorem \ref{Cpuniversalityintegralkernel}, we need a complex valued version of it.

	\begin{thm}\label{Cpuniversalityintegralkernelcomplvalue}Let $U \subset \mathbb{R}^{m}$ be an open set, $\Omega$  be a locally compact Hausdorff space and $p : \Omega \times X \times X \to \mathbb{C} \in C^{0,q,q}(\Omega\times U\times U)$, such that the kernel
		$$
		(x,y) \in U\times U\to p_{\omega}(x,y):=p(\omega, x,y)
		$$ 
		is positive definite for every $w \in \Omega$. Given a  scalar valued nonnegative finite Radon measure $\lambda$ on $\Omega$, consider the  kernel
		$$	
		P : U\times U \to \mathbb{C}, \quad P(x,y)= \int_{\Omega}p_{w}(x,y) d\lambda(w). 	  	
		$$
		Then if $p$ is $C^{q}$-dominated with respect to $\lambda$,  we have that:
		\begin{enumerate}	
			\item[(i)] The kernel $P \in C^{q,q}(U\times U)$ and  $\mathcal{H}_{K} \subset C^{q}(U)$.
			\item[(ii)] If the kernel  $p_{\omega}$ is $C^{q}$-universal for every $w \in \Omega$, then the  kernel $P$ is $C^{q}$-universal if and only if the measure $\lambda$ is nonzero.
			\item[(iii)] If the kernel  $p_{\omega}$ is $C^{q}$-integrally strictly positive definite for every $w \in \Omega$, then the  kernel $P$ is $C^{q}$-integrally strictly positive definite if and only if the measure $\lambda$ is nonzero. 
		\end{enumerate}
	\end{thm}
	
	\begin{proof}
		Relation $(i)$ is an application of Lemma \ref{diffkernelfourtrans} to the complex valued case.\\ 
		Now, we prove $(ii)$. Let $\mathcal{A} \subset U$ be a compact set for which $\overline{Int(\mathcal{A})}=\mathcal{A}$ and $\eta= (\eta_{\alpha})_{|\alpha| \leq q} \in \mathfrak{M}^{q}(\mathcal{A})$ for which the linear functional 
		$$
		\phi \in C^{q}(\mathcal{A}) \to \sum_{|\alpha|\leq q}  \int_{\mathcal{A}}\partial^{\alpha}\phi(x)d\eta_{\alpha}(x) \in \mathbb{C}
		$$
		is nonzero. Since $p_{w}$ is $C^{q}$-universal for every $w \in \Omega$, by the complex  valued version of Theorem \ref{oneintegraltodoubleintegralcq} (proved in \cite{Simon-Gabriel2018}),
		we have that $$\sum_{|\alpha|, |\beta| \leq q}\int_{\mathcal{A}} \int_{\mathcal{A}} \partial^{\alpha}_{1}\partial_{2}^{\beta}p_{w}(x,y)d\eta_{\beta}(y)d\overline{\eta_{\alpha}}(x) >0$$ for all $w\in \Omega$. So,  $\lambda$ is a nonzero measure if and only if  
		\begin{align*}
		\sum_{|\alpha|, |\beta| \leq q}\int_{\mathcal{A}} \int_{\mathcal{A}} \partial^{\alpha}_{1}\partial_{2}^{\beta}P(x,y)&d\eta_{\beta}(y)d\overline{\eta_{\alpha}}(x)
		= \sum_{|\alpha|, |\beta| \leq q}\int_{\mathcal{A}} \int_{\mathcal{A}}\int_{\Omega} \partial^{\alpha}_{1}\partial_{2}^{\beta}p_{w}(x,y)d\lambda(w)d\eta_{\beta}(y)d\overline{\eta_{\alpha}}(x)\\
		&=\int_{\Omega}   \left[ \sum_{|\alpha|, |\beta| \leq q} \int_{\mathcal{A}} \int_{\mathcal{A}} \partial^{\alpha}_{1}\partial_{2}^{\beta}p_{w}(x,y)d\eta_{\beta}(y)d\overline{\eta_{\alpha}}(x) \right ]d\lambda(w)>0
		\end{align*}
		which proves our claim.\\
		The proof of $(iii)$ is identical to the proof of $(ii)$.
	\end{proof}

	\begin{proof}[\textbf{Proof of Theorem \ref{Cpuniversalityintegralkernel}}]
		By Lemma \ref{diffkernelfourtrans}, not only the kernel $P$ is well defined and differentiable but we also have that $\partial^{\alpha}_{1}\partial_{2}^{\beta}P(x,y)= \int_{\Omega}\partial^{\alpha}_{1}\partial_{2}^{\beta}p_{w}(x,y)d\Lambda(w)$. The fact that $\mathcal{H}_{P} \subset C^{q}(U, \mathcal{H})$ is a consequence of Lemma \ref {diffimersion}. This settles $(i)$.\\
		The proof of $(ii)$ is very similar to the proof of $(iv)$ in Theorem \ref{universalityintegralkernel}. The measure $\Lambda_{v}$ is nonzero if and only if the kernel $P_{v}$ is $C^{q}$-universal  by Theorem \ref{Cpuniversalityintegralkernelcomplvalue}.  Now suppose that the measure $\Lambda_{v}$ is nonzero  for every $v \in \mathcal{H}\setminus{\{0\}}$  and let  $\mathcal{A}\subset U$ be such that $\overline{Int(\mathcal{A})}= \mathcal{A}$ and a measure $\eta= (\eta_{\alpha})_{|\alpha|\leq q} \in \mathfrak{M}^{q}(\mathcal{A}, \mathcal{H})$ for which
		$$
		\sum_{|\alpha|, |\beta|\leq q}\int_{\mathcal{A}}\langle \int_{\mathcal{A}} \langle \partial_{1}^{\alpha} \partial_{2}^{\beta}P(x,y)H_{\beta}(y), H_{\alpha}(x)\rangle_{\mathcal{H}}d|\eta|(x), d|\eta|(y)  =0 \in \mathbb{C}.
		$$
		Note that	
		%	must define a zero linear functional on $C^{q}(\mathcal{A}, \mathcal{H})$ by
		%	$$
		%	F \in C^{q}(\mathcal{A}, \mathcal{H}) \to 	\sum_{|\alpha|\leq q}\int_{\mathcal{A}} \langle \partial^{\alpha}F(x),H_{\alpha}(x)\rangle_{\mathcal{H}}d|\eta|(x)= 	\sum_{|\alpha|\leq q} \int_{\mathcal{A}} \langle \partial^{\alpha}F(x),d\eta_{\alpha}(x)\rangle. 
		%	$$
		% There exists $|\eta|$ a finite nonnegative Radon complex valued measure in $U$ of compact support and Bochner integrable functions $H_{\alpha } : \Omega \to \mathcal{H}$, for which $H_{\alpha}d|\eta|= \eta_{\alpha}$. Then, by Lemma \ref{manipulation} and Lemma  \ref{diffkernelfourtrans}, we have that
		\begin{align*} 
		0&=\sum_{|\alpha|, |\beta|\leq q}\int_{\mathcal{A}} \int_{\mathcal{A}} \langle \partial_{1}^{\alpha} \partial_{2}^{\beta}P(x,y) H_{\beta}(y),H_{\alpha}(x) \rangle_{\mathcal{H}}d|\eta|(x) d|\eta|(y)\\
		&=\sum_{|\alpha|, |\beta|\leq q}\int_{\mathcal{A}} \int_{\mathcal{A}}\int_{\Omega}\partial_{1}^{\alpha} \partial_{2}^{\beta}p_{w}(x,y) \langle G(w) H_{\beta}(y),H_{\alpha}(x) \rangle_{\mathcal{H}}d\lambda(w)d|\eta|(x) d|\eta|(y)\\
		&=\int_{\Omega}\sum_{|\alpha|, |\beta|\leq q}\int_{\mathcal{A}} \int_{\mathcal{A}}\partial_{1}^{\alpha} \partial_{2}^{\beta}p_{w}(x,y) \langle G(w) H_{\beta}(y),H_{\alpha}(x) \rangle_{\mathcal{H}}d|\eta|(x) d|\eta|(y)d\lambda(w)       
		\end{align*}
		And again, the change in the order of integration is possible because 
		$$|\partial_{1}^{\alpha} \partial_{2}^{\beta}p_{w}(x,y) \langle G(w) H_{\beta}(y),H_{\alpha}(x) \rangle_{\mathcal{H}}|\leq |h(w)|\|G(w)\|_{\mathcal{L}(\mathcal{H})}\|H(x)\|_{\mathcal{H}} \|H(y)\|_{\mathcal{H}}
		$$ 
		which is $\lambda\times |\eta| \times |\eta|$ integrable. Since the matrix valued kernel $(x,y)\to [\partial_{1}^{\alpha} \partial_{2}^{\beta}p_{w}(x,y)]_{|\alpha|, |\beta|\leq q} $ is positive definite, applying Lemma \ref{almostaschurproduct} with respect to this kernel, we get that  for all $w \in \Omega$
		$$\sum_{|\alpha|, |\beta|\leq q}\int_{\mathcal{A}} \int_{\mathcal{A}}\partial_{1}^{\alpha} \partial_{2}^{\beta}p_{w}(x,y) \langle G(w) H_{\beta}(y),H_{\alpha}(x) \rangle_{\mathcal{H}}d|\eta|(x) d|\eta|(y)\geq 0.$$
		Since $\lambda$ is nonnegative, we must have that this double integral is zero $\lambda$ almost everywhere. Let $G^{1/2}(w)$ be the unique square root of the positive semidefinite operator $G(w)$ that is positive semidefinite, then by adding coordinates	
		\begin{align*}
		0&=\sum_{|\alpha|, |\beta|\leq q}\int_{\mathcal{A}} \int_{\mathcal{A}}\partial_{1}^{\alpha} \partial_{2}^{\beta}p_{w}(x,y) \langle G(w) H_{\beta}(y),H_{\alpha}(x) \rangle_{\mathcal{H}}d|\eta|(x) d|\eta|(y)\\
		&=\sum_{|\alpha|, |\beta|\leq q}\int_{\mathcal{A}} \int_{\mathcal{A}}\partial_{1}^{\alpha} \partial_{2}^{\beta}p_{w}(x,y) \langle G^{1/2}(w)H_{\beta}(y), G^{1/2}(w)H_{\alpha}(x) \rangle_{\mathcal{H}}d|\eta|(x) d|\eta|(y)\\
		&=\sum_{|\alpha|, |\beta|\leq q}\int_{\mathcal{A}} \int_{\mathcal{A}}\sum_{\mu \in \mathcal{I}}\partial_{1}^{\alpha} \partial_{2}^{\beta}p_{w}(x,y) (G^{1/2}(w)H_{\beta}(y))_{\mu}\overline{(G^{1/2}(w)H_{\alpha}(x))_{\mu}} d|\eta|(x) d|\eta|(y)\\
		&= \sum_{|\alpha|, |\beta|\leq q}\sum_{\mu \in \mathcal{I}}\int_{\mathcal{A}} \int_{\mathcal{A}} \partial_{1}^{\alpha} \partial_{2}^{\beta}p_{w}(x,y) (G^{1/2}(w)H_{\beta}(y))_{\mu}\overline{(G^{1/2}(w)H_{\alpha}(x))_{\mu}} d|\eta|(x) d|\eta|(y)
		\end{align*}
		and once again by Lemma \ref{almostaschurproduct},  
		$$
		\sum_{|\alpha|, |\beta|\leq q}\int_{\mathcal{A}} \int_{\mathcal{A}} \partial_{1}^{\alpha} \partial_{2}^{\beta}p_{w}(x,y) (G^{1/2}(w)H_{\beta}(y))_{\mu}\overline{(G^{1/2}(w)H_{\alpha}(x))_{\mu}} d|\eta|(x) d|\eta|(y) \geq 0
		$$
		for all $\mu \in \mathcal{I}$ and $w \in \Omega$. Then, we must have that
		$$
		\int_{\mathcal{A}} \int_{\mathcal{A}}\sum_{|\alpha|, |\beta|\leq q} \partial_{1}^{\alpha} \partial_{2}^{\beta}p_{w}(x,y) (G^{1/2}(w)H_{\beta}(y))_{\mu}\overline{(G^{1/2}(w)H_{\alpha}(x))_{\mu}} d|\eta|(x) d|\eta|(y) = 0
		$$
		for all $\mu \in \mathcal{I}$ and $\lambda$ almost everywhere on $w$. But, $p_{w}$ is a $C^{q}$-universal kernel for all $w \in \Omega$ and  $\overline{(G^{1/2}(w)H_{\alpha}(x))_{\mu}}d|\eta|(x))_{|\alpha|\leq q} \in \mathfrak{M}^{q}(\mathcal{A})$, so the following linear functional is the zero functional for every $\mu \in \mathcal{I}$ 
		$$
		f \in C^{q}(\mathcal{A})  \to  \sum_{|\alpha| \leq q} \int_{\mathcal{A}}\partial^{\alpha}f(x)\overline{(G^{1/2}(w)H_{\alpha}(x))_{\mu}}d|\eta|(x) \in \mathbb{C}.  
		$$
		This implies that
		$$
		\sum_{|\alpha|, |\beta| \leq q} \int_{\mathcal{A}}\int_{\mathcal{A}}\partial^{\alpha}f(x)\overline{\partial^{\beta}f(y)}\langle G(w)H_{\beta}(y), H_{\alpha}(x) \rangle_{\mathcal{H}}d|\eta|(x)d|\eta|(y)=0  
		$$
		for $\lambda$ almost every $w$. Integrating this equality on the variable $w$ with respect to $\lambda$ we get that
		\begin{align*}
		0&=\sum_{|\alpha|, |\beta| \leq q} \int_{\mathcal{A}}\int_{\mathcal{A}}\partial^{\alpha}f(x)\overline{\partial^{\beta}f(y)}\langle \Lambda(\Omega)H_{\beta}(y), H_{\alpha}(x) \rangle_{\mathcal{H}}d|\eta|(x)d|\eta|(y)\\
		& =\langle \Lambda(\Omega) [\sum_{|\alpha| \leq q}  \int_{\mathcal{A}}\overline{\partial^{\alpha}f(x)}d\eta_{\alpha}(x)], \sum_{|\alpha| \leq q}  \int_{\mathcal{A}}\overline{\partial^{\alpha}f(x)}d\eta_{\alpha}(x) \rangle_{\mathcal{H}}.
		\end{align*}  
		Now we are at the final steps of the proof. Since $\Lambda(\Omega)$ is a positive definite operator, we must have that  
		$$
		\sum_{|\alpha| \leq q}  \int_{\mathcal{A}}\overline{\partial^{\alpha}f(x)}d\eta_{\alpha}(x)=0 \in \mathcal{H}.
		$$
		But then, by Lemma \ref{manipulation}, note that for every $F \in C^{q}(\mathcal{A}, \mathcal{H})$ we have that
		$$
		\sum_{|\alpha|\leq q}\int_{\mathcal{A}}\langle \partial^{\alpha} F(x), d\eta_{\alpha}(x) \rangle=\sum_{\mu \in \mathcal{I}}    \sum_{|\alpha|\leq q}\int_{\mathcal{A}} \partial^{\alpha} (F(x))_{\mu} \overline{(H_{\alpha}(x))_{\mu}}d|\eta|(x)=0,
		$$ 
		which proves that the kernel $P$  is $C^{q}$-universal. The remaining implication is a consequence of Lemma \ref{scalarvaluedprojectionsuniversality}. This settles $(ii)$.\\
		%Trivially, the operator $\Lambda(\Omega)$ is positive definite if and only if the measures $\Lambda_{v}$ are nonzero, for all $v \in \mathcal{H}\setminus{\{0\}}$.\\
		As for  $(iii)$, since we are assuming that the function $p$ is $C^{q}$-Dominated with respect to $\|G\|_{\mathcal{L}(\mathcal{H})}d\lambda$, we have that
		\begin{align*}
		\|\partial_{1}^{\alpha}\partial_{2}^{\alpha} P(x,y)\|_{\mathcal{L}(\mathcal{H})}&= \|\int_{\Omega}\partial_{1}^{\alpha}\partial_{2}^{\alpha}p_{w}(x,y)d\Lambda(w)\|_{\mathcal{L}(\mathcal{H})}\\
		&\leq \int_{\Omega}|h(w)|\|G(w)\|_{\mathcal{L}(\mathcal{H})}d\Lambda(w)<\infty,
		\end{align*}
		in other words, the kernels $\partial_{1}^{\alpha}\partial_{2}^{\alpha} P(x,y)$ are bounded in $\mathcal{L}(\mathcal{H})$ . The remaining  arguments are just an adaptation of the arguments we presented in $(ii)$.
	\end{proof}

	\begin{proof}[\textbf{Proof of Corollary  \ref{Cpuniversalityintegralkernelrestricmeasure}}]
		The kernel $P$ is well defined and positive definite by Proposition \ref{compcsuppor}. Moreover, by Lemma \ref{radonnikodintrace} there exists a scalar valued nonnegative finite Radon measure $\lambda_{T}$ on $(0,\infty)$ and a Bochner integrable function $G: (0,\infty)\to \mathcal{L}(\mathcal{H})$ with respect to $\lambda_{T}$ for which $G(w)$ is a positive semidefinite operator for all $w \in (0,\infty)$ and 
		$$
		P(x,y)= \int_{(0,\infty)}f(w(x-y))G(w)d\lambda_{T}(w).
		$$
		By the proof of Lemma \ref{radonnikodintrace},    
		$$
		Tr(P(x,y))= \int_{(0,\infty)}f(w(x-y))Tr(G(w))d\lambda(w)= \int_{(0,\infty)}f(w(x-y))d\lambda(w).	
		$$	
		Since $Tr(P) \in C^{q,q}(\mathbb{R}^{m})$, Lemma \ref{diffform} implies that $\int_{(0,\infty)}w^{2q}d\lambda(w)< \infty$. In particular, $p_{w}(x,y):=f(w(x-y))$ is $C^{q}$-dominated with respect to $\|G\|d\lambda_{T}= d|\Lambda|$, by choosing $h(w)= Mw^{2q}$, for a suitable number $M>0$. Then, relation $(i)$, $(ii)$ and $(iii)$ are a directly application of Theorem \ref{Cpuniversalityintegralkernel}. \end{proof}

	\begin{proof}[\textbf{Proof of Theorem \ref{universaleuclidian2cq}}]
		It is a direct application of Corollary \ref{Cpuniversalityintegralkernelrestricmeasure}, to the case $f(x)= e^{-\|x\|^{2}}$ which defines a $C^{\infty}$-integrally strictly positive definite kernel and Theorem \ref{universaleuclidian2}.\end{proof}
	
	\begin{proof}[\textbf{Proof of Theorem \ref{universaleuclidian2cqaskey}}]
		It is a direct application of Corollary \ref{Cpuniversalityintegralkernelrestricmeasure}, to the case $f(x)= (1-\|x\|)_{+}^{\ell-1}\in C^{(\ell -2)}(\mathbb{R}^{m})$ which defines a $(\ell-2)/2$ integrally strictly positive definite kernel by Corollary $38$ in \cite{Simon-Gabriel2018}.   
	\end{proof}
	
	\begin{proof}[\textbf{Proof of Theorem \ref{universaleuclidian32}}]It is a direct application of Corollary \ref{Cpuniversalityintegralkernelrestricmeasure}, to the case $ x\in \mathbb{R}^{m-1}\to f(x)= \Omega_{m}^{m-1}(\|x\|) \in \mathbb{C}$ which defines a $C^{\infty}$-universal kernel by Lemma \ref{radialfini}. \end{proof}

	%\section{Appendix: Measures and integrals related to  Hilbert spaces}\label{Appendix}
	
	\appendix
	\counterwithin{thm}{section} % number the theorems and definitions on this section with the same alleters as the appendix. Must be used together with \usepackage{chngcntr} at the preamble
	\section{\textbf{Appendix: Measures and integrals related to  Hilbert spaces}}\label{Appendix}
	In this appendix we review some definitions regarding: 
	\begin{enumerate}
		\item[$\circ$]  Integration of vector valued functions with respect to complex valued measures.
		\item[$\circ$] Integration of complex valued functions with respect to vector valued measures.
		\item[$\circ$] Integration of Hilbert valued functions with respect to Hilbert valued measures.
		\item[$\circ$]  Integration of operator valued  functions with respect to Hilbert valued measures.
	\end{enumerate}
	The main reference is \cite{vectormeasures}, more specifically, section $I$ and $II$, but we adapt some terminologies to fit our setting,  like the definition of finite nonnegative operator valued measure. We fix throughout this section that  $(\Omega, \mathscr{B})$ is a sigma algebra,   $\lambda : (\Omega, \mathscr{B}) \to \mathbb{C}$ is a complex valued finite measure, $\mathfrak{B}$ is a  Banach space and $\mathfrak{B}^{*}$ is the Banach space of all continuous linear functionals defined on $\mathfrak{B}$. 
	
	First, we focus on integrals of vector valued functions over complex valued measures. Roughly speaking, there are $3$ ways for defining this integral, two of them having a functional analysis approach.
	
	\begin{defn}\textbf{(Pettis integral)} A function $F: \Omega \to \mathfrak{B}$ is \textbf{weak measurable}, if for every $v \in \mathfrak{B}^{*}$ the complex valued function $ w \in \Omega \to F_{v}:=  ( v, F(w))_{(\mathfrak{B}^{*}, \mathfrak{B})}  \in \mathbb{C}$ is a measurable function and  $F$ is \textbf{weak integrable} (or \textbf{Pettis integrable}) with respect to the measure $\lambda $, if  all functions $F_{v}$ are integrable and  for every $A \in  (\Omega, \mathscr{B})$ the following linear functional on $\mathfrak{B}^{*}$ is continuous
		$$	
		v \in \mathfrak{B}^{*}	\to \int_{A}( v, F(w))_{(\mathfrak{B}^{*}, \mathfrak{B})}  d\lambda(w) \in \mathbb{C}. 
		$$
	\end{defn}
	\begin{defn} \textbf{(Gelfand integral)} If $\mathfrak{B}= \mathfrak{D}^{*}$, for some Banach space $\mathfrak{D}$, a function $F: \Omega \to \mathfrak{B}$ is \textbf{weak$^{*}$ measurable}, if for every $v \in \mathfrak{D}$ the complex valued function $w \in \Omega \to F_{v}:=(F(w),v)_{\mathfrak{B}, \mathfrak{D}} \in \mathbb{C}$ is measurable, and that $F$ is \textbf{weak$^{*}$ integrable }(or \textbf{Gelfand integrable}) with respect to $\lambda$ if for every $(\Omega, \mathscr{B})$ the following linear functional on $\mathfrak{D}$ is continuous
		$$	
		v \in \mathfrak{D}	\to \int_{A}(  F(w),v)_{(\mathfrak{B}, \mathfrak{D})}  d\lambda(w) \in \mathbb{C}. 
		$$
	\end{defn}
	Note that the Pettis integral is an element in $\mathfrak{B}^{**}$ while the Gelfand integral is an element in $\mathfrak{B}$. When the Banach space is reflexive, the Gelfand and Pettis integrals are the same object.
	%$$
	%\int_{A}( v, F(w))_{(\mathfrak{B}^{*}, \mathfrak{B})}  d\lambda(w)= ( v,  \int_{A}F(w) d\lambda(w) )_{(\mathfrak{B}^{*}, \mathfrak{B})}, \quad v \in \mathfrak{B}^{*}, \quad A \in  (\Omega, \mathscr{B}).
	%$$

	\begin{defn}\textbf{(Bochner integral)} A function $\phi : \Omega \to \mathfrak{B}$ is called a \textbf{simple function} if there exists a finite amount of $b_{1}, \ldots , b_{n} \in \mathfrak{B}$ and   measurable sets $A_{1}, \ldots , A_{n} \in (\Omega, \mathscr{B})$ for which 
		$$
		\phi(x)= \sum_{i=1}^{n}b_{i}\chi_{A_{i}}(x). 
		$$
		A function $F: \Omega \to \mathfrak{B}$ is \textbf{Bochner measurable} with respect to $\lambda$ if there exists a sequence of simple functions $F_{n} : \Omega \to \mathfrak{B}$ for which $\lim_{n \to \infty}\|F(w)-F_{n}(w)\|_{\mathfrak{B}}=0$, $|\lambda|$ almost everywhere. The function $F$ is \textbf{Bochner integrable} with respect to  $\lambda $, if there exists a sequence of simple functions $F_{n} : \Omega \to \mathfrak{B}$ for which 
		$$\lim_{n \to \infty}\int_{\Omega}\|F(w)-F_{n}(w)\|_{\mathfrak{B}} d|\lambda(w)|=0 
		$$
	\end{defn}
	Under the setting on the previous definition, the sequence $\int_{\Omega}F_{n}(w)d\lambda(w)$ in $\mathfrak{B}$ is convergent, and we denote its value as $\int_{\Omega}F(w)d\lambda(w) \in \mathfrak{B}$, which is independent of the choice for the simple functions.

	The following famous result connects the concept of weak and Bochner measurability, page $42$ \cite{vectormeasures}.

	%It is easy to check that all simple functions are Pettis integrable, and following the previous notation $\int_{\Omega}\phi(w)d\lambda(w) =  \sum_{i=1}^{n}b_{i}\lambda(A_{i})$
	
	\begin{thm}(\textbf{Pettis Measurability Theorem}) \label{Pettismeastheorem} A function $F : \Omega \to \mathfrak{B}$ is Bochner measurable with respect to  $\lambda$ if and only if $F$ is weak measurable with respect to $\lambda$ and there exists $E \in (\Omega,  \mathscr{B})$ for which $|\lambda|(E)=0$ and the set $\{ F(w), w \in \Omega-E  \} \subset \mathfrak{B}$ is separable.
	\end{thm}
	
	Similar to complex valued functions (page $45$ in \cite{vectormeasures}), a Bochner measurable function $F: \Omega \to \mathfrak{B}$ is Bochner integrable if and only if
	\begin{equation}\label{bochnerintegration}
	\int_{\Omega}\|F(w)\|_{\mathfrak{B}}d|\lambda|(w)< \infty.
	\end{equation} 
	And because of this, every Bochner integrable function is weak (also weak$^{*}$, whenever is possible) integrable, and the value is the same.
	
	%In this article, we only use the Bochner integral, however Theorem  \ref{Pettismeastheorem} simplifies some arguments, for instance at Lemma \ref{radonnikodintrace}.  

	Now we focus on integration of complex valued functions with respect to finite vector valued measures. Similarly, there are $3$ main types for this type of integral, but on this setting it is related on how we define a vector valued measure.
	
	\begin{defn}\label{finitemeasdef} A  set  function $\Lambda : (\Omega, \mathscr{B}) \to \mathfrak{B}$ is a \textbf{finite vector valued measure} (\textbf{weak finite vector valued measure}) if:
		\begin{enumerate}
			\item $\Lambda(\emptyset)= 0 \in \mathfrak{B}$.
			\item There exist $M >0$ such that $\|\Lambda(A)\|_{\mathfrak{B}} \leq M$ for all $A \in \mathscr{B}$.
		\end{enumerate}
		If $(A_{n})_{n \in \mathbb{N}}$ is a countable  family of sets in $\mathscr{B}$ such that every two of them are disjoint, then it satisfies: 
		\begin{enumerate}
			\item[3.] $\displaystyle{\sum_{n \in \mathbb{N}} |(v, \Lambda(A_{n}))_{(\mathfrak{B}^{*},\mathfrak{B})}|< \infty \text{ for all } v \in \mathfrak{B}^{*}.}$
			\item[4.] $\displaystyle{\Lambda(\bigcup_{n \in \mathbb{N}} A_{n}) = \sum_{n \in \mathbb{N}} \Lambda(A_{n}) \text{ on the norm of $\mathfrak{B}$ (weak convergence ).}
			}$
		\end{enumerate}
	\end{defn}
	If $\mathfrak{B}= \mathfrak{D}^{*}$, for some Banach space $\mathfrak{D}$, a  set  function $\Lambda : (\Omega, \mathscr{B}) \to \mathfrak{B}$ is a \textbf{weak$^{*}$ finite vector valued measure} if it satisfies satisfies $1$ and $2$ at Definition \ref{finitemeasdef}, but $3$ and $4$ are replaced by
	\begin{enumerate}
		\item[$3^{\prime}.$] $\displaystyle{\sum_{n \in \mathbb{N}} |( \Lambda(A_{n}),v)_{(\mathfrak{B},\mathfrak{D})}|< \infty \text{ for all } v \in \mathfrak{D}.}$
		\item[$4^{\prime}.$] $\displaystyle{\Lambda(\bigcup_{n \in \mathbb{N}} A_{n}) = \sum_{n \in \mathbb{N}} \Lambda(A_{n}) \text{  (weak$^{*}$ convergence)}.}$
	\end{enumerate}

	Our definition of finite vector valued measure is based on the definition of countably additive vector valued measure from \cite{vectormeasures}. Note that if $\Lambda$ is a finite vector valued measure on $(\Omega, \mathscr{B})$, then by $2$ and $3$ for every $v \in \mathfrak{B}^{*}$ the function  $A \in (\Omega, \mathscr{B}) \to \Lambda_{v}(A):= (v,\Lambda(A))_{(\mathfrak{B}^{*}, \mathfrak{B})} \in \mathbb{C}$ defines a finite complex valued measure (usually, on a vector valued setting they do not impose this condition, but this consequence is highly important and by doing this we  do not have a conflict of definitions with  finite complex valued measures). A similar property holds for weak and weak$^{*}$ finite vector valued measures
	
	%Depending on the context, the strong convergence at Definition \ref{finitemeasdef}, might be changed to another type like weak convergence, weak$^{*}$ in case $\mathfrak{B}$ is the dual space of another Banach space, ultraweakly if $\mathfrak{B}= \mathcal{L}(\mathcal{H})$. For instance, the operator valued Bochner Theorem proved in \cite{neeb1998}, the   
	
	%\begin{defn} Let $f: \Omega \to \mathbb{C}$  be a $(\Omega, \mathscr{B})$ measurable function. The function $f$ is \textbf{weak integrable} with respect to  $\Lambda $, if  the function $f$  is $|\Lambda_{v}|$ integrable for every $v \in \mathfrak{B}^{*}$ and  the following linear functional on $\mathfrak{B}^{*}$ is continuous
	%	$$		v \in \mathfrak{B}^{*}	\to \int_{\Omega}f(w) d(v,\Lambda(w))_{[\mathfrak{B}^{*}, \mathfrak{B}]} \in \mathbb{C}. $$
	%\end{defn}*

	The integration of complex valued functions with respect to weak$^{*}$ finite vector valued measures is straightforward, a function $f: \Omega \to \mathbb{C}$ is weak$^{*}$ integrable with respect to $\Lambda$ if  the following  linear functional  on $\mathfrak{D}$ is continuous
	$$	
	v \in \mathfrak{D}	\to \int_{A}  f(w)d(\Lambda,v)_{(\mathfrak{B}, \mathfrak{D})}(w)   \in \mathbb{C}, 
	$$
	for every $A \in (\Omega, \mathscr{B})$. A similar definition for the weak integral is possible, but we do not use this integral in this article. However, the  weak$^{*}$ integral of complex valued functions occurs  naturally on the generalization of some complex valued kernels to the operator valued setting, for instance the operator valued Bochner Theorem at \cite{neeb1998} and at Lemma \ref{uniquelcomplmon}.

	On the special case that $\mathfrak{B}= \mathcal{L}(\mathcal{H})$, we say that a finite vector valued measure (weak$^{*}$ finite vector valued measure) is \textbf{nonnegative} if $\Lambda(A)$ is a positive semidefinite operator for every $A \in (\Omega, \mathscr{B})$. By the representation of trace class operators on $\mathcal{H}$, instead of analysing the weak$^{*}$ convergence of measures and integrals, it is sufficient to analyse the ultraweakly convergence. 
	
	%$When the Banach space is reflexive, if a function $f$ is weak integrable there exists an unique vector in $\mathfrak{B}$, which we denote by $$\int_{\Omega}f(w)\Lambda(w)$, such that
	%$$\int f(w)d(v,\Lambda(w))_{[\mathfrak{B}^{*}, \mathfrak{B}]} = (v, \int_{\Omega}f(w)\Lambda(w))_{[\mathfrak{B}^{*}, \mathfrak{B}]}$$

	Unlike finite complex valued measures, finite vector valued measures can behave on an unwanted way. One class of measures that has a more appealing  behavior  are those with  bounded variation.
	
	\begin{defn} Let $\Lambda: (\Omega, \mathscr{B}) \to \mathfrak{B} $ be a finite vector valued measure. We say that this finite measure has \textbf{bounded variation} if exists $M>0$ such that for every measurable set $A$, and every disjoint sequence of measurable sets $(A_{j})_{j \in \mathbb{N}}$ for which $A= \cup_{j \in \mathbb{N}}A_{j}$, we have that
		$$	
		\sum_{j \in \mathbb{N}}\| \Lambda(A_{j})\|_{\mathfrak{B}} \leq M.	 
		$$	
		And if that occurs, we define the set function $|\Lambda| : (\Omega, \mathscr{B}) \to [0, \infty) $, called the \textbf{variation} of $\Lambda$, by
		$$
		|\Lambda|(A)= \sup_{\pi(A)}  \sum_{j \in \mathbb{N}}\| \Lambda(A_{j})\|_{\mathfrak{B}} 
		$$ 
		where $\pi(A)$ stands for the set of partitions of the set $A$ into  countable measurable disjoint sets. 	
	\end{defn}

	From now on to simplify the notation, we omit the term finite when dealing with measures with bounded variation. On a measure with bounded variation, its variation is a complex valued nonnegative finite measure, \cite{vectormeasures} Chapter $1$, and with it it is possible to define a norm over the vector space of $\mathfrak{B}$ valued measures of bounded variation defined on $(\Omega, \mathscr{B})$, which we denote by  $\mathfrak{M}(\Omega, \mathscr{B}, \mathfrak{B})$ and the norm is $\| \Lambda_{1} - \Lambda_{2}\|_{\mathfrak{M}(\Omega, \mathscr{B}, \mathfrak{B})}:= |\Lambda_{1} - \Lambda_{2}|(\Omega)$.
	
	When $\mathfrak{B}= \mathbb{C}$, this definition of bounded variation and the variation measure agrees with the standard definition, \cite{folland} Chapter $3$.

	After this comments, we are able to define the integral of a complex valued function with respect to a  vector measure of bounded variation,  in a similar that we defined the Bochner integration.
	
	\begin{defn} Let $f: \Omega \to \mathbb{C}$ be a measurable function. The function $f$ is integrable with respect to a vector measure with bounded variation $\Lambda$, if $f$ is integrable with respect to the finite complex valued measure $|\Lambda|$. In that case, if $f_{n}: \Omega \to \mathbb{C}$ is a  sequence of simple functions, for which $\lim_{n\to \infty}|f_{n} - f|_{\mathbb{C}}=0$, $|\Lambda|$ almost everywhere then $\int_{\Omega}f_{n}(w)d\Lambda(w)$ is a convergent sequence in $\mathfrak{B}$, it is independent of the simple functions choosen, and we denote its value by $\int_{\Omega}f(w)d\Lambda(w)$.
	\end{defn}

	In some sense, the Bochner integral of vector valued functions with respect to complex valued measures is more technically advantageous then the integral of complex valued functions with respect to vector valued measures of bounded variation. If $F: \Omega \to \mathbb{B}$ is a Bochner integrable function with respect to $\lambda$, then the set function
	$$
	A \in (\Omega, \mathscr{B}) \to \Lambda_{F}(A):=\int_{A}F(w)d\lambda(w) \in \mathfrak{B}
	$$
	is a vector measure of bounded variation in $(\Omega, \mathscr{B})$, more precisely\\ $|\Lambda_{F}|(A)= \int_{A}\|F(w)\|_{\mathfrak{B}}d\|\lambda|(w)$, page $46$ at \cite{vectormeasures}. The following definition aims the opposite, it is the terminology for  which vector measures of bounded variation are defined by a Bochner integral.
	
	\begin{defn}\label{Radon Nykodin Property} Let  $\Lambda: (\Omega, \mathscr{B})\to \mathfrak{B}$ be a measure with bounded variation. We say that $\Lambda$ admits \textbf{Radon-Nikod\'ym decomposition} if there exists a Bochner measurable function $G: (\Omega, \mathscr{B})\to \mathfrak{B} $ that is Bochner integrable with respect to a finite nonnegative measure $\lambda : (\Omega, \mathscr{B})\to \mathbb{C}$  of bounded variation, for which
		$$
		\Lambda(A)= \int_{A}G(w)d\lambda(w).
		$$
	\end{defn}  
	
	On reflexive spaces (in particular, Hilbert spaces), all measures of bounded variation admits a Radon-Nikod\'ym decomposition with respect to its variation measure. This is a consequence of the famous Radon-Nikod\'ym property, \cite{vectormeasures} section $III$, which is a property on the Banach space $\mathfrak{B}$ rather than a particular measure, in the sense that a Banach space $\mathfrak{B}$ satisfy the  Radon-Nikod\'ym property if all $\mathfrak{B}$ valued measures of bounded variation admits a Radon-Nikod\'ym decomposition, with $\lambda$ being the variation of the measure. 
	
	The Banach space $\mathfrak{B}= c_{0}(\mathbb{N})$  does not satisfy the Radon Nykodin Property as shown in \cite{vectormeasures} page $60$. The Banach space $\mathcal{L}(\mathcal{H})$ also does not satisfy it, as shown  in \cite{zhang2012}, using  a similar argument as the $c_{0}(\mathbb{N})$ case. The author does not know if there exists a nonnegative   $\mathcal{L}(\mathcal{H})$-valued measure with bounded variation that does not admits a Radon-Nikod\'ym decomposition.

	A finite nonnegative operator valued measure might not have bounded variation. Indeed, if $(v_{\mu})_{ \mu \in \mathbb{N}} \in c_{0}(\mathbb{N}) \setminus \ell^{1}(\mathbb{N})$ and all coefficients $v_{\mu}$ are positive, then the weak$^{*}$finite operator valued measure $\Lambda : \mathbb{N} \to \mathcal{L}(\ell^{2}(\mathbb{N}))$ given by $\Lambda(\{\mu\})= v_{\mu}e_{\mu}^{*}e_{\mu}$ is nonnegative  but it is not of bounded variation. In particular, there is a gap between the operator valued Bochner Theorem at \cite{neeb1998} and  relation $(i)$ in Theorem \ref{universalityintegralkernel} (where $p_{w}(x,y)= e^{-i(x-y)w}$, $\Omega= X= \mathbb{R}^{m}$), being our setting more restrictive, since we are assuming that the  nonnegative operator valued measure admits a Radon-Nikod\'ym decomposition.

	Next, we clarify what happens on a nonnegative operator valued measure of bounded variation that admits a Radon-Nikod\'ym decomposition and we also explain why we focus on separable Hilbert spaces.  
	
	\begin{prop}\label{simplifyoperatornonegativemeasure}Let  $\Lambda: (\Omega, \mathscr{B})\to \mathcal{L}(\mathcal{H})$ be a nonnegative measure with bounded variation and a Radon-Nikod\'ym decomposition $d\Lambda= Gd\lambda$. Then for every $v \in \mathcal{H}$,  the function $\langle G(w)v,v \rangle_{\mathcal{H}}\geq 0$ almost everywhere on  $\lambda$. In particular, if $\mathcal{H}$  is separable, we can assume that   the operator $G(w)$ is positive semidefinite  for all $w \in \Omega$.
	\end{prop}
	\begin{proof}
		Indeed, since $\int_{A}\langle G(w)v,v \rangle_{\mathcal{H}} d\lambda(w)= \langle \Lambda(A)v,v \rangle_{\mathcal{H}}$ (we used the fact that Bochner integral implies Pettis integral) then for every $v \in \mathcal{H}$ the function $\langle G(w)v,v \rangle_{\mathcal{H}}\geq 0$ almost everywhere on  $\lambda$. \\If $\mathcal{H}$ is separable, choose a sequence $v_{k} \in \mathcal{H}$, $k \in \mathbb{N}$, that is dense on it. Define
		$$
		E_{v}:= \{ w \in \Omega, \quad \langle G(w)v,v \rangle_{\mathcal{H}} \notin[0,\infty )  \}
		$$
		Each set $E_{v}$ has $\lambda$ measure $0$. Because $\{v_{k}\}_{k \in \mathbb{N}}$ is dense, we have that
		$$
		\bigcup_{v \in \mathcal{H}}E_{v}= \bigcup_{k \in \mathbb{N}}E_{v_{k}}, \quad  \text{ and then } \lambda(\bigcup_{v \in \mathcal{H}}E_{v})=0
		$$
		which implies that $\{w \in \Omega, \quad G(w) \text{ is not positive semidefinite}\}$ has $\lambda$ zero measure. Because of that, redefine  $G(w)$ as the zero operator on $\bigcup_{v \in \mathcal{H}}E_{v}$ (which do not change the outcome of the integral), and then we obtain that $G(w)$ is positive semidefinite for every $w \in \Omega$.  
	\end{proof}

	Next we deal with integration of vector valued functions with respect to vector valued measures. We only present the definition and results for Hilbert spaces,  the general definition and properties can be found at \cite{bartle}. It is sufficient for our purposes to define only one type of this integral. 
	
	\begin{defn}\label{intvecfunvecmec} Let  $\eta: (\Omega, \mathscr{B})\to \mathcal{H}$ be a measure with bounded variation and $\phi: \Omega  \to \mathcal{H}$, $ \phi=\sum_{i=1}^{n}v_{i} \chi_{A_{i}}$    be a simple function. Then, we define
		$$	
		\int_{\Omega}\langle \phi(w),d\eta(w)\rangle= \sum_{i=1}^{n}\langle v_{i}, \eta(A_{i}) \rangle_{\mathcal{H}}  \in \mathbb{C}.
		$$
		A  Bochner measurable function $F: \Omega \to \mathcal{H}$ is   \textbf{Bartle integrable} with respect to $\eta$  if it is Bochner integrable with respect to $|\eta|$, and we define $\int_{\Omega}\langle F(w),d\eta(w)\rangle$ as the limit of $\int_{\Omega}\langle F_{n}(w),d\eta(w)\rangle$, where $F_{n}$ are simple functions for which
		$\int_{\Omega}\|F(w) -F_{n}(w)\|_{\mathcal{H}}d|\eta|(w) \to 0$.\end{defn}

	This definition  is independent from the sequence of simple functions that converges to $F$. If $d\eta=Hd\lambda$ is a Radon-Nikod\'ym decomposition for a vector measure of bounded variation $\eta$ (which always exists by the Radon-Nikod\'ym of $\mathcal{H}$), it is possible to prove that  a Bochner measurable function $F$ is integrable with respect to $\eta$ if and only if $$\int_{\Omega}\|F(w)\|_{\mathcal{H}}\|H(w)\|_{\mathcal{H}}d\lambda(w)<\infty.$$

	The final integral  we define is of  operator valued functions with respect to a vector measure of bounded variation. The definition that is relevant for our purposes is a mixed type integral, involving a Bochner and a Pettis integrability.

	\begin{defn}\label{weak-bochner} Let  $\eta: (\Omega, \mathscr{B})\to \mathcal{H}$ be a measure with bounded variation and $T : \Omega \to  \mathcal{L}(\mathcal{H})$, a function such that for every $v \in \mathcal{H}$, $Tv : \Omega \to  \mathcal{L}(\mathcal{H})$ is Bochner measurable and integrable with respect to $\eta$. We say that $T$ is \textbf{weak-Bochner integrable} with respect to $\eta$ if the following linear operator is continuous
		$$
		v \in \mathcal{H} \to \int_{\Omega}\langle T(w)v, d\eta(w)\rangle \in \mathbb{C}  
		$$
		and in that case, the value of the integral is the unique vector, which we denote by $\int_{\Omega} T(w) d\eta(w)$, that satisfies
		$$
		\langle v, \int_{\Omega} T(w) d\eta(w) \rangle_{\mathcal{H}}= \int_{\Omega}\langle T(w)v, d\eta(w)\rangle, \quad v \in \mathcal{H}
		$$
	\end{defn}
	
	This integral is used as a middle step in order to go from the characterization of universal kernels at Theorem $11$ in \cite{Caponnetto2008} to Theorem \ref{oneintegraltodoubleintegral}. It also has a similar purpose on the context of differentiable  universal kernels.

	Below we state and prove a very important result that simplifies the integral of vector valued (and operator valued) functions with respect to vector valued measures using the Radon-Nikod\'ym decomposition.
	
	\begin{lem}\label{manipulation} Let  $\eta: (\Omega, \mathscr{B})\to \mathcal{H}$ be a measure with bounded variation and functions $T: \Omega \to \mathcal{L}(\mathcal{H})$(weak-Bochner integrable with respect to $\eta$) and $F: \Omega \to \mathcal{H}$(Bochner integrable with respect to $|\eta|$). Then if $\eta= Hd\lambda$ is a Radon-Nikod\'ym decomposition for $\eta$, we have that
		$$
		\int_{\Omega}\langle F(w), d\eta(w)\rangle = \int_{\Omega}\langle F(w), H(w) \rangle_{\mathcal{H}}d\lambda(w)
		$$	
		$$
		\langle v, \int_{\Omega}T(w) d\eta(w)\rangle_{\mathcal{H}} = \int_{\Omega}	\langle T(w)v, H(w) \rangle_{\mathcal{H}}d\lambda(w), \quad v \in \mathcal{H}
		$$
		Moreover, if $F$ is bounded, then
		$$
		\int_{\Omega}\langle F(w), d\eta(w)\rangle = \sum_{\mu \in \mathcal{I}}\int_{\Omega}F_{\mu}(w)\overline{H_{\mu}(w)}d\lambda(w)= \sum_{\mu \in \mathcal{I}}\int_{\Omega}F_{\mu}(w)d\eta_{\mu}(w)
		$$
		where $(e_{\mu})_{\mu \in \mathcal{I}}$ is a complete orthonormal basis for $\mathcal{H}$, $F_{\mu}(w):= \langle F(w), e_{\mu}\rangle_{\mathcal{H}}$, $H_{\mu}(w):=\langle H(w), e_{\mu}\rangle_{\mathcal{H}}$ are complex valued functions defined on $\Omega$ and the finite complex valued  measure with bounded variation $\eta_{\mu}:= \langle e_{\mu}, \eta\rangle_{\mathcal{H}} $.
	\end{lem}
	\begin{proof}
		By the definition of a weak-Bochner integrable function, the second equality is a direct consequence of the first one. As for the first equality, the left side is well defined by the definition of the integral, while the right side is well defined because
		$$
		\int_{\Omega} |\langle F(w), H(w) \rangle_{\mathcal{H}}|d\lambda(w)\leq \int_{\Omega} \| F(w)\|_{\mathcal{H}} \|H(w) \|_{\mathcal{H}}d\lambda(w)= \int_{\Omega} \| F(w)\|_{\mathcal{H}} d|\eta|(w)< \infty.
		$$
		If $F= v\chi_{A}$, then 
		$$
		\int_{\Omega}\langle F(w), d\eta(w)\rangle =\langle v, \eta(A)\rangle_{\mathcal{H}} = \int_{\Omega}\langle F(w), H(w) \rangle_{\mathcal{H}}d\lambda(w).
		$$
		By linearity, this equality holds for every simple function. For a general Bochner measurable function $F$ and Bartle integrable with respect to $\eta$, consider $F_{n}$, $n \in \mathbb{N}$, a sequence of simple functions that converges to $F$, then
		$$ |\int_{\Omega}\langle F(w) - F_{n}(w), H(w) \rangle_{\mathcal{H}}d\lambda(w)|\leq \int_{\Omega} \| F(w) -F_{n}(w)\|_{\mathcal{H}} d|\eta|(w) \to 0 $$
		which proves the first equality. \\
		For the second claim, note that
		$$
		\int_{\Omega} \langle F(w), H(w) \rangle_{\mathcal{H}}d\lambda(w) = \sum_{\mu \in \mathcal{I}} \int_{\Omega}F_{\mu}(w)\overline{H_{\mu}}(w)d\lambda(w)
		$$
		by the Lebesgue Dominated convergence, since
		\begin{align*}
		& \sum_{\mu \in \mathcal{I}} \int_{\Omega}|F_{\mu}(w)\overline{H_{\mu}}(w)|d\lambda(w)= \int_{\Omega}\sum_{\mu \in \mathcal{I}}|F_{\mu}(w)H_{\mu}(w)|d\lambda(w)\\
		&\leq  \int_{\Omega}\|F(w)\|_{\mathcal{H}}\|H(w)\|_{\mathcal{H}}d\lambda(w)
		\leq \sup_{\omega \in \Omega}\|F(\omega)\|_{\mathcal{H}}\int_{\Omega}\|H(w)\|_{\mathcal{H}}d\lambda(w)\\
		&=\sup_{\omega \in \Omega}\|F(\omega )\|_{\mathcal{H}}   |\eta|(\Omega) < \infty
		\end{align*}
		the remaining equality is a consequence that $$\eta_{\mu}(A) =\langle e_{\mu}, \eta(A)\rangle_{\mathcal{H}} = \int_{A}\langle e_{\mu}, H(w)\rangle_{\mathcal{H}}d\lambda(w)=\int_{A} H_{\mu}(w)d\lambda(w) $$
	\end{proof}
	
	Although the next result is technical, it is a good source of examples for nonnegative operator valued measures of bounded variation that admits a Radon-Nikod\'ym decomposition, which is a critical condition for our results.

	\begin{lem}\label{radonnikodintrace}Let $X$ be a Hausdorff space and  
		$\Lambda : \mathscr{B}(X) \to \mathcal{L}(\mathcal{H})$ be a set function for which  $\Lambda(\Omega) \in \mathcal{L}(\mathcal{H})$ is a trace class operator and that for every $v \in \mathcal{H}$ 
		$$\Lambda_{v}(A):= \langle \Lambda(A)v,v \rangle_{\mathcal{H}}, \quad A \in \mathscr{B}(X),$$
		is a finite complex valued nonnegative Radon measure. Then,  $\Lambda$ is an operator valued measure of bounded variation  that admits a  Radon-Nikod\'ym   $d\Lambda= Gd\Lambda_{T}$, where
		$$
		\Lambda_{T}(A):= \sum_{\mu \in \mathcal{I}}\langle \Lambda(A) e_{\mu}, e_{\mu} \rangle_{\mathcal{H}} \in \mathbb{C}, \quad A \in  \mathscr{B}(X)
		$$
		$G(w)$ is positive semidefinite and $Tr(G(w))\leq 1$ for all $w \in X$.	
	\end{lem}
	
	\begin{proof} Let $(e_{\mu})_{\mu \in \mathcal{I}}$ be a complete orthonormal basis on $\mathcal{H}$. Since the operator $\Lambda(X)$ is positive semidefinite, trace class and the measure $\Lambda_{v}$ is nonnegative for every $v \in \mathcal{H}$, then for every Borel measurable set $A \subset X$
		$$	
		\sum_{\mu \in \mathcal{I}}\langle \Lambda(A) e_{\mu}, e_{\mu} \rangle_{\mathcal{H}} \leq  \sum_{\mu \in \mathcal{I}}\langle \Lambda(A) e_{\mu}, e_{\mu} \rangle_{\mathcal{H}}  + \sum_{\mu \in \mathcal{I}}\langle \Lambda(X-A) e_{\mu}, e_{\mu}\rangle_{\mathcal{H}}   = \sum_{\mu \in \mathcal{I}}\langle \Lambda(X) e_{\mu}, e_{\mu} \rangle_{\mathcal{H}} < \infty   	 
		$$	
		and then  $\Lambda(A)$ is a trace class positive semidefinite operator. Define the  complex valued function $\Lambda_{T}:  \mathscr{B}(X) \to \mathbb{C} $ as
		$$
		\Lambda_{T}(A):= \sum_{\mu \in \mathcal{I}}\langle \Lambda(A) e_{\mu}, e_{\mu} \rangle_{\mathcal{H}} \in \mathbb{C}.
		$$
		Then $\Lambda_{T}$ is a finite nonnegative complex valued Radon measure on $X$. For every $A \in \mathscr{B}$ the kernel (on the variables $\mu, \nu$) $\Lambda_{\mu, \nu}(A):=\langle \Lambda(A) e_{\mu}, e_{\nu} \rangle_{\mathcal{H}}$ is positive definite, and moreover $\Lambda_{\mu, \nu}$ is a finite complex valued Radon measure on $X$ that is absolutely continuous with respect to $\Lambda_{T}$ ($\Lambda_{\mu, \nu} << \Lambda_{T}$). Indeed, 
		$$
		2|\Lambda_{\mu, \nu}(A)|= 2|\langle \Lambda(A) e_{\mu}, e_{\nu}  \rangle_{\mathcal{H}} | \leq \langle \Lambda(A) e_{\nu}, e_{\nu}  \rangle_{\mathcal{H}}+ \langle \Lambda(A) e_{\mu}, e_{\mu}  \rangle_{\mathcal{H}} \leq 2 \Lambda_{T}(A).
		$$
		By the famous Radon-Nikod\'ym Theorem \cite{folland}, there exists a function $g_{\mu, \nu} : X \to \mathbb{C}$ for which $\Lambda_{\mu, \nu} =g_{\mu, \nu}\Lambda_{T}$. Note that $g_{\mu,\mu}$ is a nonnegative function (almost everywhere on $\Lambda_{T}$) and   
		$$
		\Lambda_{T}(A) = \sum_{\mu \in \mathcal{I}}\langle \Lambda(A) e_{\mu}, e_{\mu} \rangle_{\mathcal{H}}=  \sum_{\mu \in \mathcal{I}}\int_{A}g_{\mu, \mu}(w)d\Lambda_{T}(w)= \int_{A} \left [\sum_{\mu \in \mathcal{I}}g_{\mu, \mu}(w) \right]d\Lambda_{T}(w)
		$$	
		the last equality is an application of the Monotone Convergence Theorem,  \cite{folland}. As a direct consequence we obtain that  $\sum_{\mu \in \mathcal{I}}g_{\mu, \mu}(w) =1$ (almost everywhere on $\Lambda_{T}$). \\
		The kernel $g_{w}: \mathcal{I} \times  \mathcal{I} \to \mathbb{C}$, $w \in X$,  $g_{w}(\mu, \nu)= g_{\mu, \nu}(w)$ is positive definite (almost everywhere on $\Lambda_{T}$). This happens because if $F \subset \mathcal{I}$ is a finite set and  complex numbers $c_{\mu}$, $\mu \in F$, we have that for every Borel measurable set $A $
		$$
		0\leq \langle \Lambda(A)  \sum_{\mu \in  F} c_{\mu}e_{\mu}, \sum_{\mu \in  F} c_{\mu}e_{\mu} \rangle_{\mathcal{H}}= \int_{A} \sum_{\mu, \nu \in  F} c_{\mu}\overline{c_{\nu}}g_{\mu, \nu}(w)d\Lambda_{T}
		$$
		and then we must have that $\sum_{\mu, \nu \in  F} c_{\mu}\overline{c_{\nu}}g_{\mu, \nu}(w)\geq 0$ (almost everywhere on $\Lambda_{T}$, for every fixed finite set $F$ and scalars $c_{\mu}$).  This implies that the kernel $g_{w}$ is positive definite  almost everywhere on $X$, because since $\mathcal{I}$ is countable, we only need a countable amount of finite sets and complex numbers indexed on this finite set (and independent of the kernel) to test if $g_{w}$ is a positive definite kernel.\\
		From those information we obtain that the operator described on the statement of the Theorem is continuous (almost everywhere on $\Lambda_{T}$). Indeed, if $v=\sum_{\mu \in \mathcal{I}}v_{\mu}e_{\mu} \in \mathcal{H}$, then
		$$
		\sum_{\nu \in \mathcal{I}} |g_{\mu, \nu}(w)v_{\nu}| \leq \sum_{\nu \in \mathcal{I}} (g_{\mu, \mu}(w))^{1/2}(g_{\nu, \nu}(w))^{1/2}|v_{\nu}|= (g_{\mu, \mu}(w))^{1/2}\|v\|_{\mathcal{H}}  
		$$
		because $\|\sum_{\nu \in \mathcal{I}} (g_{\nu, \nu}(w))^{1/2} e_{\mu}\|_{\mathcal{H}}= \sum_{\nu \in \mathcal{I}} (g_{\nu, \nu}(w))= 1 $, so
		$$
		\|G(w)v\|_{\mathcal{H}}= \sqrt{\sum_{\mu \in \mathcal{I}} (\sum_{\nu \in \mathcal{I}} g_{\mu, \nu}(w)v_{\nu})^{2}} \leq \sqrt{\sum_{\mu \in \mathcal{I}} g_{\mu, \mu}(w)\|v\|^{2}_{\mathcal{H}}} =  \|v\|_{\mathcal{H}}  
		$$
		which implies that the operator $G(w)$ is continuous (almost everywhere on $\Lambda_{T}$).    Since $\mathcal{I}$ is countable, there exists a sequence of finite sets $(I_{k})_{k \in \mathbb{N}}$ for which $I_{k} \subset I_{k+1} \subset \mathcal{I}$ and $\cup_{k \in \mathbb{N}}I_{k}= \mathcal{I}$.  Given $v= \sum_{\mu \in \mathcal{I}}v_{\mu}e_{\mu} \in \mathcal{H}$, consider $v_{k} :=\sum_{\mu \in I_{k}}v_{\mu}e_{\mu} \in \mathcal{H}$, and note that $v_{k} \to v $ in $\mathcal{H}$. The Lebesgue convergence Theorem implies that
		$$
		\langle \Lambda(A)  v_{k}, v_{k} \rangle_{\mathcal{H}} = \int_{A} \sum_{\mu, \nu  \in I_{k}}v_{\mu}\overline{v_{\nu}}g_{\mu,\nu}(w)d\Lambda_{T}(w) \to \int_{A}\langle G(w) v ,v \rangle_{\mathcal{H}} d\Lambda_{T}(w)
		$$
		while the fact that $v_{k} \to v$ in $\mathcal{H}$ implies that
		$$
		\langle \Lambda(A)  v_{k}, v_{k} \rangle_{\mathcal{H}} \to  \langle \Lambda(A)  v, v \rangle_{\mathcal{H}}. 
		$$
		The general case follow from this case since $\langle \Lambda(A)  u, v \rangle_{\mathcal{H}}$ is finite linear combination of
		$$\langle \Lambda(A)  (u+v), v+v \rangle_{\mathcal{H}}, \quad  \langle \Lambda(A)  (u-v), u-v \rangle_{\mathcal{H}},\quad  \langle \Lambda(A)  (u+iv), u+iv \rangle_{\mathcal{H}}.$$
		Gathering all this information, we can define an  operator valued weak measurable function $G: X \to \mathcal{L}(\mathcal{H})$  for which $G(w)$  is a positive semidefinite  and $Tr(G(w)) \leq 1$ for all $w \in X$, by defining $G(w)$ as the zero operator on the problematic sets listed previously, which is  countable amount of  sets of $\Lambda_{T}$ zero measure.\\
		It only remains to prove the Bochner integrability of the function $G$, which is direct  consequence of the following inequality
		$$
		\int_{X}\|G(w)\|_{\mathcal{L}(\mathcal{H})}d\Lambda_{T}(w) \leq \int_{X} Tr(G(w))d\Lambda_{T}(w) \leq \int_{X} d\Lambda_{T}(w)=\Lambda_{T}(X) < \infty. 
		$$\end{proof}


\begin{thebibliography}{}
		%	\bibitem{reescaling} Arcozzi, N.; Rochberg, R.; Sawyer, E.;  Wick, B. D.: Distance functions for reproducing kernel Hilbert spaces. Function Spaces in Modern Analysis, Contemporary Mathematics, {\bf547 } 25-53 (2011) 
		\bibitem{vecdiff}Abraham, R.:  Marsden, J. E.; Ratiu, R.:  Manifolds, tensor analysis, and applications: 2nd edition. Springer-Verlag, Berlin, Heidelberg (1988)
		\bibitem{Aronszajn1950} Aronszajn, N.: Theory of Reproducing Kernels. Transactions of the American Mathematical Society {\bf 68}(3), 337-404 (1950)
		\bibitem{askey1}Askey, R.: Refinements  of  Abel  summability  for  Jacobi  series, in ``Harmonic  Analysis  on Homogeneous  Spaces.  Proceedings  of  Symposia  in  Pure  Mathematics''  (C.  C.  Moore,Ed.), Am. Math.  Soc.,  {\bf 26}  335-338  (1972)
		\bibitem{askey2}Askey, R.: Radial characteristic functions, Tech. Report No. 1262, Math. Research Center,University of Wisconsin-Madison (1973)
		\bibitem{bartle}Bartle, R.: A general bilinear vector integral. Studia Mathematica, {\bf 15}(3) 337-352 (1956)
		%	\bibitem{berg} Berg, C.; Christensen, J. P. R.; Ressel, P.: Harmonic analysis on semigroups.Theory of positive definiteand related functions. Graduate Texts in Mathematics, 100. Springer-Verlag, New York (1984)
		\bibitem{Caponnetto2008}Caponnetto A.;   Micchelli, C. A.; Pontil, M.; Ying, Y.: Universal multi-task kernels. Journal of Machine Learning Research, {\bf 9} 1615-1646 (2008)
		\bibitem{carmelivitotoigoumanita2010} Carmeli, C.; De Vito, E.; Toigo, A.; Umanita,  V.: Vector valued reproducing kernel Hilbert spaces and universality. Analysis and Applications, { \bf 08 }(01), 19-61 (2010)
		\bibitem{cheney1995}Cheney, E. W.: Approximation using positive definite functions,  Series in Approximation  and Decompositions, {\bf 6}  145-168 (1995)
		\bibitem{CS} Cucker, F.; Smale, S.: On the mathematical foundations of learning. Bull. Amer Math. Soc., {\bf 39}(1) 1-49 (2002)
		\bibitem{CZ} Cucker, F.; Zhou, Ding-Xuan.: Learning theory: an approximation theory viewpoint. Cambridge University Press, Vol. 24  (2007)
		\bibitem{Emilio} Daley, D.; Porcu, E.:  Dimension walks and Schoenberg spectral measures. Proceedings of the American Mathematical Society, {\bf 142}(5), 1813-1824 (2014)
		\bibitem{vectormeasures}Diestel, J.;  Uhl, J. J. Jr.: Vector Measures. American Mathematical Society, Providence, (1977)
		%	\bibitem{buh} Buhmann, M.: Radial Basis Functions: Theory and Implementations.\ Cambridge University Press (2003)
		%	\bibitem{chang} Chang, Kuei-Fang: Strictly positive definite functions.\ J. Approx. Theory {\bf 87}(2), 148-158 (1996)
		\bibitem{folland} Folland, G. B.: Real analysis. Modern techniques and their applications. Second edition. Pure and Applied Mathematics. A Wiley-Interscience Publication. John Wiley $\&$ Sons, Inc., New York (1999)
		%	\bibitem{gelfand} Gelfand, I. M.; Vilenkin,N. Ya.: Generalized Functions, Volume 4: Applications of Harmonic Analysis.\ AMS Chelsea Publishing (1964)
		%	\bibitem{gesz} Gesztesy, Fritz; Pang, Michael: On (conditional) positive semidefiniteness in a matrix-valued context.\ Studia Math. {\bf 236}(2), 143-192 (2017)
		%	\bibitem{jpv} Guella, J. C.; Menegatto, V. A.; Porcu, E.: Strictly positive definite multivariate covariance functions on spheres.\ J. Multivariate Anal., {\bf 166}, 150-159 (2018)
		%	\bibitem{sunguo} Guo K.; Hu S.; Sun X.: Conditionally positive definite functions and Laplace-Stieltjes integrals.\ J. Approx. Theory {\bf 74}(3), 249-265 (1993)
		\bibitem{compcsimgraph}  Morse, B. S.; Yoo, T. S.; Rheingans, P.: Chen, D. T.; Subramanian, K. R.: Interpolating implicit surfaces from scattered
		surface data using compactly supported radial basis functions. Shape Modelling International, 89-98 (2001)
		\bibitem{Sriperumbudur2}Fukumizu, K.; Gretton, A.; Sch\"{o}lkopf, B.; Sriperumbudur, B. K.: Characteristic kernels on groups and semigroups,  Advances in Neural Information Processing Systems, {\bf 21}  473-480 (2009)
		\bibitem{gneiting}Gneiting, T.: Radial Positive Definite Functions Generated
		by Euclid`s Hat. Journal of Multivariate Analysis {\bf 69}(1) 88-119 (1999)
		\bibitem{gneiting2}Gneiting, T.: On the Derivatives of Radial Positive Definite Functions. Journal of Mathematical Analysis and Applications, {\bf 236} 86-93 (1999)
		\bibitem{guella3} Guella, J. C.; Menegatto, V. A.; Porcu, E.: Strictly positive definite multivariate covariance functions on spheres. Journal of Multivariate Analysis {\bf 166}  150-159 (2018)
		\bibitem{Guella2019} Guella, J. C.; Menegatto, V. A.: Conditionally Positive Definite Matrix Valued Kernels on Euclidean Spaces. Constructive Approximation (2019)
		
		%	\bibitem{bousquethein}Hein, M.; Bousquet, O.: Kernels, Associated Structures and Generalizations. Tbingen, Germany: Max Planck Institute for Biological Cybernetics { \bf 127} (2004)
		%	\bibitem{hornT} Horn, R. A.; Johnson, C. R.: Topics in matrix analysis.\ Corrected reprint of the 1991 original.\ Cambridge University Press, Cambridge (1994)
		%	\bibitem{JorgensenTian}	Jorgensen, P.; Tian, F.: Infinite weighted graphs with bounded resistance metric.\ Mathematica Scandinavica,  {\bf 123}(1) 5-38 (2018)
		\bibitem{kilmer1996}Kilmer, S. J.; Light, W. A.; Sun, Xingping; Yu, X. M.: Approximation by translates of a positive definite function. Journal of mathematical analysis and applications, {\bf 201}(2) 631-641  (1996)
		\bibitem{Michelli2014} Micheli, M.; Glaun\`{e}s, J. A.: Matrix-valued Kernels for Shape Deformation Analysis.  Geometry, Imaging, and Computing, {\bf 1}(1) 57-139 (2014)
		%\bibitem{Minh2010}	 Minh, H. Q.: Some properties of Gaussian reproducing kernel Hilbert spaces and their implications for function approximation and learning theory.\ Constructive Approximation, {\bf 32 }(2) 307-338 (2010)
		\bibitem{Minh2016}	Minh, H. Q.; Bazzani, L.; Murino, V.:  A unifying framework in vector-valued reproducing kernel Hilbert spaces for manifold regularization and co-regularized multi-view learning. Journal of Machine Learning Research, {\bf 17}(1) 769-840 (2016)
		
		%	\bibitem{Nachbin} Nachbin, L.: Sur les algebres denses de fonctions differentiables sur une variete. Comptes Rendus del' Academie des Sciences de Paris, { \bf 228} 1549-1551 (1949)
		\bibitem{neeb1998}Neeb, K. H.: Operator-valued positive definite kernels on tubes. Monatshefte f{\"u}r Mathematik, {\bf 126}(2) 125-160 (1998)
		%\bibitem{michdistance} Micchelli, C. A.: Interpolation of scattered data: distance matrices and conditionally positive definite functions.\ Constr. Approx {\bf 2}(11), 11-22 (1986)
		%\bibitem{nar}Narcowich, Francis J.: Generalized Hermite interpolation and positive definite kernels on a Riemannian manifold. J. Math. Anal. Appl. {\bf 190}(1), 165-193 (1995).
		%\bibitem{narward}Narcowich, Francis J.; Ward, Joseph D.: Generalized Hermite interpolation via matrix-valued conditionally positive definite functions. Math. Comp. {\bf 63}(208), 661-687 (1994).
		%	\bibitem{Pinkusreal}Pinkus, A.:Strictly Positive Definite Functions on a Real Inner Product Space. Advances in Computational Mathematics, {\bf 20}(4) 263-271 (2004)
		%	\bibitem{reams}Reams, R.: Hadamard inverses, square roots and products of almost semidefinite matrices. Linear Algebra and its Applications,  {\bf 288} 35-43 (1999)	
		%	\bibitem{Saitoh}Saitoh, Saburou:Integral transforms, reproducing kernels and their applications.\ CRC Press {\bf 369}, (1997)
		%	\bibitem{schilling} Schilling, Ren\'{e} L.; Song, Renming; Vondracek, Zoran: Bernstein functions.\ Theory and applications.\ Second edition.\ De Gruyter Studies in Mathematics, 37.\ Walter de Gruyter $\&$ Co., Berlin (2012)
		\bibitem{schoenb} Schoenberg, I. J.: Metric spaces and completely monotone functions. Ann. of Math, {\bf 39}(4) 811-841 (1938)
		\bibitem{Simon-Gabriel2018}Simon-Gabriel, C. J.; Sch{{\"o}}lkopf, B.: Kernel Distribution Embeddings: Universal Kernels, Characteristic Kernels and Kernel Metrics on Distributions, Journal of Machine Learning Research, { \bf 19}(44) 1-29 (2018)
		\bibitem{Schwartzkernell} Schwartz, L.: Sous-espaces hilbertiens d'espaces vectoriels topologiques et noyaux associ\'{e}s (noyaux reproduisants). J. Analyse Math., {\bf 13} 115-256 (1964)
		\bibitem{Sriperumbudur} Sriperumbudur, B. K.; Gretton, A.; Fukumizu, K.; Scholkopf, B.; Lanckriet, G. R.: Hilbert space embeddings and metrics on probability measures. Journal of Machine Learning Research, { \bf 11 } 1517-1561 (2010)
		\bibitem{Sriperumbudur3} Sriperumbudur, B. K.; Fukumizu, K.; Lanckriet, G. R.:  Universality, characteristic kernels and RKHS embedding of measures. Journal of Machine Learning Research {\bf 12}  2389-2410 (2011)
		\bibitem{sun} Sun, Xingping: Conditionally positive definite functions and their application to multivariate interpolations. Journal of Approximation Theory, {\bf 74}(2) 159-180 (1993)
		%\bibitem{xie} Xie, Tailiang: Positive definite matrix-valued functions and matrix variogram modeling.\ Thesis (Ph.D.)\ The University of Arizona (1994)
		\bibitem{zucastell} Zu Castell, Wolfgang; Filbir, Frank; Szwarc, Ryszard: Strictly positive definite functions in $\mathbb{R}^d$. Journal of Approximation Theory, {\bf 137}(2) 277-280 (2005)
		%	\bibitem{wang} Wang, Renxiang; Du, Juan; Ma, Chunsheng: Covariance matrix functions of isotropic vector random fields.\ Comm. Statist. Theory Methods, {\bf 43}(10-12) 2081-2093 (2014)
		%\bibitem{wells} Wells, J. H.; Williams, L. R.: Embeddings and extensions in analysis.\ Ergebnisse der Mathematik und ihrer Grenzgebiete, Band 84. Springer-Verlag, New York-Heidelberg (1975)
		\bibitem{wend} Wendland, H.: Scattered data approximation. Cambridge University Press, Cambridge (2005)
		\bibitem{williamson} Williamson, R. E.:Multiply monotone functions and their Laplace transforms, Duke Math. {\bf 23}  189-207 (1956)
		%	\bibitem{commutativekernel}	Wittwar, D.; Santin, G.; Haasdonk, B.: Interpolation with uncoupled separable matrix-valued kernel. Dolomites Research Notes on Approximation, {\bf 11}(3) 23-39 (2018)
		\bibitem{zhang2012} Haizhang Zhang; Yuesheng Xu; Qinghui Zhang: Refinement of Operator-valued Reproducing Kernels. Journal of Machine Learning Research, {\bf 13 } 91-136 (2012)
		\bibitem{derivativekernel}Zhou, Ding-Xuan: Derivative reproducing properties for kernel methods in learning theory. Journal of Computational and Applied Mathematics, {\bf 220} 456-463 (2008) 
		
		
		
		% etc
	\end{thebibliography}
\end{document}